\documentclass[a4paper]{article}
\usepackage[T1]{fontenc}
\usepackage{lmodern}
\usepackage{amsmath,amssymb}
\usepackage{amsthm}
\usepackage{mathrsfs}
\usepackage[all]{xy}
\usepackage{rotate}
\usepackage{enumitem}
\usepackage{mathtools}
\usepackage[dvipdfmx]{graphics}
\usepackage{color}

\usepackage{here}

\usepackage{atbegshi}

\newcommand{\Ad}{\textup{Ad}}
\newcommand{\ad}{\textup{ad}}

\renewcommand{\Im}{\textup{Im}}

\newcommand{\id}{\textup{id}}
\newcommand{\tr}{\textup{tr}}

\newcommand{\GL}{\textup{GL}}
\newcommand{\gl}{\mathfrak{gl}}

\renewcommand{\sl}{\mathfrak{sl}}

\renewcommand{\sp}{\mathfrak{sp}}

\newcommand{\upO}{\textup{O}}
\newcommand{\SO}{\textup{SO}}
\newcommand{\so}{\mathfrak{so}}

\newcommand{\End}{\textup{End}}
\newcommand{\RR}{\mathbb{R}}

\newcommand{\CC}{\mathbb{C}}
\newcommand{\ZZ}{\mathbb{Z}}
\newcommand{\NN}{\mathbb{N}}

\newcommand{\OO}{\mathbb{O}}
\newcommand{\DD}{\mathbb{D}}

\newcommand{\Ind}{\textup{Ind}}
\newcommand{\sgn}{\textup{sgn}}
\newcommand{\calF}{\mathcal{F}}
\newcommand{\calI}{\mathcal{I}}
\newcommand{\calO}{\mathcal{O}}

\newcommand{\calW}{\mathcal{W}}
\newcommand{\calB}{\mathcal{B}}
\newcommand{\calC}{\mathcal{C}}
\newcommand{\calH}{\mathcal{H}}
\newcommand{\calP}{\mathcal{P}}
\newcommand{\calU}{\mathcal{U}}
\newcommand{\calL}{\mathcal{L}}

\newcommand{\calD}{\mathcal{D}}

\newcommand{\calV}{\mathcal{V}}
\newcommand{\calZ}{\mathcal{Z}}

\newcommand{\calR}{\mathcal{R}}
\newcommand{\calM}{\mathcal{M}}
\newcommand{\calN}{\mathcal{N}}

\newcommand{\rad}{\textup{rad}}
\newcommand{\soc}{\textup{soc}}

\newcommand{\spec}{\textup{spec}}
\newcommand{\Ann}{\textup{Ann}}

\newcommand{\gr}{\textup{gr}}
\newcommand{\Hom}{\textup{Hom}}

\newcommand{\ev}{\textup{ev}}

\newcommand{\Tor}{\mathrm{Tor}}

\theoremstyle{plain}
\newtheorem{theorem}{Theorem}[section]
\newtheorem*{theorem*}{Theorem}

\newtheorem{proposition}[theorem]{Proposition}
\newtheorem{lemma}[theorem]{Lemma}
\newtheorem{corollary}[theorem]{Corollary}

\newtheorem{fact}[theorem]{Fact}

\theoremstyle{definition}
\newtheorem{definition}[theorem]{Definition}
\newtheorem{example}[theorem]{Example}

\newtheorem{remark}[theorem]{Remark}

\theoremstyle{remark}
\renewenvironment{pf}{\begin{proof}[\sc Proof]}{\end{proof}}

\numberwithin{equation}{subsection}

\newcommand{\rring}[1]{\calO(#1)}
\newcommand{\rsheaf}[1]{\calO_{#1}}
\newcommand{\define}[1]{\textit{#1}}

\newcommand{\AV}{\mathcal{AV}}

\newcommand{\PIdeg}{\textup{PI.deg}}
\newcommand{\U}{\textup{U}}
\newcommand{\calA}{\mathcal{A}}
\newcommand{\spn}[1]{\mathrm{span}_{\CC}\{#1\}}

\newcommand{\set}[1]{\left\{#1 \right\}}

{\begin{equation} 
\addtolength{\abovedisplayskip}{-1ex}
\addtolength{\abovedisplayshortskip}{-1ex}
\addtolength{\belowdisplayskip}{-1ex}
\addtolength{\belowdisplayshortskip}{-1ex}
\begin{minipage}[t]{0.87\linewidth}}
{\end{minipage} \end{equation} \ignorespacesafterend}

\newcommand{\frakS}{\mathfrak{S}}
\newcommand{\univ}[1]{\calU(\mathfrak{#1})}
\newcommand{\univcent}[1]{\calZ(\mathfrak{#1})}
\newcommand{\esssup}{\textup{ess sup}}

\newcommand{\pro}[2]{\mathrm{pro}^{#1}_{#2}}
\newcommand{\Ext}{\mathrm{Ext}}

\newcommand{\Len}{\mathrm{Len}}

\newcommand{\Ker}{\mathrm{Ker}}

\newcommand{\Irr}{\mathbf{Irr}}

\newcommand{\roots}{\Delta}
\newcommand{\proots}{{\Delta^+}}

\newcommand{\lie}[1]{{\mathfrak{#1}}}

\newcommand{\lieR}[1]{\mathfrak{#1_{\RR}}}

\newcommand{\lieCent}{\mathfrak{c}}
\newcommand{\supmul}{\mathcal{M}}

\newcommand{\Spec}{\mathrm{Spec}}
\newcommand{\Mod}{\mathrm{Mod}}

\newcommand{\sect}{\Gamma}

\newcommand{\zuck}[2]{\Gamma^{#1}_{#2}}
\newcommand{\Dzuck}[3]{\DD^{#3}\zuck{#1}{#2}}

\newcommand{\ind}[2]{\mathrm{ind}^{\lie{#1}}_{\lie{#2}}}
\newcommand{\dual}[1]{#1^{\vee}}

\newcommand{\ntDsheaf}{\mathscr{D}}

\newcommand{\PR}{\mathrm{Pr}}

\newcommand{\opalg}{\mathrm{op}}

\newcommand{\Loewy}[1]{\mathrm{LLen}({#1})}

\newcommand{\SupDim}{\mathrm{Dim}}
\newcommand{\CWCat}[1]{\mathcal{CW}({#1})}

\newcommand{\topdual}[1]{#1^*}

\newcommand{\supp}{\mathrm{supp}}
\newcommand{\ccInd}{C^\infty_c\text{-}\Ind}
\newcommand{\uInd}{L^2\text{-}\Ind}
\newcommand{\toptensor}[1][]{\mathbin{\widehat{\otimes}_{#1}}}
\newcommand{\Oind}[2]{\calO^{\lie{#1}}_{\lie{#2}}}

\newenvironment{enumparen}{\begin{enumerate}[label=\upshape(\arabic*),ref=(\arabic*)]}{\end{enumerate}}
\newenvironment{enumalpha}{\begin{enumerate}[label=\upshape(\alph*),ref=(\alph*)]}{\end{enumerate}}
\newenvironment{enumalphadash}{\begin{enumerate}[label=\upshape(\alph*'),ref=(\alph*')]}{\end{enumerate}}

\begin{document}

\title{Uniformly bounded multiplicities, polynomial identities and coisotropic actions}

\author{Masatoshi Kitagawa}

\date{}

\maketitle

\begin{abstract}
	Let $G_\RR$ be a real reductive Lie group and $G'_\RR$ a reductive subgroup of $G_\RR$ such that $\lie{g'}$ is algebraic in $\lie{g}$.
	In this paper, we consider restrictions of irreducible representations of $G_\RR$ to $G'_\RR$ and induced representations of irreducible representations of $G'_\RR$ to $G_\RR$.
	Our main concern is when such a representation has uniformly bounded multiplicities, i.e.\ the multiplicities in the representation are (essentially) bounded.
	We give characterizations of the uniform boundedness by polynomial identities and coisotropic actions.

	For the restriction of (cohomologically) parabolically induced representations, we find a sufficient condition for the uniform boundedness by spherical actions and some fiber condition.
	This result gives an affirmative answer to a conjecture by T. Kobayashi.

	Our results can be applied to $(\lie{g}, K)$-modules, Casselman--Wallach representations, unitary representations and objects in the BGG category $\calO$.
	We also treat with an upper bound of cohomological multiplicities.
\end{abstract}

\section{Introduction}

Our main concern in this paper is the boundedness of multiplicities in a representation of a real reductive Lie group.
We characterize the boundedness of multiplicities using two invariants of representations.
One is a purely algebraic invariant called the PI degree, and the other is a geometric invariant called a nilpotent coadjoint orbit.
We deal with restrictions and inductions of unitary representations,
Casselman--Wallach representations, $(\lie{g}, K)$-modules and objects in the BGG category $\calO$.

Our study is originated from several studies of multiplicity-free representations and uniform boundedness theorems.
We recall a characterization of spherical varieties because our theorems are described by very similar forms to the characterization.
For an affine homogeneous variety $X$ of a connected reductive algebraic group $G$ over $\CC$, the following conditions are equivalent:
\begin{enumerate}
	\item $X$ is $G$-spherical
	\item (representation theory) the coordinate ring $\rring{X}$ is multiplicity-free as a $G$-module
	\item (algebra) the algebra of $G$-invariant differential operators on $X$ is commutative
	\item (symplectic geometry) the $G$-action on the cotangent bundle $T^*(X)$ is coisotropic, i.e.\ for generic $x \in T^*(X)$, the tangent space $T_x(Gx)$ is coisotropic in $T_x(T^*(X))$ with respect to the canonical symplectic form on $T^*(X)$.
\end{enumerate}
See \cite[Theorem 25.4]{Ti11} and references therein.
Our main theorems (Theorems \ref{intro:thm:Restriction} and \ref{intro:thm:Induction}) can be considered as an analogue and a generalization of this characterization to restrictions and inductions of infinite-dimensional representations of real reductive Lie groups.

Another origin of our study is Kobayashi's uniform boundedness theorem for unitary highest weight modules.
Following to the study \cite{Ko08}, we recall the uniform boundedness of multiplicities.
Let $G_\RR$ be a real reductive Lie group with a maximal compact subgroup $K_\RR$.
It is well-known that a unitary representation $\calH$ of $G_\RR$ has the irreducible decomposition:
\begin{align}
	\calH\simeq \int_{\widehat{G}_\RR}^\oplus \calH_\pi\toptensor \calM_\pi d\mu(\pi), \label{intro:eqn:DirectIntegral}
\end{align}
where $\mu$ is a Borel measure on the unitary dual $\widehat{G}_\RR$ of $G_\RR$, $\calH_\pi$ is a representation space of $\pi \in \widehat{G}_\RR$ and $\calM_\pi$ is a Hilbert space on which $G_\RR$ acts trivially.
We say that $\calH$ has uniformly bounded multiplicities if $\dim_{\CC}(\calM_\pi)$ is essentially bounded on $\pi \in \widehat{G}_\RR$.
Similarly, a smooth representation $V$ of $G_\RR$ (resp.\ $(\lie{g}, K)$-module $V$) is said to have uniformly bounded multiplicities if there exists a constant $C$ such that
\begin{align*}
	\dim_{\CC}(\Hom_{G_\RR}(V, W)) \leq C\quad (\text{resp.\ }\dim_{\CC}(\Hom_{\lie{g}, K}(V, W)) \leq C)
\end{align*}
for any irreducible Casselman--Wallach representation $W$ of $G_\RR$ (resp.\ $(\lie{g}, K)$-module $W$).
Here $\lie{g}$ is the complexification of the Lie algebra $\lie{g}_\RR$ of $G_\RR$ and $K$ is the complexification of $K_\RR$.

As we have mentioned, the commutativity of an algebra of invariant operators (e.g.\ differential operators) is strongly related to the multiplicity-freeness.
To capture the uniform boundedness of multiplicities, the commutativity is not enough.
The notion of PI degree defined below plays a central role in our study.
We refer the reader to \cite[Chapter 13]{McRo01_noncommutative} and \cite{DrFo04}.

\begin{definition}\label{def:PIdeg}
	Let $\calA$ be a $\CC$-algebra.
	For a $\ZZ$-coefficient non-commutative polynomial $f$ with $n$ indeterminates, we say that $f$ is a \define{polynomial identity} of $\calA$ if $f(X_1, X_2, \ldots, X_n) = 0$ for any $X_i \in \calA$.
	We say that $\calA$ has \define{PI degree} $n$ if $n$ is the largest natural number such that any polynomial identity of $\calA$ is 
	a polynomial identity of the matrix algebra $M_n(\CC)$.
	We denote by $\PIdeg(\calA)$ the PI degree of $\calA$.
	(We set $\PIdeg(\calA) = \infty$ if $n$ does not exist.)
\end{definition}

Note that the PI degree distinguishes the matrix algebras, that is,
we have $\PIdeg(M_n(\CC)) = n$ by the Amitsur--Levitzki theorem \cite{AmLe50}.
Let $\calA$ be a $\CC$-algebra of at most countable dimension.
Obviously, if $m > \PIdeg(\calA)$, there is no surjective homomorphism $\calA \rightarrow M_m(\CC)$.
In other words, the dimension of any irreducible $\calA$-module is less than or equal to $\PIdeg(\calA)$.
Conversely, if there is a family $\set{(\pi_\lambda, V_\lambda)}_{\lambda \in \Lambda}$ of finite-dimensional $\calA$-modules
such that $\bigcap_\lambda \Ker(\pi_\lambda) = 0$, then we have
\begin{align*}
	\PIdeg(\calA) \leq \sup_{\lambda \in \Lambda} \dim_{\CC}(V_\lambda).
\end{align*}
These properties are the reason why the PI degree can capture the supremum of multiplicities.
In many cases, we will let $V_\lambda$ be the space $\Hom_{G'_\RR}(V, V')$ of $G'_\RR$-homomorphisms or similar spaces such as $\Hom_{\lie{g'}, K'}(V, V')$ and $V\otimes_{\univ{g'}}V'$.

Let us state the main theorem for the restriction case.
Let $G_\RR$ be a real reductive Lie group and $G'_\RR$ a reductive subgroup of $G_\RR$.
Fix a maximal compact subgroup $K_\RR$ of $G_\RR$ such that $K'_\RR:= K_\RR\cap G'_\RR$ is maximal compact in $G'_\RR$.
Let $\lie{g}$ (resp.\ $\lie{g'}$, $K$ and $K'$) denote the complexification of the Lie algebra $\lie{g}_\RR$ of $G_\RR$ (resp.\ $\lie{g}'_\RR$, $K_\RR$ and $K'_\RR$).
In our results, the assumption that $\lie{g'}$ is algebraic in $\lie{g}$ is essential, that is, there exist a reductive algebraic group $G$ with the Lie algebra $\lie{g}$ and a (Zariski closed) reductive subgroup $G'$ of $G$ with the Lie algebra $\lie{g'}$.
If necessary, replacing $G$ and $G'$, we can assume that the inclusion $\lie{g}_\RR \hookrightarrow \lie{g}$ (resp.\ $\lie{g}'_\RR \hookrightarrow \lie{g'}$) lifts to a Lie group homomorphism $G_\RR \rightarrow G$ (resp.\ $G'_\RR \rightarrow G'$) with Zariski dense image.

For an affine variety $X$ and an ideal $I$ of the coordinate ring $\rring{X}$, we denote by $\calV(I)$ the closed subset of $X$ defined by $I$.

\begin{theorem}\label{intro:thm:Restriction}
	Let $V$ be a Casselman--Wallach representation of $G_\RR$.
	Set $I:=\Ann_{\univ{g}}(V)$.
	Then the following conditions on $V$ are equivalent:
	\begin{enumparen}
		\item\label{intro:item:RestrictionCW} $V|_{G'_\RR}$ has uniformly bounded multiplicities
		\item\label{intro:item:RestrictionGK} $V_K|_{\lie{g'}, K'}$ has uniformly bounded multiplicities
		\item\label{intro:item:RestrictionPIdeg} $\PIdeg((\univ{g}/I)^{G'}) < \infty$
		\item\label{intro:item:RestrictionPoisson} $(S(\lie{g})/\sqrt{\gr(I)})^{G'}$ is Poisson-commutative
		\item\label{intro:item:RestrictionCoisotropic} the action of $G'$ on $\calV(\gr(I)) (\subset \lie{g}^*)$ is coisotropic
		\item\label{intro:item:RestrictionFinGen} $(\univ{g}/I)^{G'}$ is finitely generated as a $\univcent{g'}$-module.
	\end{enumparen}
	If, in addition, $V$ is unitarizable and $\calH$ is the Hilbert completion of $V$, then the above conditions are equivalent to
	\begin{enumparen}
		\setcounter{enumi}{6}
		\item\label{intro:item:RestrictionUnitary} $\calH|_{G'_\RR}$ has uniformly bounded multiplicities.
	\end{enumparen}
\end{theorem}

We shall sketch the proof of the theorem.
The implication \ref{intro:item:RestrictionPIdeg} $\Rightarrow$ \ref{intro:item:RestrictionPoisson} is a purely algebraic result.
Roughly speaking, an algebra of finite PI degree is `almost commutative' and hence its associated graded algebra should be `almost Poisson-commutative'.
We prove the implication in Subsection \ref{subsect:PoissonPolynomialIdentity} using polynomial identities coming from the Cayley--Hamilton theorem and their associated graded versions.

A large part of the equivalence \ref{intro:item:RestrictionPoisson} $\Leftrightarrow$ \ref{intro:item:RestrictionCoisotropic}
and the implication \ref{intro:item:RestrictionCoisotropic} $\Rightarrow$ \ref{intro:item:RestrictionFinGen} is a Poisson geometrical result.
The proof largely depends on the study of conical $G$-Hamiltonian Poisson varieties by I. Losev \cite{Lo09}.
We summarize several results in Poisson geometry in Section \ref{sect:Coisotropic}.

For the condition \ref{intro:item:RestrictionUnitary}, we use the theory of disintegration of unbounded operators and representations.
To study the algebra $\univ{g}^{G'}$, we construct $*$-representations of $\univ{g}^{G'}$ on the spaces $\calM_{\pi}$ in the direct integral decomposition \eqref{intro:eqn:DirectIntegral}.
We deal with unitary representations only in Section \ref{sect:Unitary}.

The first half of this paper is devoted to the proof of \ref{intro:item:RestrictionCW} $\Leftrightarrow$ \ref{intro:item:RestrictionGK} $\Leftrightarrow$ \ref{intro:item:RestrictionPIdeg}.
We will show that there exists a constant $C > 0$ such that for any irreducible Casselman--Wallach representation $V$ of $G_\RR$, 
\begin{align}
	\PIdeg((\univ{g}/I)^{G'}) &\leq \sup_{V'}\dim_{\CC}(\Hom_{G'_\RR}(V, V'))\\
	&\leq \sup_{V'}\dim_{\CC}(\Hom_{\lie{g'}, K'}(V_K, V'_{K'})) \\
	&\leq C \cdot \PIdeg((\univ{g}/I)^{G'}), \label{intro:eqn:BoundSup}
\end{align}
where the two supremums are over all irreducible Casselman--Wallach representations $V'$ of $G'_\RR$.
Note that we assume the irreducibility of $V$ for simplicity, and if $V$ is not irreducible, we need additional terms.
The inequation \eqref{intro:eqn:BoundSup} is proved in Theorem \ref{thm:BoundedRestriction}.
We will prove a similar upper bound for cohomological multiplicities in Lemma \ref{lem:UpperBoundTorGeneral} and Corollary \ref{cor:BoundedHomology}.

If $G_\RR$ is connected and $G'_\RR$ is compact, then we can take $C = 1$ in \eqref{intro:eqn:BoundSup}.
This fact is proved by Penkov--Serganova in the proof of \cite[Theorem 4.3]{PeSe12} in the context of generalized Harish-Chandra modules.
In the case, the inequation follows from the properties of the PI degree and the following two claims:
\begin{enumalpha}
	\item\label{intro:item:Jacobson} The $\univ{g}^{G'}$-module $\Hom_{G'_\RR}(F', V)$ is irreducible for any irreducible finite-dimensional representation $F'$ of $G'_\RR$.
	\item\label{intro:item:Faithful} The annihilator of the $\univ{g}^{G'}$-module $\prod_{F'} \Hom_{G'_\RR}(F', V)$ coincides with $\Ann_{\univ{g}^{G'}}(V)$,
	where the direct product is taken over all irreducible finite-dimensional representations $F'$ of $G'_\RR$.
\end{enumalpha}
The property \ref{intro:item:Jacobson} is a classical result following from the Jacobson density theorem.
The property \ref{intro:item:Faithful} is trivial since $V|_{G'_\RR}$ is completely reducible.

Our strategy to show the inequation \eqref{intro:eqn:BoundSup} is to generalize the properties \ref{intro:item:Jacobson} and \ref{intro:item:Faithful} to the infinite-dimensional case as
\begin{enumalphadash}
	\item\label{intro:item:JacobsonDash} There exists a constant $C > 0$ such that for any irreducible $(\lie{g}, K)$-module $V$ and irreducible $(\lie{g'}, K')$-module $V'$, we have
	\begin{align*}
		\Len_{\univ{g}^{G'}}(H_0(\lie{g'}, K'; V\otimes V')) \leq C,
	\end{align*}
	where $\Len(\cdot)$ means the length of a module.
	\item\label{intro:item:FaithfulDash} Let $V$ be an irreducible Casselman--Wallach representation of $G_\RR$ and set $J:= \bigcap_{V'} \Ann_{\univ{g}^{G'}} (\Hom_{G'_\RR}(V, V'))$,
	where the union is over all irreducible Casselman--Wallach representations $V'$ of $G'_\RR$.
	Then there exists a constant $C > 0$ independent of $V$ such that $J^C \subset \Ann_{\univ{g}^{G'}}(V)$.
\end{enumalphadash}
The property \ref{intro:item:JacobsonDash} has been proved in our previous paper \cite[Theorem 7.18]{Ki20} (Fact \ref{fact:BoundedTorHom}).
The property \ref{intro:item:FaithfulDash} is proved in Section \ref{sect:LowerBound} in a general setting.

The upper bound in \eqref{intro:eqn:BoundSup} follows easily from \ref{intro:item:JacobsonDash} and the natural isomorphism
$H_0(\lie{g'}, K'; V\otimes V')^* \simeq \Hom_{\lie{g'}, K'}(V, (V')^*_{K'})$.
Note that the $\univ{g}^{G'}$-module $\Hom_{\lie{g'}, K'}(V, V')$ may not have finite length unlike $H_0(\lie{g'}, K'; V\otimes V')$.
In fact, if $\Hom_{\lie{g'}, K'}(V, V')$ is not finite dimensional, then it is uncountable infinite dimensional and hence should have infinite length as a $\univ{g}^{G'}$-module.
The lower bound in \eqref{intro:eqn:BoundSup} follows from \ref{intro:item:FaithfulDash}.
Roughly speaking, \ref{intro:item:FaithfulDash} asserts that $V|_{G'_\RR}$ has enough many irreducible quotients to recover $\Ann_{\univ{g}^{G'}}(V)$.

If $G'_\RR$ is compact, the equivalence \ref{intro:item:RestrictionGK} $\Leftrightarrow$ \ref{intro:item:RestrictionCoisotropic} is proved by A. Petukhov in \cite[Theorem 3.1]{Pe18} in the context of generalized Harish-Chandra modules.
The sphericity of the $G'$-action on the associated variety of $V$ plays an important role in his proof.
Since $G'$ may not act on the associated variety in our setting, we can not generalize his proof to our setting directly.
See Fact \ref{fact:Petukhov} and its surroundings.

If $V$ is a parabolically induced representation $\Ind^{G_\RR}_{P_\RR}(W)$, we can reduce the uniform boundedness of multiplicities in $V|_{G'_\RR}$ to that of the fiber $W$.
We shall state this result.
For simplicity, we assume that $G_\RR$ and $G'_\RR$ are connected.
Let $P_\RR$ be a parabolic subgroup of $G_\RR$ and $P$ the analytic subgroup of $G$ with the Lie algebra $\lie{p}\subset \lie{g}$.

If $G/P$ is $G'$-spherical (i.e.\ a Borel subgroup $B'$ has an open orbit $B'xP$ in $G/P$), we set
\begin{align}
	P'&:= \set{g \in G': gB'xP = B'xP}, \\
	L'&:= x^{-1}P'x \cap P. \label{intro:eqn:L}
\end{align}
By the local structure theorem \cite{BLV86}, $L'$ is a reductive subgroup of $x^{-1}G'x \cap P$.
In the case, we can assume that $G'_\RR P_\RR$ is open in $G_\RR$
and hence we can choose the point $x = e$.
Moreover, we can replace the Borel subgroup $B'$ with a good one for $\lie{l'}\cap \lie{g}'_\RR$ to be a real from of $\lie{l'}$ (see Proposition \ref{prop:BLVreal}).
Let $L'_\RR$ be the analytic subgroup of $G'_\RR$ with the Lie algebra $\lie{l}'_\RR$.

Let $M_\RR$ be a Levi subgroup of $P_\RR$ containing $L'_\RR$.
We extend a representation of $M_\RR$ to one of $P_\RR$ letting the nilradical of $P_\RR$ act trivially.
We denote by $\Ind^{G_\RR}_{P_\RR}(V)$ the induced representation of a representation $V$ of $P_\RR$ (without $\rho$-shift).
Let $K_M$ be the complexification of a maximal compact subgroup of $M_\RR$.

\begin{theorem}\label{intro:thm:ReductionToFiberParabolic}
	Let $V$ be a Casselman--Wallach representation of $M_\RR$ and $X$ a non-zero closed subrepresentation of $\Ind^{G_\RR}_{P_\RR}(V)$.
	Then $X|_{G'_\RR}$ has uniformly bounded multiplicities if $G/P$ is $G'$-spherical and $V|_{L'_\RR}$ has uniformly bounded multiplicities.
	The converse is also true if $\univ{g}\otimes_{\univ{p}} (V^*)_{K_M}$ is irreducible as a $(\lie{g}, K_M)$-module or $X = \Ind^{G_\RR}_{P_\RR}(V)$.
\end{theorem}

Note that the assumption in the converse part implies $\Ann_{\univ{g}^{G'}}(X) = \Ann_{\univ{g}^{G'}}(\Ind^{G_\RR}_{P_\RR}(V))$.
Although we can replace the assumption with this condition about the annihilators, the condition is not easy to verify directly.
Theorem \ref{intro:thm:ReductionToFiberParabolic} is proved in Theorem \ref{thm:ReductionToFiberParabolicInd}.

For cohomological parabolic inductions, we can prove an analogue of Theorem \ref{intro:thm:ReductionToFiberParabolic}.
Let $\lie{p}$ be a parabolic subalgebra of $\lie{g}$ with a Levi subalgebra $\lie{m}$, and $K_M$ a reductive subgroup of $K$ such that $\lie{k}_M$ is contained in $\lie{m}$ and $K_M$ normalizes $\lie{p}$ and $\lie{m}$.
For an $(\lie{m}, K_M)$-module $V$, we consider the cohomologically induced module
\begin{align*}
	\calR^i(V) := \Dzuck{K}{K_M}{i}(\univ{g}\otimes_{\univ{p}}V),
\end{align*}
where $\Dzuck{K}{K_M}{i}$ is the $i$-th Zuckerman derived functor.
We refer the reader to \cite{KnVo95_cohomological_induction} for the cohomological induction.
If $G/P$ is $G'$-spherical, we take a reductive subgroup $L'$ of $M$ as in \eqref{intro:eqn:L}.
The following theorem is proved in Theorem \ref{thm:ReductionToFiberCohInd}.

\begin{theorem}\label{intro:thm:ReductionToFiberCoh}
	Let $V$ be an irreducible $(\lie{m}, K_M)$-module with annihilator $I \subset \univ{m}$.
	If $G/P$ is $G'$-spherical and $\PIdeg((\univ{m}/I)^{L'}) < \infty$, then $\calR^i(V)|_{(\lie{g'}, K')}$ has uniformly bounded multiplicities.
	If $\Ann_{\univ{g}}(\calR^i(V)) = \Ann_{\univ{g}}(\univ{g}\otimes_{\univ{p}}V)$, then the converse is also true.
\end{theorem}

The assumption $\Ann_{\univ{g}}(\calR^i(V)) = \Ann_{\univ{g}}(\univ{g}\otimes_{\univ{p}}V)$ is not true in general.
In fact, $\calR^i(V)$ can be finite dimensional and non-zero even if $P$ is Borel.
If $K_M = K$ and $i = 0$, then the assumption of the annihilators clearly holds.
In the case, $\calR^0(V)$ is the underlying Harish-Chandra module of a holomorphic discrete series representation (or its analytic continuation).
In \cite[Theorems B and D]{Ko08}, using the method of visible actions, T. Kobayashi shows that if $(G_\RR, G'_\RR)$ is a holomorphic symmetric pair, then $\calR^0(V)|_{\lie{g'}, K'}$ has uniformly bounded multiplicities.
Theorem \ref{intro:thm:ReductionToFiberCoh} gives another proof and a generalization of his theorems.

In the case of holomorphic discrete series representations, the unipotent radical of $P$ is abelian.
Suppose that $(G, G')$ is a symmetric pair and the unipotent radical of $P$ is abelian.
T. Kobayashi shows in \cite[Corollary 15]{Ko05} that $G/P$ is $G'$-spherical, and gives a conjecture in \cite[Conjecture 4.3]{Ko11} that $\calR^i(V)|_{\lie{g'}, K'}$ has uniformly bounded multiplicities if $V$ is finite dimensional.
Theorem \ref{intro:thm:ReductionToFiberCoh} gives an affirmative answer to the conjecture.

To use Theorems \ref{intro:thm:ReductionToFiberParabolic} and \ref{intro:thm:ReductionToFiberCoh}, it is important to find $G'$-spherical flag varieties $G/P$.
Recently, the sphericity of flag varieties are well-studied by several mathematicians.
If $(G, G')$ is a symmetric pair, the classification of $G'$-spherical flag varieties $G/P$ is given by He--Ochiai--Nishiyama--Oshima \cite{HeOcNiOs13_double_flag}.
If $G$ is a classical group, R. Avdeev and A. Petukhov give a complete classification of tuples $(G, G', P)$ such that $G/P$ is $G'$-spherical in \cite{AvPe14,AvPe20}.
See also references therein.

We shall state the main theorem for the induction case.
A large part of the theorem and its proof is essentially the same as one for the restriction case.
In fact, the problem to determine multiplicities in induced representations can be regarded as the `branching problem' from $\rring{G/G'}\otimes \univ{g}$ to $\univ{g}$.
Suppose that $G$ and $G'$ are connected.

For a $G'$-stable subvariety $S\subset \lie{g'}^*$, we set
\begin{align*}
	\Ind^G_{G'}(S) := \set{(gG', \lambda) \in G/G'\times \lie{g}^*:
	(\Ad^*(g^{-1})\lambda)|_{\lie{g'}} \in S}.
\end{align*}
Then if $S$ is a coadjoint $G'$-orbit, $\Ind^G_{G'}(S)$ is a symplectic subvariety of the Poisson variety $G/G'\times \lie{g}^*$.
See Lemma \ref{lem:InducedOrbit}.
If $G/G'$ is $G$-spherical, we take a connected reductive subgroup $L_\RR$ of $G'_\RR$ as $L'_\RR$ in Theorem \ref{intro:thm:ReductionToFiberParabolic}.

\begin{theorem}\label{intro:thm:Induction}
	Let $V'$ be a Casselman--Wallach representation of $G'_\RR$.
	Set $I:=\Ann_{\univ{g'}}(V')$.
	Then the following conditions on $V'$ are equivalent:
	\begin{enumparen}
		\item $\sup_{V}\dim_{\CC}(\Hom_{G'_\RR}(V, V')) < \infty$, where the supremum is over all irreducible Casselman--Wallach representations $V$ of $G_\RR$
		\item $\sup_{V}\dim_{\CC}(\Hom_{\lie{g'}, K'}(V, V'_K)) < \infty$, where the supremum is over all irreducible $(\lie{g}, K)$-modules $V$
		\item $\PIdeg((\univ{g}/I\univ{g})^{G'}) < \infty$
		\item $(S(\lie{g})/\sqrt{\gr(I)}S(\lie{g}))^{G'}$ is Poisson-commutative
		\item the action of $G$ on $\Ind^{G}_{G'}(\calV(\gr(I)))$ is coisotropic
		\item $(\univ{g}/I\univ{g})^{G'}$ is finitely generated as a $\univcent{g}$-module
		\item $G/G'$ is $G$-spherical and $V'|_{L_\RR}$ has uniformly bounded multiplicities.
	\end{enumparen}
	If, in addition, $V'$ is unitarizable and $\calH'$ is the Hilbert completion, then the above conditions are equivalent to
	\begin{enumparen}
		\setcounter{enumi}{7}
		\item the unitarily induced representation $\uInd^{G_\RR}_{G'_\RR}(\calH')$ of $\calH'$ has uniformly bounded multiplicities.
	\end{enumparen}
\end{theorem}

If $V'$ is the trivial representation, then $\Ind^{G}_{G'}(\calV(\gr(I)))$ is just the cotangent bundle $T^*(G/G')$ and $(\univ{g}/I\univ{g})^{G'}$ is isomorphic to the algebra of $G$-invariant differential operators on $G/G'$.
Hence Theorem \ref{intro:thm:Induction} is a generalization of the characterization of the sphericity of $G/G'$ stated in the beginning of this section.

Suppose that $G/G'$ is $G$-spherical and $V'$ is irreducible.
In Corollary \ref{cor:ReductionToFiber}, we will see that there exists a constant $C > 0$ independent of $V'$ such that
\begin{align}
	C^{-1}\cdot \sup_{V_L}\dim_{\CC}(\Hom_{L_\RR}(V', V_L)) &\leq 
	\sup_{V} \dim_{\CC}(\Hom_{G'_\RR}(V, V')) \\
	&\leq C\cdot \sup_{V_L}\dim_{\CC}(\Hom_{L_\RR}(V', V_L)) \label{intro:eqn:ReductionToFiber}
\end{align}
where $V$ (resp.\ $V_L$) is an irreducible Casselman--Wallach representation of $G_\RR$ (resp.\ $L_\RR$).
Note that a similar inequation holds for Theorem \ref{intro:thm:ReductionToFiberParabolic}.

The inequation \eqref{intro:eqn:ReductionToFiber} can be considered as a generalization of special cases of known results.
In \cite{Sa93}, F. Sato shows that we can choose $C = 1$ in $\eqref{intro:eqn:ReductionToFiber}$ in the context of reductive algebraic groups (i.e.\ $G_\RR$ and $G'_\RR$ are compact in our setting).
In the case, we need to replace $L_\RR$ by a complexification $L$ defined as $L'$ in \eqref{intro:eqn:L}.
He also shows that the multiplicity $\Hom_{G'}(V, V')$ is given by $\Hom_{L}(V', V_L)$ if the highest weight of $V$ is enough large in some sense.
The author generalize these results to quasi-affine spherical varieties in \cite{Ki14}.

T. Kobayashi introduces the notion of visible actions on complex manifolds in \cite{Ko04}, and shows the propagation theorem of the multiplicity-freeness in \cite{Ko97_multiplicity_free,Ko13}.
Let $X$ be a complex manifold with a strongly visible $G_\RR$-action
and $\calV$ a $G_\RR$-equivariant vector bundle on $X$ with some properties.
Roughly speaking, the propagation theorem says that if the fiber $\calV_x|_{L_\RR}$ is multiplicity-free for a suitable subgroup $L_\RR$ of $(G_\RR)_x$, then any unitary representation realized in $\calO(X, \calV)$ is multiplicity-free.
Our theorems can be considered as a propagation theorem for the uniform boundedness of multiplicities based on spherical actions.
Remark that unlike the above two known results, we can not prove the multiplicity-freeness by our estimate.

The Kobayashi--Oshima uniformly bounded theorem \cite[Theorem B]{KoOs13} says that if $G/G'$ is $G$-spherical, then there exists a constant $C > 0$ such that $\dim_{\CC}\Hom_{G'_\RR}(V, V')\leq C\cdot \dim_{\CC}(V')$ for any irreducible Casselman--Wallach representation $V$ of $G_\RR$ and irreducible finite-dimensional representation $V'$ of $G'_\RR$.
The estimate \eqref{intro:eqn:ReductionToFiber} is a refinement of this theorem for reductive and algebraic $G'_\RR$.
Remark that in the paper, $G'_\RR$ is not assumed to be reductive and algebraic, and the constant $C$ is given explicitly.
See \cite{AiGoDm16,Ta18} for alternative proofs and generalizations.

In \cite{Ya94}, H. Yamashita gives a sufficient condition for the finiteness of multiplicities in a $\lie{g}$-module by the associated variety of the module.
His theorem says that the restriction of a finitely generated $\lie{g}$-module $V$ to a subalgebra $\lie{h}$ is finitely generated if $\AV(V)\cap \lie{h}^{\perp} = 0$, where $\AV(V)$ is the associated variety of $V$.
As an application of his result to a Borel subalgebra of $\lie{g'}$, one can give a sufficient condition for the uniform boundedness by the associated variety.
We do not know any relation between Yamashita's criterion and our criteria.

Our theorem can be applied to the cases of so called the multiplicity one theorem.
If $(G_\RR, G'_\RR) = (\GL(n+1,\CC), \GL(n, \CC))$, $(\upO(n+1, \CC), \upO(n, \CC))$, $(\SO(n+1, \CC), \SO(n,\CC)$ or their four real forms, then $\dim_{\CC}\Hom_{G'_\RR}(V, V')$ is at most one for any irreducible Casselman--Wallach representation $V$ of $G_\RR$ and irreducible Casselman--Wallach representation $V'$ of $G'_\RR$.
This result is proved by Aizenbud--Gourevitch \cite{AiGo09} (for $\GL$) and Sun--Zhu \cite{SuZh12}.
If $(G_\RR, G'_\RR)$ is locally isomorphic to one of the above $7$ pairs, 
$\dim_{\CC}\Hom_{G'_\RR}(V, V')$ is bounded by a constant independent of $V$ and $V'$ by the Kobayashi--Oshima uniformly bounded theorem \cite[Theorem D]{KoOs13}.
Applying Theorem \ref{intro:thm:Induction} to the setting, we can give another proof of the Kobayashi--Oshima's result.
Sun--Zhu also proves the multiplicity one theorem for the Jacobi group cases.
We can apply Theorem \ref{intro:thm:Induction} to the cases and obtain the uniform boundedness of multiplicities.
We treat this problem in Subsection \ref{subsect:MultiplicityOne}.

In the above known results, the uniform boundedness of multiplicities (or the multiplicity-freeness) is proved only for Casselman--Wallach representations.
Although our estimate is not sharp, our theorems can be applied to $(\lie{g}, K)$-modules, Casselman--Wallach representations and unitary representations (and objects in the BGG category $\calO$).
Remark that it is an open problem whether the restriction map
\begin{align*}
	\Hom_{G'_\RR}(V, V') \hookrightarrow \Hom_{\lie{g'}, K'}(V_K, V'_K)
\end{align*}
is bijective if they are finite dimensional.

This paper is organized as follows.
In Section 2, we review the properties of the PI degree and prove \ref{intro:item:RestrictionPIdeg} $\Rightarrow$ \ref{intro:item:RestrictionPoisson} in Theorem \ref{intro:thm:Restriction}.
In Section 3, we recall several fundamental notions about $(\lie{g}, K)$-modules and Casselman--Wallach representations.
In Section 4, we recall several results in our previous paper \cite{Ki20} related to uniformly bounded families, and prove the upper bound of \eqref{intro:eqn:BoundSup}.
Section 5 is devoted to prove the lower bound of \eqref{intro:eqn:BoundSup} in a general setting.
To show Theorem \ref{intro:thm:Induction}, we study $\rring{G/G'}\otimes \univ{g}$-modules in Section 6.
The proof of \ref{intro:item:RestrictionCW} $\Leftrightarrow$ \ref{intro:item:RestrictionGK} $\Leftrightarrow$ \ref{intro:item:RestrictionPIdeg} in Theorem \ref{intro:thm:Restriction} is completed in Section 7.
In Section 8, we deal with (cohomologically) parabolically induced representations.
We prove Theorems \ref{intro:thm:ReductionToFiberParabolic} and \ref{intro:thm:ReductionToFiberCoh} here.
Section 9 is devoted to the study of Poisson varieties.
Here we prove the equivalence \ref{intro:item:RestrictionPIdeg} $\Leftrightarrow$ \ref{intro:item:RestrictionPoisson} $\Leftrightarrow$ \ref{intro:item:RestrictionCoisotropic} in Theorem \ref{intro:thm:Restriction}.
We treat with unitary representations in Section 10.
We summarize fundamental facts in functional analysis in the section.

\subsection*{Notation and convention}

Let $\rring{X}$ denote the algebra of regular functions on an algebraic variety $X$.
Any algebraic variety in this paper is defined over $\CC$.
Whenever we say that $\calA$ is a $\CC$-algebra, $\calA$ is unital and associative.

In this paper, any Lie algebra (without free ones) is assumed to be finite dimensional.
We express real Lie groups and their Lie algebras by Roman alphabets and corresponding German letters with subscript $(\cdot)_\RR$, and express complex Lie groups (or affine algebraic groups) and their Lie algebras by Roman alphabets and corresponding German letters without the subscript $(\cdot)_\RR$, respectively.
Similarly, we express the complexification of a real Lie algebra by the same German letter as that of the real form without the subscript $(\cdot)_\RR$.
For example, the Lie algebras of real Lie groups $G_\RR, K_\RR$ and $H_\RR$ are denoted as $\lie{g}_\RR, \lie{k}_\RR$ and $\lie{h}_\RR$ with complexifications $\lie{g}$, $\lie{k}$ and $\lie{h}$, respectively.
For a topological group $G$, let $G_0$ denote the identity component of $G$.

In this paper, we use several kinds of categories of modules (or representations) such as $\Mod(\calA)$, $\Mod_{fl}(\lie{g}, K)$, $\Mod_{fd}(G)$, $\calC_{G_\RR}$ and $\calO_{\lie{b}}$.
See Subsection \ref{subsect:CategoriesGKCW} for the notations.
Any representation and module in this paper are assumed to be defined over $\CC$.
Let $\lie{g}$ be a complex reductive Lie algebra and $\lie{p}$ a parabolic subalgebra of $\lie{g}$ with a Levi subalgebra $\lie{l}$.
We regard an $\lie{l}$-module as a $\lie{p}$-module by letting the nilradical of $\lie{p}$ act trivially.
We do the same identification for the Lie group case.

For a representation $V$ of a group $G$ (resp.\ a $\lie{g}$-module of a Lie algebra $\lie{g}$), we write $V^G$ (resp.\ $V^{\lie{g}}$) for the space of all invariant vectors in $V$.
Similarly, for an $\calA$-module $V$ of a $\CC$-algebra $\calA$ and a subset $I \subset \calA$, we denote by $V^I$ the space of all vectors annihilated by $I$.

We list fundamental notations and operations used in a large part of this paper:
\begin{itemize}
	\item $\univ{g}$: the universal enveloping algebra of a complex Lie algebra $\lie{g}$
	\item $\univcent{g}$: the center of $\univ{g}$
	\item $\PIdeg(\cdot)$: the PI degree of a ring defined in Definition \ref{def:PIdeg}
	\item $\Len_{\calA}(\cdot)$: the length of an $\calA$-module
	\item $\Loewy{\cdot}$: the Loewy length of a module
	\item $\SupDim(F(\calC)) := \sup_V \dim_{\CC}(F(V))$, where the supremum is over all isomorphism classes of irreducible objects in an abelian category $\calC$
	\item $(\cdot) \otimes (\cdot)$ (without subscript): the tensor product over $\CC$
	\item $(\cdot) \boxtimes (\cdot)$: the external tensor product over $\CC$.
	\item $\Ind^{G_\RR}_{G'_\RR}(\cdot)$: the (non-normalized) induced representation defined by smooth functions
	\item $\Dzuck{K}{M}{i}(\cdot)$: the $i$-the Zuckerman derived functor
\end{itemize}

\subsection*{Acknowledgments}

I would like to thank T. Tauchi  for reading the manuscript very carefully and for many corrections.
The author was partially supported by Waseda University Grants for Special Research Projects (No. 2019C-528).

\section{Polynomial identity}

In this section, we deal with polynomial identities on algebras.
The notion of PI degree plays a central role in this paper.

\subsection{PI degree}

We review fundamental properties of the PI degree defined in Definition \ref{def:PIdeg}.
We refer the reader to \cite[Chapter 13]{McRo01_noncommutative}.

As a $\ZZ$-coefficient non-commutative polynomial with $n$ indeterminates, we set
\begin{align*}
s_{n}(X_1, X_2, \ldots, X_n) := \sum_{w \in \frakS_n} \sgn(w)X_{w(1)}X_{w(2)}\dots X_{w(n)},
\end{align*}
where $\frakS_n$ is the symmetric group of degree $n$ and $\sgn$ is the signature character of $\frakS_n$.
The following fact is a key to control multiplicities in the representation theory by the PI degree.

\begin{fact}[{\cite[Proposition 3.2 and Theorem 3.3]{McRo01_noncommutative}}]\label{fact:amitsurLevitzki}
	Fix an integer $n > 0$.
	\begin{enumparen}
		\item(Amitsur--Levitzki) The polynomial $s_{m}$ ($m\geq 2n$) is a polynomial identity of the matrix algebra $M_n(\CC)$.
		\item Conversely, $s_m$ ($m < 2n$) is not a polynomial identity of $M_n(\CC)$.
	\end{enumparen}
	In particular, we have $\PIdeg(M_n(\CC)) = n$.
\end{fact}

Let $\calA$ be a $\CC$-algebra of at most countable dimension
and $(\pi, V)$ an irreducible $\calA$-module.
Then $\pi(\calA)$ is dense in $\End_{\CC}(V)$ in the sense of the Jacobson density theorem as follows.
If $\dim_{\CC}(V) < \infty$, then we have $\pi(\calA) = \End_{\CC}(V)$.
If $\dim_{\CC}(V) = \infty$, then we have $\pi(\calA)|_F = \Hom_{\CC}(F, V)$
for any finite-dimensional subspace $F$ of $V$.
This implies $\PIdeg(\pi(\calA)) = \dim_{\CC}(V)$ by Fact \ref{fact:amitsurLevitzki}.
Note that $\pi(\calA)$ may not be dense in $\End_{\CC}(V)$
without the assumption that $\calA$ has at most countable dimension
because $\End_{\calA}(V)$ may be greater than $\CC$.

This observation tells us that the PI degree restricts the dimensions of all irreducible $\calA$-modules.
More precisely, the following propositions hold.

\begin{proposition}\label{prop:pidLowerbound}
	Let $\calA$ be a $\CC$-algebra of at most countable dimension.
	Then the dimension of any irreducible $\calA$-module is less than or equal to $\PIdeg(\calA)$.
\end{proposition}

\begin{proof}
	As we have seen in the above, we have
	\begin{equation*}
	\dim_{\CC}(V) = \PIdeg(\pi(\calA)) \leq \PIdeg(\calA)
	\end{equation*}
	for any irreducible $\calA$-module $(\pi, V)$.
\end{proof}

\begin{proposition}\label{prop:pidegUpperbound}
	Let $\calA$ be a $\CC$-algebra and $\set{(V_\lambda, \pi_\lambda)}_{\lambda \in \Lambda}$
	a family of $\calA$-modules.
	If $\bigcap_{\lambda \in \Lambda}\Ker(\pi_\lambda)=0$, then we have
	\begin{align*}
		\PIdeg(\calA)\leq \sup_{\lambda\in \Lambda}\set{\dim_{\CC}(V_{\lambda})}.
	\end{align*}
\end{proposition}

\begin{proof}
	By assumption, the direct product of $\pi_\lambda$ is injective:
	\begin{align*}
		\prod_{\lambda \in \Lambda}\pi_{\lambda}\colon \calA \hookrightarrow \prod_{\lambda \in \Lambda}\End_{\CC}(V_{\lambda}).
	\end{align*}
	Hence we have $\PIdeg(\calA)\leq \sup_{\lambda\in \Lambda}\set{\PIdeg(\End_{\CC}(V_{\lambda}))}=\sup_{\lambda\in \Lambda}\set{\dim_{\CC}(V_{\lambda})}$.
\end{proof}

Recall that the Jacobson radical $J(\calA)$ of a $\CC$-algebra $\calA$ is the intersection of all primitive ideals in $\calA$.
The following fact is stated for universal enveloping algebras in \cite[Proposition 3.1.15]{Di96}.
The same proof works for our setting.
See also \cite[Lemma 3.1.14]{Di96}.

\begin{fact}\label{fact:NilpotentJacobson}
	Let $\calA$ be a noetherian $\CC$-algebra of at most countable dimension.
	The intersection of all prime ideals in $\calA$ coincides with the Jacobson radical $J(\calA)$.
	In particular, $J(\calA)$ is a nilpotent ideal.
\end{fact}

In many places, we will use the following results to reduce the computation of the PI degree to that of easier algebras.

\begin{proposition}\label{prop:pidegIdeal}
	Let $\calA$ be a $\CC$-algebra and $I$ a two-sided ideal of $\calA$.
	\begin{enumparen}
		\item\label{item:pidegIdealNilpotent} If $I$ is a nilpotent ideal, then we have $\PIdeg(\calA/I) = \PIdeg(\calA)$.
		\item\label{item:pidegIdealProduct} Let $J_1, J_2, \ldots, J_n$ be two-sided ideals containing $I$ such that $J_1J_2\ldots J_n \subset I$.
		Then we have
		\begin{align*}
			\PIdeg(\calA/I) = \max_i\set{\PIdeg(\calA/J_i)}.
		\end{align*}
	\end{enumparen}
\end{proposition}

\begin{proof}
	We shall show the first assertion.
	$\PIdeg(\calA/I) \leq \PIdeg(\calA)$ is trivial.
	To show the converse inequality, we assume $\PIdeg(\calA/I) < \infty$.
	Set $n := \PIdeg(\calA/I)$ and take a constant $r$ such that $I^r = 0$.
	Then we can take a polynomial identity $f$ of $\calA/I$ with $m$ indeterminates such that $f$ is not a polynomial identity of $M_{n+1}(\CC)$.
	Since $f$ is not a polynomial identity of $M_{n+1}(\CC)$, we can take $A_1, A_2, \ldots, A_{m+1} \in M_{n+1}(\CC)$ such that $f(A_1, A_2, \ldots, A_m)A_{m+1}$ is an idempotent.
	Hence $(f\cdot X_{m+1})^r$ is not a polynomial identity of $M_{n+1}(\CC)$.
	Since $(f\cdot X_{m+1})^r$ is a polynomial identity of $\calA$, we obtain $\PIdeg(\calA) \leq n$.

	We shall show the second assertion.
	Set $J:= \bigcap_i J_i$.
	By assumption, we have $J^n \subset I \subset J$.
	Hence, by \ref{item:pidegIdealNilpotent}, we obtain
	\begin{align*}
		\PIdeg(\calA/J) \leq \PIdeg(\calA/I) \leq \PIdeg(\calA/J^n) = \PIdeg(\calA/J).
	\end{align*}
	Since $\PIdeg(\calA/J) = \max_i\set{ \PIdeg(\calA/J_i)}$, this shows the assertion.
\end{proof}

\begin{proposition}\label{prop:PIdegSemiprimitive}
	Let $\calA$ be a $\CC$-algebra of at most countable dimension.
	If $\calA$ is semiprimitive (i.e. $J(\calA) = 0$) or noetherian, then the supremum of the dimensions of all irreducible $\calA$-modules is equal to $\PIdeg(\calA)$.
\end{proposition}

\begin{proof}
	If $\calA$ is semiprimitive, we can take a family $((\pi_\lambda, V_\lambda))_{\lambda \in \Lambda}$ of irreducible $\calA$-modules such that $\bigcap_{\lambda} \Ker(\pi_\lambda) = 0$.
	Hence the assertion follows from Propositions \ref{prop:pidLowerbound} and \ref{prop:pidegUpperbound}.

	If $\calA$ is noetherian, then $J(\calA)$ is nilpotent by Fact \ref{fact:NilpotentJacobson}.
	Hence the assertion follows from Proposition \ref{prop:pidegIdeal} \ref{item:pidegIdealNilpotent} and the semiprimitive case.
\end{proof}

By Proposition \ref{prop:PIdegSemiprimitive}, if $\calA$ is semiprimitive and $\PIdeg(\calA) = 1$, then $\calA$ is commutative.
We consider the PI degree of an algebra of group invariant elements.
A $\CC$-algebra $\calA$ equipped with a linear $H$-action of a group $H$ is called an $H$-algebra if the $H$-action is given by automorphisms of the algebra.
An $\calA$-module $V$ equipped with a linear $H$-action is said to be an $(\calA, H)$-module if the multiplication $\calA\otimes V\rightarrow V$ is $H$-linear.

\begin{lemma}\label{lem:GeneralizedPairConnected2}
	Let $H$ be a finite group and $V$ an $\calA$-module of an $H$-algebra $\calA$.
	Then we have
	\begin{align*}
		\Len_{\calA}(V)\leq \Len_{\calA^H}(V) \leq |H|\Len_{\calA}(V).
	\end{align*}
\end{lemma}
	
\begin{proof}
	The first inequality is trivial.
	We denote by $\CC[H]$ the group ring of $H$.
	
	We consider $\CC[H]\otimes V$ as an $(\calA, H)$-module via
	$a\cdot (g \otimes v) = g \otimes (g^{-1}a)v$ ($a\in \calA, g\in H, v \in V$)
	and the left translation of $H$ on $\CC[H]$.
	Then $(\CC[H]\otimes V)^H$ is isomorphic to $V$ as an $\calA^{H}$-module.
	The functor $(\cdot)^H$ is exact and sends an irreducible $(\calA, H)$-module to an irreducible $\calA^H$-module or zero.
	Hence we have
	\begin{equation*}
		\Len_{\calA^H}(V) \leq \Len_{\calA, H}(\CC[H]\otimes V) \leq \Len_{\calA}(\CC[H]\otimes V) = |H|\Len_{\calA}(V). \qedhere
	\end{equation*}
\end{proof}

\begin{proposition}\label{prop:PIdegStableConnected}
	Let $H$ be a finite group, $\calA$ a noetherian $H$-algebra of at most countable dimension and $I$ a two-sided ideal of $\calA$.
	Then we have
	\begin{align*}
		\PIdeg(\calA^H / I \cap \calA^H) \leq \PIdeg(\calA/I) \leq |H| \PIdeg(\calA^H /I\cap \calA^H).
	\end{align*}
\end{proposition}

\begin{proof}
	The first inequality is trivial since $\calA^H / I \cap \calA^H$ is a subalgebra of $\calA/I$.
	By Proposition \ref{prop:PIdegSemiprimitive}, $\PIdeg(\calA/I)$ coincides with the supremum of the dimensions of all irreducible $\calA$-modules annihilated by $I$.
	Let $V$ be an irreducible $\calA$-module annihilated by $I$.
	Then we have
	\begin{align*}
		\dim_{\CC}(V) &\leq \Len_{\calA^H}(V)\cdot \PIdeg(\calA^H / I \cap \calA^H) \\
		&\leq |H| \PIdeg(\calA^H /I\cap \calA^H)
	\end{align*}
	by Proposition \ref{prop:pidLowerbound} and Lemma \ref{lem:GeneralizedPairConnected2}.
	This shows the second inequality.
\end{proof}

Under a strong assumption, we can give an upper bound of the dimension of an algebra using the PI degree.

\begin{proposition}\label{prop:UpperBoundDimA}
	Let $\calA$ be a noetherian $\CC$-algebra of at most countable dimension.
	Then we have
	\begin{align*}
		\dim_{\CC}(\calA) \leq \Len_{\calA\otimes \calA^\opalg}(\calA) \cdot \PIdeg(\calA)^2,
	\end{align*}
	where $\calA^\opalg$ is the opposite algebra of $\calA$.
\end{proposition}

\begin{proof}
	Clearly we can assume $\Len_{\calA\otimes \calA^\opalg}(\calA), \PIdeg(\calA) < \infty$.
	We shall show that $\calA$ is finite dimensional.

	Since $\calA/J(\calA)$ is semiprimitive and $\Len_{\calA\otimes \calA^\opalg}(\calA), \PIdeg(\calA) < \infty$, the algebra $\calA/J(\calA)$ can be embedded in a finite direct product of matrix algebras.
	In particular, $\calA/J(\calA)$ is finite dimensional.
	$J(\calA)^k/J(\calA)^{k+1}$ is a finitely generated $\calA/J(\calA)$-module and hence finite dimensional for any $k \geq 0$.
	Since $J(\calA)$ is nilpotent ideal by Fact \ref{fact:NilpotentJacobson}, $\calA$ is also finite dimensional.

	Any irreducible $\calA\otimes \calA^\opalg$-module $M$ is finite dimensional and
	we have $M\simeq M_1 \boxtimes M_2$ for some irreducible $\calA$-module $M_1$ and irreducible $\calA^\opalg$-module $M_2$.
	By Proposition \ref{prop:pidegUpperbound}, we obtain
	\begin{equation*}
		\begin{split}
			\dim_\CC(\calA) &\leq \Len_{\calA\otimes \calA^\opalg}(\calA) \cdot \PIdeg(\calA) \PIdeg(\calA^\opalg) \\
			&= \Len_{\calA\otimes \calA^\opalg}(\calA) \cdot \PIdeg(\calA)^2. \qedhere
		\end{split}
	\end{equation*}
\end{proof}

\subsection{Poisson algebra and polynomial identity}\label{subsect:PoissonPolynomialIdentity}

Recall that when a $\CC$-algebra $\calA$ is semiprimitive, then $\calA$ is commutative if and only if $\PIdeg(\calA) = 1$.
Intuitively, a $\CC$-algebra of finite PI degree is `almost commutative'.
If a $\CC$-algebra $\calA$ has a filtration such that $\gr\calA$ is commutative,
then $\gr\calA$ has a natural Poisson structure.
In this subsection, we will show that if $\calA$ has finite PI degree and the Jacobson radical of $\calA$ is nilpotent, then $\gr \calA$ is Poisson-commutative modulo its nilradical.

We consider polynomial identities on the matrix algebra $M_n(\CC)$.
Several polynomial identities on the algebra are known, e.g.\ $s_{2n}$ (Fact \ref{fact:amitsurLevitzki}).
We use naive identities coming from the Cayley--Hamilton theorem.
We define non-commutative polynomials $p_i$ ($i\geq 1$) with indeterminates $X$, $Y$ and $Z$ by
\begin{align}
	p_1(X, Y; Z) &:= [Z, X], \\
	p_{i}(X, Y; Z) &:= [p_{i-1}(X, Y; Z), p_{i-1}(X, Y; Y^{i-1})] \quad (i \geq 2). \label{def:IdentityMatrix}
\end{align}

\begin{lemma}\label{lem:IdentityMatrix}
	For any $A, B \in M_n(\CC)$, we have $p_n(A, B; B^n) = 0$.
\end{lemma}

\begin{proof}
	By the Cayley--Hamilton theorem, there are $a_0, a_1, \ldots, a_{n-1} \in \CC$ such that $B^n + a_{n-1}B^{n-1} + \cdots + a_0 = 0$.
	Applying $[\cdot, A]$ to this equation, we have
	\begin{align*}
		p_1(A, B; B^n) + a_{n-1} p_1(A, B; B^{n-1}) + \cdots + a_1 p_1(A, B; B) = 0.
	\end{align*}
	Applying $[\cdot, p_1(A, B; B)]$ to the equation, we have
	\begin{align*}
		p_2(A, B; B^n) + a_{n-1} p_2(A, B; B^{n-1}) + \cdots + a_2 p_2(A, B; B^2) = 0.
	\end{align*}
	Applying such operations iteratively, we obtain
	$p_n(A, B; B^n) = 0$.
\end{proof}

We recall the definition of Poisson algebras.
We say that a $\CC$-algebra $\calA$ equipped with a skew-symmetric bilinear product $\{\cdot, \cdot\}$ is a Poisson algebra if $\{\cdot, \cdot\}$ satisfies
the Jacobi identity and $\{X, \cdot\} \in \End_{\CC}(\calA)$ is a derivation of the $\CC$-algebra $\calA$ for any $X \in \calA$.

Let $\calP\langle X_1, X_2, \ldots, X_n\rangle$ be the free Poisson algebra with free generators $X_i$, which is the symmetric algebra of the free Lie algebra with free generators $X_i$.
We define a Poisson algebra version of the polynomials $p_i$ by
\begin{align}
	q_1(X, Y; Z) &:= \{Z, X\}, \\
	q_{i}(X, Y; Z) &:= \{q_{i-1}(X, Y; Z), q_{i-1}(X, Y; Y^{i-1})\} \quad (i \geq 2), \label{def:PoissonIdentity}
\end{align}
which are elements of $\calP\langle X, Y, Z\rangle$.

\begin{lemma}\label{lem:RelationQ}
	For any $j \geq i \geq 1$, we have
	\begin{align*}
		q_i(X, Y; Y^j) &= \binom{j}{i} q_i(X, Y; Y^i)Y^{j-i}, \\
		q_i(X, Y; Y^i) &= i q_{i-1}(X, Y; Y^{i-1}) \{Y, q_{i-1}(X;Y,Y^{i-1})\} \quad (i \geq 2).
	\end{align*}
\end{lemma}

\begin{proof}
	We show the lemma by induction on $i$.
	Since $\{\cdot, X\}$ is a derivation, the equation is trivial for $i = 1$.
	For $i > 1$, we have
	\begin{align*}
		&q_i(X, Y; Y^j) \\
		= &\{q_{i-1}(X, Y; Y^j), q_{i-1}(X, Y; Y^{i-1})\} \\
		= & \binom{j}{i-1} \{q_{i-1}(X, Y; Y^{i-1})Y^{j-i+1}, q_{i-1}(X, Y; Y^{i-1})\} \\
		= & \binom{j}{i-1} q_{i-1}(X, Y; Y^{i-1}) (j-i+1)Y^{j-i}\{Y, q_{i-1}(X, Y; Y^{i-1})\} \\
		= & \binom{j}{i-1}\frac{j - i + 1}{i}Y^{j-i}q_i(X, Y; Y^i)\\
		=& \binom{j}{i} Y^{j-i}q_i(X, Y; Y^i).
	\end{align*}
	In the fourth line for $j = i$, we have obtained the second equation in the assertion,
	and in the fifth line, we have used the obtained equation.
\end{proof}

\begin{lemma}\label{lem:PolynomialPoissonIdentity}
	For any $n\geq 1$, there exists a unique polynomial $f_n(X, Y; t) \in \calP\langle X, Y \rangle [t]$ such that for any integer $m\geq n$, we have
	\begin{align*}
		q_n(X^m, Y; Y^n) = X^{(m-n)2^{n-1}} f_{n}(X, Y; m).
	\end{align*}
	The leading coefficient of $f_{n}(X, Y; t)$ is of the form $c\set{Y, X}^k X^{l-1}$ for some positive integers $c$, $k$ and $l$.
\end{lemma}

\begin{proof}
	We write $f_n(t) = f_{n}(X, Y; t)$ for short.
	As a $\CC$-algebra, $\calP\langle X, Y\rangle$ is an integral domain since $\calP\langle X, Y \rangle$ is the symmetric algebra of the free Lie algebra.
	This shows the uniqueness in the assertion.

	We construct $f_n$ inductively on $n$.
	By $\set{Y, X^m} = mX^{m-1}\set{Y, X}$, the polynomial $f_1(t) = \set{Y, X} t$ satisfies the assertion.
	For $m\geq n \geq 2$, we have
	\begin{align*}
		&q_{n}(X^m, Y; Y^n)\\
		= &n X^{(m-n+1)2^{n-2}}f_{n-1}(m)\{Y, X^{(m-n+1)2^{n-2}}f_{n-1}(m)\} \\
		= &n X^{(m-n)2^{n-1}} f_{n-1}(m)\\
		&\times ((m-n+1)2^{n-2} X^{2^{n-1}-1}\{Y, X\}f_{n-1}(m) + X^{2^{n-1}}\{Y, f_{n-1}(m)\})
	\end{align*}
	by Lemma \ref{lem:RelationQ}.
	Hence the polynomial
	\begin{align*}
		f_n(t) := n f_{n-1}(t)((t-n+1)2^{n-2} X^{2^{n-1}-1}\{Y, X\}f_{n-1}(t) + X^{2^{n-1}}\{Y, f_{n-1}(t)\})
	\end{align*}
	satisfies the desired equation, and the leading coefficient of $f_n$ is 
	\begin{align*}
		n 2^{n-2} c_{n-1}^2 X^{2^{n-1}-1}\{Y, X\},
	\end{align*}
	where $c_{n-1}$ is the leading coefficient of $f_{n-1}$.
	We have completed the proof.
\end{proof}

Let $\calA$ be a filtered (associative) $\CC$-algebra with filtration $0 = \calA_{-1} \subset \calA_0 \subset \cdots $ satisfying the following conditions:
\begin{enumerate}
	\item $\CC\cdot 1 \subset \calA_0$
	\item $\bigcup_{i} \calA_i = \calA$
	\item $\calA_i \cdot \calA_j \subset \calA_{i+j}$ for any $i, j \geq 0$
	\item $[\calA_i, \calA_j] \subset \calA_{i+j-1}$ for any $i, j \geq 0$.
\end{enumerate}
Then the associated graded algebra $\gr\calA = \bigoplus_{i=0}^{\infty} \calA_{i+1}/\calA_i$ is a Poisson algebra with the Poisson bracket
\begin{align*}
	\{a + \calA_{i - 1}, b + \calA_{j-1}\} := [a, b] + \calA_{i + j - 2}
	\quad (a \in \calA_i, b \in \calA_j).
\end{align*}

In addition to the above conditions, we assume the following:
\begin{enumerate}
	\setcounter{enumi}{4}
	\item \label{item:FinGenGr} $\gr\calA$ is finitely generated as a $\CC$-algebra
\end{enumerate}
By the assumption \ref{item:FinGenGr}, $\calA$ is noetherian and has at most countable dimension.
Note that every minimal prime ideals and the nilradical $\sqrt{(0)}$
of $\gr\calA$ are Poisson ideals, i.e.\ they are stable under $\set{X, \cdot}$ for any $X \in \gr\calA$ (see \cite[Lemma 3.3.3]{Di96}).

\begin{lemma}\label{lem:PIdegAndIdentityP}
	If $n:=\PIdeg(\calA)$ is finite, then $p_n(A, B; B^n) \in J(\calA)$ for any $A, B\in \calA$.
\end{lemma}

\begin{proof}
	By the definition of the Jacobson radical, we can take a family $\set{M_i}_{i\in I}$ of irreducible $\calA$-modules such that $\bigcap_{i \in I} \Ann_{\calA}(M_i) = J(\calA)$.
	By Propositions \ref{prop:pidLowerbound} and \ref{prop:pidegUpperbound}, we have $\sup\set{\dim_{\CC}(M_i): i\in I} = n$.
	This implies $p_n(A, B; B^n) \in J(\calA)$ by Lemma \ref{lem:IdentityMatrix}.
\end{proof}

\begin{theorem}\label{thm:PIdegAndPoissonCommutative}
	If $\PIdeg(\calA)$ is finite, then $\gr\calA / \sqrt{(0)}$ is Poisson-commutative, i.e.\ $\{A, B\} = 0$ for any $A, B \in \gr\calA/ \sqrt{(0)}$.
\end{theorem}

\begin{proof}
	Set $n:=\PIdeg(\calA)$.
	By Lemma \ref{lem:PIdegAndIdentityP}, $p_n(A, B; B^n) \in J(\calA)$ for any $A, B \in \calA$.
	Recall that $J(\calA)$ is a nilpotent ideal by Fact \ref{fact:NilpotentJacobson}.
	By the definition of $p_n$ and $q_n$ (see \eqref{def:IdentityMatrix} and \eqref{def:PoissonIdentity}), we have $q_n(A, B; B^n) \in \gr(J(\calA)) \subset \sqrt{(0)}$ for any $A, B\in \gr\calA$.

	Let $P$ be a minimal prime ideal of $\gr\calA$.
	Then $\gr \calA/P$ is a Poisson algebra.
	Consider the polynomial $f_n(X, Y; t)$ in Lemma \ref{lem:PolynomialPoissonIdentity}.
	Fix $A, B \in \gr\calA / P$.
	Then we have $q_n(A^m, B; B^n) = 0$ and hence $f_n(A, B; m) = 0$ for any integer $m \geq n$.
	Since $f_n(A, B; t)$ is a polynomial on $t$ with coefficients in the integral domain $\gr\calA/ P$, this implies $f_n(A, B; t) = 0$.
	The leading coefficient of $f_n(X, Y; t)$ is of the form $c\{Y, X\}^k X^{l-1}$ for some positive integer $c$, $k$ and $l$.
	Therefore we obtain $\{B, A\} = 0$.
	This shows that $\gr \calA / \sqrt{0}$ is Poisson-commutative
	since $\sqrt{0}$ is the intersection of all minimal prime ideals of $\gr \calA$.
\end{proof}

Let $G$ be an affine algebraic group over $\CC$ and $G'$ a reductive subgroup of $G$.
The following corollary is important to relate the uniform boundedness of multiplicities to coisotropic actions on nilpotent coadjoint orbits.
We will show the relation in Section \ref{sect:Coisotropic} (see Theorems \ref{thm:PIdegAndCoisotropicRest} and \ref{thm:PIdegAndCoisotropicInd}).

\begin{corollary}\label{cor:PidegPoissonCommutativeUg}
	Let $I$ be a two-sided ideal of $\univ{g}^{G'}$.
	If $\PIdeg(\univ{g}^{G'}/I) < \infty$, then $S(\lie{g})^{G'}/\sqrt{\gr(I)}$ is Poisson-commutative.
\end{corollary}

\section{Category of \texorpdfstring{$\lie{g}$}{g}-modules}

We prepare notation about categories of $(\lie{g}, K)$-modules and representations.
Categories $\calC_{\lie{g}, K}$ and $\calC_{G_\RR}$ play an important role to give lower bounds of multiplicities of restrictions and inductions of $(\lie{g}, K)$-modules and representations of $G_\RR$.

\subsection{\texorpdfstring{$(\lie{g}, K)$}{(g, K)}-module}

Our main concern in this paper is restrictions and inductions of $(\lie{g}, K)$-modules.
We recall the notion of pairs $(\lie{g}, K)$ and $(\lie{g}, K)$-modules referring to \cite[Chapter I, IV]{KnVo95_cohomological_induction}.

Let $G$ be an affine algebraic group over $\CC$.
We say that a representation $V$ of $G$ as an abstract group is a $G$-module
if the $G$-action on $V$ is locally finite and for any finite-dimensional subrepresentation $W$ of $V$, the action $G\times W\rightarrow W$ is a morphism of varieties.
Then we say that the $G$-action on $V$ is rational.

\begin{definition}\label{def:pair}
	Let $\lie{g}$ be a complex Lie algebra and $K$ an algebraic group with Lie algebra $\lie{k} \subset \lie{g}$.
	We say that $(\lie{g}, K)$ is a \define{pair} if it satisfies the following conditions.
	\begin{enumparen}
			\item A rational action of $K$ on $\lie{g}$ by Lie algebra automorphisms is given, and the inclusion map $\lie{k}\hookrightarrow \lie{g}$ is $K$-linear.
			\item The differential of the $K$-action on $\lie{g}$ coincides with the adjoint action of $\lie{k}$ on $\lie{g}$.
	\end{enumparen}
	The action of $K$ on $\lie{g}$ is denoted by $\Ad_{\lie{g}}$ (or $\Ad$).
	We say that a pair $(\lie{g'}, K')$ is a \define{subpair} of $(\lie{g}, K)$ if $\lie{g'}$ is a subalgebra of $\lie{g}$ and $K'$ is a (Zariski) closed subgroup of $K$.
\end{definition}

A typical example of pairs is one constructed from a real reductive Lie group $G_{\RR}$.
Let $K_\RR$ be a maximal compact subgroup of $G_\RR$.
Then $(\lie{g}, K)$ forms a pair, where $\lie{g}$ and $K$ are the complexifications of $\lie{g}_\RR$ and $K_\RR$, respectively.

\begin{definition}
	Let $(\lie{g}, K)$ be a pair.
	A $\lie{g}$-module $V$ with a rational $K$-action is called a \define{$(\lie{g}, K)$-module} if the following conditions hold.
	\begin{enumparen}
		\item The multiplication $\lie{g}\otimes V\rightarrow V$ is $K$-linear.
		\item The differential of the $K$-action on $V$ coincides with the restriction of the $\lie{g}$-action to $\lie{k}$.
	\end{enumparen}
\end{definition}

\subsection{Categories \texorpdfstring{$\calC_{\lie{g}, K}$, $\calC_{G_\RR}$ and $\calO_{\lie{g}}$}{CgK, CGR and Og}}\label{subsect:CategoriesGKCW}

We prepare several categories and notation related to $(\lie{g}, K)$-modules.
For a compact Lie group $K_\RR$ and a continuous representation $V$ of $K_\RR$, we denote by $V_K$ the subspace of all $K_\RR$-finite vectors in $V$.
Then the $K_\RR$-action on $V_K$ lifts canonically to a rational $K$-action, where $K$ is the complexification of $K_\RR$.

Let $G_\RR$ be a real reductive Lie group and $(\lie{g}, M)$ a pair.
We use the following notation:
\begin{itemize}
	\item $\Mod(\lie{g}, M)$: the category of $(\lie{g}, M)$-modules
	\item $\Mod_{fl}(\lie{g}, M)$: the category of $(\lie{g}, M)$-modules of finite length
	\item $\Mod_{fd}(\lie{g}, M)$: the category of completely reducible finite-dimensional \sloppy$(\lie{g}, M)$-modules
	\item $\Mod(G_\RR)$: the category of continuous representations of $G_\RR$
	\item $\Mod_{fd}(G_\RR)$: the category of completely reducible finite-dimensional (continuous) representations of $G_\RR$
	\item $\CWCat{G_\RR}$: the full subcategory of $\Mod(G_\RR)$ whose objects are Casselman--Wallach representations.
\end{itemize}

A Casselman--Wallach representation of $G_\RR$ is a smooth admissible representation on a nuclear Fr\'echet space of moderate growth and of finite length.
See e.g.\ \cite[Chapter 11]{Wa92_real_reductive_II} and \cite{Vo08} for Casselman--Wallach representations.
If $M = \set{e}$, we omit the second parameter, e.g.\ $\Mod(\lie{g})$ and $\Mod_{fd}(\lie{g})$.
If the Lie algebra of $M$ is equal to $\lie{g}$, we omit the first parameter, e.g.\ $\Mod(M)$ and $\Mod_{fd}(M)$.

Let $K_\RR$ be a maximal compact subgroup of $G_\RR$.
Then $(\lie{g}, K)$ is a pair.
The following equivalence of categories is well-known.

\begin{fact}[Casselman--Wallach]\label{fact:CasselmanWallach}
	The functor $(\cdot)_K \colon \CWCat{G_\RR}\rightarrow \Mod_{fl}(\lie{g}, K)$ ($V\mapsto V_K$) gives an equivalence of categories.
\end{fact}

The quasi-inverse of the functor $(\cdot)_K$ is given by $V\mapsto \calH^\infty$,
where $\calH$ is a Hilbert globalization of the $(\lie{g}, K)$-module $V$ (i.e.\ $\calH_K\simeq V$) and $\calH^\infty$ is the subspace of all smooth vectors in $\calH$.
To study restrictions and inductions, we need the following categories:

\begin{itemize}
	\item $\calC_{\lie{g}, K}$: the full subcategory of $\Mod_{fl}(\lie{g}, K)$ whose object is isomorphic to a subquotient of $F\otimes V$ for some $F \in \Mod_{fd}(\lie{g}, K)$ and semisimple object $V \in \Mod_{fl}(\lie{g}, K)$
	\item $\calC_{G_\RR}$: the full subcategory of $\CWCat{G_\RR}$ whose object is isomorphic to a subquotient of $F\otimes V$ for some $F \in \Mod_{fd}(G_\RR)$ and semisimple object $V \in \CWCat{G_\RR}$.
\end{itemize}

By Fact \ref{fact:CasselmanWallach}, the categories $\calC_{\lie{g}, K}$ and $\calC_{G_\RR}$ are equivalent.
Let $\lie{b}$ be a Borel subalgebra of $\lie{g}$ and $\lie{t} \subset \lie{b}$ a Cartan subalgebra of $\lie{g}$.

\begin{itemize}
	\item $\calO_{\lie{b}}$: the BGG category of $\lie{g}$ with respect to $\lie{b}$, which is the full subcategory of $\Mod(\lie{g})$ whose object is finitely generated, $\lie{t}$-semisimple and locally $\lie{b}$-finite
	\item $\calO_{\lie{g}}$: the full subcategory of $\Mod(\lie{g})$ whose object is isomorphic to a subquotient of $V_1\oplus V_2 \oplus \cdots \oplus V_r$ for some Borel subalgebras $\lie{b}_i$ and $V_i \in \calO_{\lie{b}_i}$.
\end{itemize}

$\calO_{\lie{b}}$ does not depend on the choice of $\lie{t}$.
We refer the reader to \cite{Hu08_category_o} for the BGG category $\calO$.
By definition, $\calO_{\lie{b}}$ is a full subcategory of $\calO_{\lie{g}}$.

All the categories defined above are closed under taking subquotients and tensoring with finite-dimensional objects.
The categories $\calC_{\lie{g}, K}$, $\calC_{G_\RR}$, $\calO_{\lie{b}}$ and $\calO_{\lie{g}}$ satisfy the following finiteness property: there exists a constant $C > 0$ such that the socle filtration of any object has length $\leq C$.
This is the reason why we use the categories $\calC_{\lie{g}, K}$ and $\calC_{G_\RR}$.
We will show this property in Subsection \ref{sect:Category}.

The constant $C$ will appear in many theorems in this paper.
We give a notation for the constant.

\begin{definition}\label{def:Loewy}
	Let $V$ be an object of an abelian category $\calC$.
	We denote by $\soc(V)$ the maximum semisimple subobject of $V$ and denote by $\soc_k(V)$ the unique subobject of $V$ such that $\soc_k(V)/\soc_{k-1}(V) = \soc(V/\soc_{k-1}(V))$ for $k > 0$.
	Then the filtration
	\begin{align*}
		0 = \soc_0(V) \subset \soc_1(V) \subset \cdots \subset \soc_i(V) \subset \cdots \subset V
	\end{align*}
	is called the socle filtration of $V$.
	We set
	\begin{align*}
		\Loewy{V}:=\inf\set{i \in \ZZ_{>0} : \soc_i(V) = V},
	\end{align*}
	which is called the Loewy length of $V$, and set
	\begin{align*}
		\Loewy{\calC} := \sup\set{\Loewy{V}: V \in \calC}.
	\end{align*}
\end{definition}

Remark that $\Loewy{0} = 1$ in our definition.
As we have mentioned above, we can show $\Loewy{\calC_{\lie{g}, K}}, \Loewy{\calC_{G_\RR}} < \infty$ although $\Loewy{\Mod_{fl}(\lie{g}, K)} = \infty$ in general.

\begin{definition}
	Let $V$ be an object in $\calC$.
	We denote by $\rad(V)$ the radical of $V$, that is, $\rad(V)$ is the intersection of maximal subjects of $V$.
\end{definition}

The following proposition is easy and well-known.

\begin{proposition}\label{prop:LoewyAndRad}
	Let $V$ be an object in $\calC$.
	Set $\rad_0(V) := V$ and $\rad_k(V) := \rad(\rad_{k-1}(V))$ for $k>0$.
	Then we have
	\begin{align*}
		\Loewy{V} = \inf\set{k \in \ZZ_{>0} : \rad_k(V) = 0}.
	\end{align*}
\end{proposition}

\subsection{Induced representations and Zuckerman derived functors}

We recall several kinds of induced representations.
Let $G_\RR$ be a Lie group and $G'_\RR$ a closed subgroup of $G_\RR$.

\begin{definition}
	Let $(\pi, V')$ be a smooth Fr\'echet representation of $G'_\RR$.
	We set
	\begin{align*}
		\Ind^{G_\RR}_{G'_\RR}(V') := \set{f\colon G_\RR\rightarrow V' \text{ (smooth)}: \begin{array}{l} f(gh) = \pi(h)^{-1}f(g) \\(\forall g \in G_\RR, h \in G'_\RR) \end{array}}.
	\end{align*}
	We provide $\Ind^{G_\RR}_{G'_\RR}(V')$ with the topology of uniform convergence on compact subsets of functions and their derivatives.
\end{definition}

It is well-known that $\Ind^{G_\RR}_{G'_\RR}(V')$ is a smooth (and hence continuous) representation of $G_\RR$ by the left translation.
We consider unitarily induced representations only in Section \ref{sect:Unitary}.
Hence we postpone the definition of the induction.

Here we consider two kinds of infinitesimal versions of $\Ind^{G_\RR}_{G'_\RR}(V')$.
Let $(\lie{g}, M)$ be a pair and $\lie{g'}$ an $M$-stable subalgebra of $\lie{g}$ containing $\lie{m}$.
Then $(\lie{g'}, M)$ is a subpair of $(\lie{g}, M)$.

\begin{definition}
	For a $(\lie{g'}, M)$-module $V'$, we set
	\begin{align*}
		\pro{\lie{g}, M}{\lie{g'}, M}(V') &:= \Hom_{\univ{g'}}(\univ{g}, V')_M.
	\end{align*}
	Here $\univ{g}$ is regarded as a $\lie{g'}$-module via the left action
	and the subspace of $M$-finite vectors is taken with respect to the action $(m\cdot \varphi)(\cdot) = m\varphi(\Ad(m)^{-1}\cdot )$ ($m \in M$, $\varphi \in \Hom_{\univ{g'}}(\univ{g}, V')$).
\end{definition}

It is easy to see that $\pro{\lie{g}, M}{\lie{g'}, M}(V')$ is a $(\lie{g}, M)$-module.
The $\lie{g}$-action is induced from the right $\lie{g}$-action on $\univ{g}$.
The $(\lie{g}, M)$-module $\pro{\lie{g}, M}{\lie{g'}, M}(V')$ is called the production of $V'$.
See \cite[II.1]{KnVo95_cohomological_induction} for the production and the induction below.

\begin{definition}
	For a $\lie{g'}$-module $V'$, we set
	\begin{align*}
		\ind{g}{g'}(V') := \univ{g}\otimes_{\univ{g'}} V'.
	\end{align*}
\end{definition}

$\ind{g}{g'}(V')$ is a $\lie{g}$-module via the left action on $\univ{g}$.
The $\lie{g}$-module is called the induction of $V'$.
If $V'$ is a $(\lie{g'}, M)$-module, then $\ind{g}{g'}(V')$ admits a natural $(\lie{g}, M)$-module structure.
We use the same notation for the $(\lie{g}, M)$-module.

In many places, we regard $\ind{g}{g'}(V')$ as a right $\univ{g}$-module using the transpose ${}^t(\cdot)$ on $\univ{g}$.
Note that ${}^t(\cdot)$ is the antiautomorphism of $\univ{g}$ such that ${}^tX = -X$ for any $X \in \lie{g}$.
In the case, $\ind{g}{g'}(V')$ is naturally isomorphic to $V'\otimes_{\univ{g'}} \univ{g}$, where $V'$ is regarded as a right $\univ{g'}$-module using the transpose.

We recall the Zuckerman derived functors following \cite{KnVo95_cohomological_induction}.
Let $(\lie{g}, K)$ be a pair and $(\lie{g}, M)$ a subpair.
Assume that $K$ and $M$ are reductive.

Let $(L, \rring{K})$ (resp.\ $(R, \rring{K})$) denote the left (resp.\ right) regular representation
on the coordinate ring $\rring{K}$ on $K$.
For a $(\lie{g}, M)$-module $(\pi, V)$, we provide $\rring{K}\otimes V$ with a $\lie{g}$-action via
\begin{align*}
	\mu(X)(f \otimes v) = \sum_i f_i f \otimes \pi(X_i) v,
\end{align*}
where $f_i \in \rring{K}$ and $X_i \in \lie{g}$ with $\sum_i f_i(g)X_i = \Ad(g^{-1})X$.

\begin{definition}\label{def:Zuckerman}
	For a $(\lie{g}, M)$-module $(\pi, V)$, we set
	\begin{align*}
		\Dzuck{K}{M}{i}(V):=H^i(\lie{k}, M; \rring{K}\otimes V),
	\end{align*}
	where we apply the relative Lie algebra cohomology $H^i(\lie{k}, M;\cdot)$ with respect to
	the tensor product $R\otimes \pi$.
	Then the actions $\mu$ and $L$ on $\rring{K}\otimes V$ give a $(\lie{g}, K)$-module structure on $\Dzuck{K}{M}{i}(V)$.
	We call the functors $\Dzuck{K}{M}{i}\colon \Mod(\lie{g}, M)\rightarrow \Mod(\lie{g}, K)$ the Zuckerman derived functors.
\end{definition}

\subsection{Multiplicity}

A unitary representation of a reductive Lie group has a canonical irreducible decomposition by direct integral (see Subsection \ref{subsect:DirectIntegral}).
Then we can consider the (essential) supremum of the multiplicities.
Unlike unitary representations, there are several ways to define multiplicities for $(\lie{g}, K)$-modules and Casselman--Wallach representations.
We refer the reader to \cite{Ko15_vogan} for the difficulty.

For an abelian category $\calC$, we denote by $\Irr(\calC)$ the set of all equivalence classes of simple objects in $\calC$.
Our main concern in this paper is the supremum of multiplicities.
We prepare the following notation.

\begin{definition}\label{def:supmul}
	Let $\calC$ be an abelian category.
	For a functor $F\colon \calC \rightarrow \Mod(\CC)$, we set
	\begin{align*}
		\SupDim(F(\calC)) := \sup\set{\dim_{\CC}(F(V)): V \in \Irr(\calC)}.
	\end{align*}
\end{definition}

Let $G_\RR$ be a reductive Lie group with a Cartan involution $\theta$ and $G'_\RR$ a $\theta$-stable reductive subgroup of $G_\RR$.
Set $K_\RR:=G_\RR^\theta$ and $K'_\RR := (G'_\RR)^\theta$.
For a Casselman--Wallach representation $V$ of $G_\RR$, we can consider the following supremums of multiplicities:
\begin{align*}
	&\SupDim(\Hom_{G'_\RR}(V, \calC_{G'_\RR})), \\
	&\SupDim(\Hom_{G'_\RR}(\calC_{G'_\RR}, V)), \\
	&\SupDim(\Hom_{\lie{g'}, K'}(V_K, \calC_{\lie{g'}, K'})), \\
	&\SupDim(\Hom_{\lie{g'}, K'}(\calC_{\lie{g'}, K'}, V_K)).
\end{align*}
Although these are similar quantities, their relationships to each other are non-trivial.
Since $V_K$ is dense in $V$, we have
\begin{align}
	\dim_{\CC}(\Hom_{G_\RR}(V, V'))\leq \dim_{\CC}(\Hom_{\lie{g}, K}(V_K, V'_{K'}))
	\label{eqn:CWandGKhom}
\end{align}
for any $V' \in \calC_{G'_\RR}$
It is still open when the equality holds.
Remark that all of the supremums do not change if we replace $\calC_{\lie{g'}, K'}$ (resp.\ $\calC_{G'_\RR}$) with $\Mod_{fl}(\lie{g}, K)$ (resp.\ $\CWCat{G'_\RR}$).
In fact, we have $\Irr(\calC_{\lie{g}, K}) = \Irr(\Mod_{fl}(\lie{g}, K))$ by definition.

Note that we can prove
\begin{align*}
	\SupDim(\Hom_{G'_\RR}(\calC_{G'_\RR}, V)) = \SupDim(\Hom_{\lie{g'}, K'}(\calC_{\lie{g'}, K'}, V_K)).
\end{align*}
The equation is strongly related to the theory of discretely decomposable branching laws (see \cite[Problem 5.15]{Ko15_vogan}).
We do not treat this topic in this paper and we postpone the proof of the equation.

For a Casselman--Wallach representation $V'$ of $G'_\RR$, we can consider the induced representation version of the supremums:
\begin{align*}
	\SupDim(\Hom_{G'_\RR}(\calC_{G_\RR}, V')) &= \SupDim(\Hom_{G_\RR}(\calC_{G_\RR}, \Ind^{G_\RR}_{G'_\RR}(V'))), \\
	\SupDim(\Hom_{\lie{g'}, K'}(\calC_{\lie{g}, K}, V'_{K'})) &= \SupDim(\Hom_{\lie{g}, K'}(\calC_{\lie{g}, K}, \pro{\lie{g}, K'}{\lie{g'}, K'}(V'_{K'}))).
\end{align*}
By \eqref{eqn:CWandGKhom}, we have
\begin{align*}
	\SupDim(\Hom_{G'_\RR}(\calC_{G_\RR}, V')) \leq \SupDim(\Hom_{\lie{g'}, K'}(\calC_{\lie{g}, K}, V'_{K'})).
\end{align*}
The finiteness of the above quantities are the main targets of our study.

\subsection{Two-sided ideal of \texorpdfstring{$\univ{g}$}{U(g)}}

Let $\lie{g}$ be a complex reductive Lie algebra.
We will use the following result by M. Duflo in many places to reduce results about $(\lie{g}, K)$-modules to ones about highest weight modules.

\begin{fact}[{\cite{Du77_primitive_ideal}}]\label{fact:DufloIdeal}
	Let $\lie{b}$ be a Borel subalgebra of $\lie{g}$.
	For any primitive ideal $I$ of $\univ{g}$, there exists some irreducible highest weight module $V$ of $\lie{g}$ with respect to $\lie{b}$ such that $\Ann_{\univ{g}}(V) = I$.
\end{fact}

For non-primitive $I$, we can use the following lemma.

\begin{lemma}\label{lem:TwosidedPrimitive}
	Let $I$ be a proper two-sided ideal of $\univ{g}$ such that $I \cap \univcent{g}$
	has finite codimension in $\univcent{g}$.
	Then there exist some finitely many primitive ideals $J_1, J_2, \ldots, J_n \subset \univ{g}$ containing $I$ such that $J_1 J_2\cdots J_n \subset I$.
	Moreover, $n$ is bounded by a constant depending only on the codimension of $I\cap \univcent{g}$ in $\univcent{g}$.
\end{lemma}

\begin{proof}
	Since $I\cap \univcent{g}$ has finite codimension in $\univcent{g}$,
	the finitely generated $(\lie{g\oplus g}, \Delta(\lie{g}))$-module $\univ{g}/I$ has finite length.
	Let $0 = V_0 \subset V_1 \subset \cdots \subset V_n = \univ{g}/I$ be a composition series of the $(\lie{g\oplus g})$-module $\univ{g}/I$
	and put $J_i := \Ann_{\univ{g}}(V_i / V_{i - 1})$ for $1\leq i \leq n$.
	Here we take the annihilators with respect to the actions of $\lie{g}\oplus 0 \subset \lie{g\oplus g}$.
	Then we have $J_1J_2\cdots J_n \subset I$.

	We shall show that $J_i$ is primitive.
	Since the $(\lie{g\oplus g})$-module $V_i/V_{i-1}$ is irreducible,
	$\Ann_{\univ{g\oplus g}}(V_i/V_{i-1})$ is primitive.
	By Fact \ref{fact:DufloIdeal}, there exist some irreducible highest weight modules $H_1$ and $H_2$ of $\lie{g}$ such that
	\begin{align*}
		\Ann_{\univ{g\oplus g}}(V_i / V_{i - 1}) = \Ann_{\univ{g\oplus g}}(H_1\boxtimes H_2).
	\end{align*}
	Hence we have $J_i = \Ann_{\univ{g\oplus g}}(V_i / V_{i - 1}) \cap (\univ{g}\otimes 1) = \Ann_{\univ{g}}(H_1)$.
	This shows that $J_i$ is primitive.

	Since the number of two-sided ideals of $\univ{g'}$ with a fixed infinitesimal character is bounded by a constant independent of the infinitesimal character \cite[Proposition 7.7]{Ki20}, $n$ is bounded by a constant depending only on the codimension of $I \cap \univcent{g}$ in $\univcent{g}$.
\end{proof}

Let $G$ be a reductive algebraic group and $G'$ a reductive subgroup of $G$.
For a $\lie{g}$-module $V$ and a two-sided ideal $I$ of $\univ{g'}$, we can consider the two $\univ{g}^{G'}$-modules
$V^I$ and $V/IV$.
More precisely, $\univ{g}^{G'}/\univ{g}I\cap \univ{g}^{G'}$ acts on $V^I$ and $\univ{g}^{G'}/I\univ{g}\cap \univ{g}^{G'}$ on $V/IV$.
We shall show that the two algebras are equal.
Note that the two algebras may not coincide if $G'$ is not reductive.

If $I = \lie{g'}\univ{g'}$, the equation $\univ{g}I\cap \univ{g}^{G'} = I\univ{g}\cap \univ{g}^{G'}$ can be interpreted
as follows.
$V/\lie{g'}V$ and $V^{\lie{g'}}$ can be realized as the Lie algebra cohomology and homology:
\begin{align*}
	V/\lie{g'}V \simeq H_0(\lie{g'}; V)\simeq H^n(\lie{g'}; V), \quad V^{\lie{g'}} \simeq H^0(\lie{g'}; V) \simeq H_n(\lie{g'}; V),
\end{align*}
where $n = \dim_{\CC}(\lie{g'})$.
The homology groups $H_i(\lie{g'}; V)$ (resp.\ $H^i(\lie{g'}; V)$) can be computed by taking a flat resolution (resp.\ injective resolution) of $V$
as a $\univ{g}$-module since $\univ{g}$ is a free left/right $\univ{g'}$-module.
Hence $(\univ{g}/\lie{g'}\univ{g})^{G'}$ acts naturally on $H_i(\lie{g'}; V)$, and $(\univ{g}/\univ{g}\lie{g'})^{G'}$ on $H^i(\lie{g'}; V)$.
In particular, the two algebras acts naturally on both of $V/\lie{g'}V$ and $V^{\lie{g'}}$.
This is the reason why the two algebras should coincide.

\begin{lemma}\label{lem:CommuteTwoSidedIdeal}
	Let $I$ be a two-sided ideal of $\univ{g'}$.
	Then we have $I\univ{g}\cap \univ{g}^{G'} = \univ{g}I \cap \univ{g}^{G'}$.
	In particular, we have
	\begin{align*}
		\PIdeg(\univ{g}^{G'}/I\univ{g}\cap \univ{g}^{G'}) = \PIdeg(\univ{g}^{G'}/{}^t I\univ{g}\cap \univ{g}^{G'}).
	\end{align*}
\end{lemma}

\begin{proof}
	Clearly, it is enough to show the assertion for connected $G'$.
	First we consider the case of $I = \lie{g'}\univ{g'}$.
	Since $G'$ is reductive, the multiplication map $\lie{g'}\otimes \univ{g} \rightarrow \lie{g'}\univ{g}$
	induces a surjection $(\lie{g'}\otimes \univ{g})^{G'} \twoheadrightarrow (\lie{g'}\univ{g})^{G'}$.

	Let $\set{X_i}_i$ be a basis of $\lie{g'}$, $\set{\lambda_i}_i$ its dual basis
	and $\sum_i X_i \otimes Y_i \in (\lie{g'}\otimes \univ{g})^{G'}$.
	For any $Z \in \lie{g'}$, we have
	\begin{align*}
		\sum_i X_i \otimes [Z, Y_i] = -\sum_i [Z, X_i]\otimes Y_i = \sum_{i,j} \lambda_j([X_i, Z]) X_j\otimes Y_i.
	\end{align*}
	This implies $[Z, Y_i] = \sum_j \lambda_i([X_j, Z]) Y_j$.
	Hence we have
	\begin{align*}
		\sum_i X_i Y_i &= \sum_i Y_i X_i + \sum_i [X_i, Y_i] \\
		&= \sum_i Y_i X_i + \sum_{i, j} \lambda_i([X_j, X_i]) Y_j \\
		&= \sum_i Y_i X_i + \sum_{j} \tr(\ad(X_j)) Y_j = \sum_i Y_i X_i.
	\end{align*}
	This shows the assertion for $I = \lie{g'}\univ{g'}$.

	We shall show
	\begin{align*}
		\sum_i X_i Y_i = \sum_i Y_i X_i
	\end{align*}
	for any $X_i \in \univ{g}$ and $Y_i \in \univ{g'}$ with $\sum_i X_i Y_i \in \univ{g}^{G'}$.
	It is clear that the assertion follows from this claim.

	Let $\Delta\colon G'\rightarrow G\times G'$ be the diagonal embedding.
	We consider the following two isomorphisms:
	\begin{align*}
		\varphi_1 &\colon (\univ{g\oplus g'}/\univ{g\oplus g'}\Delta(\lie{g'}))^{\Delta(G')} \rightarrow \univ{g}^{G'}, \\
		\varphi_2 &\colon (\univ{g\oplus g'}/\Delta(\lie{g'})\univ{g\oplus g'})^{\Delta(G')} \rightarrow \univ{g}^{G'}
	\end{align*}
	defined by $\varphi_1(X\otimes Y) = X{}^t Y$ and $\varphi_2(X\otimes Y) = {}^t Y X$.
	Here we identify $\univ{g\oplus g'}$ with $\univ{g}\otimes \univ{g'}$.
	As we have shown for $I = \lie{g'}\univ{g'}$, we have
	\begin{align*}
		\univ{g\oplus g'}\Delta(\lie{g'}) \cap \univ{g\oplus g'}^{\Delta(G')} = \Delta(\lie{g'})\univ{g\oplus g'} \cap \univ{g\oplus g'}^{\Delta(G')}.
	\end{align*}
	Hence $\varphi_1$ and $\varphi_2$ have the same domain.
	The inverses of $\varphi_1$ and $\varphi_2$ are given by the same map $\psi(X) = X\otimes 1$ ($X \in \univ{g}^{G'}$).
	This implies $\varphi_1 = \varphi_2$.
	Therefore we have shown the claim.
\end{proof}

\section{Uniformly bounded family and applications}

We have introduced the notion of uniformly bounded families of $\lie{g}$-modules
in \cite{Ki20}.
We recall the notion and related results.
As applications of uniformly bounded families, we will give an upper bound of multiplicities by the PI degree
and prove $\Loewy{\calC_{\lie{g}, K}} < \infty$.

\subsection{Uniformly bounded family}

Let $\lie{g}$ be a complex reductive Lie algebra.
A uniformly bounded family of $\lie{g}$-modules is a family of $\lie{g}$-modules with bounded lengths whose Beilinson--Bernstein localization is a uniformly bounded family of twisted $\ntDsheaf$-modules on the full flag variety of $\lie{g}$.
For a pair $(\lie{g}, K)$, a family of $(\lie{g}, K)$-modules is said to be uniformly bounded if the family is uniformly bounded as a family of $\lie{g}$-modules.
We omit the details because we do not use the definition itself of uniformly bounded families in this paper.
See \cite[Definition 7.2]{Ki20} for the definition.

It is important that the uniform boundedness is preserved by several operations of $\lie{g}$-modules.
We shall enumerate such results.

\begin{fact}[{\cite[Proposition 7.3]{Ki20}}]
	Let $\lie{g}$ and $\lie{h}$ be complex reductive Lie algebras.
	\begin{enumparen}
		\item For a short exact sequence $0\rightarrow L \rightarrow M \rightarrow N \rightarrow 0$ in $\prod_{i \in I} \Mod(\lie{g})$, both $L$ and $N$ are uniformly bounded if and only if so is $M$.
		\item For a family $(V_i)_{i\in I}$ (resp, $(W_i)_{i\in I}$) of $\lie{g}$-modules (resp.\ $\lie{h}$-modules),
		$(V_i\boxtimes W_i)_{i \in I}$ is a uniformly bounded family of $(\lie{g}\oplus \lie{h})$-modules if and only if $(V_i)_{i\in I}$ and $(W_i)_{i\in I}$ are uniformly bounded.
	\end{enumparen}
\end{fact}

Let $(\lie{g}, K)$ be a pair and $\lie{p}$ a parabolic subalgebra of $\lie{g}$ with a Levi subalgebra $\lie{l}$.
Let $(\lie{l}, K_L)$ be a subpair of $(\lie{g}, K)$ such that $K_L$ normalizes $\lie{p}$.
We consider $(\lie{l}, K_L)$-modules as $(\lie{p}, K_L)$-modules through the natural surjection $\lie{p}\rightarrow \lie{l}$.

\begin{fact}[{\cite[Theorem 7.11]{Ki20}}]\label{fact:BoundedGeneralizedVerma}
	Suppose that $K$ and $K_L$ are reductive.
    Let $(V_i)_{i \in I}$ be a uniformly bounded family of $(\lie{l}, K_L)$-modules.
	Then the family $(\Dzuck{K}{K_L}{j}(\univ{g}\otimes_{\univ{p}}V_i))_{i\in I, j\in \ZZ}$ of $(\lie{g}, K)$-modules is uniformly bounded.
	In particular, $(\univ{g}\otimes_{\univ{p}}V_i)_{i\in I}$ is uniformly bounded.
\end{fact}

Typical and important examples of uniformly bounded families are families of Harish-Chandra modules and families of objects in the BGG category $\calO$.

\begin{fact}[{\cite[Proposition 7.6]{Ki20}}]\label{fact:UniformlyBoundedFamilyGK}
	Suppose that $K$ has finitely many orbits in the full flag variety of $\lie{g}$.
	Then any family of $(\lie{g}, \lie{k})$-modules with bounded lengths is uniformly bounded.
	Here, a $(\lie{g}, \lie{k})$-module means a locally $\lie{k}$-finite $\lie{g}$-module.
\end{fact}

\begin{fact}[{\cite[Propositions 7.3 and 7.6]{Ki20}}] \label{fact:UniformlyBoundedO}
	Let $\set{V_i}_{i \in I}$ be a family of objects in $\calO_{\lie{g}}$ with bounded lengths.
	Then $\set{V_i}_{i\in I}$ is uniformly bounded.
\end{fact}

Let $G'$ be a reductive algebraic group such that $(\lie{g}, G')$ is a pair.
As we have seen in Definition \ref{def:Zuckerman}, the Zuckerman derived functors is defined by relative Lie algebra homologies.
Hence we can deduce the boundedness of the lengths of $\univ{g}^{G'}$-modules on relative Lie algebra homologies.


\begin{fact}[{\cite[Corollary 7.16]{Ki20}}] \label{fact:BoundedTorCenter}
	Let $(V_i)_{i \in I}$ be a uniformly bounded family of $\lie{g}$-modules.
	Then there exists some constant $C > 0$ such that
	for any maximal ideal $\calI\subset \univcent{g'}$, $i \in I$ and $j\in \ZZ$, 
	we have 
	\begin{align*}
		\Len_{\univ{g}^{G'}\otimes \univ{g'}} (\Tor_j^{\univcent{g'}}(\univcent{g'}/\calI, V_i)) \leq C.
	\end{align*}
\end{fact}

\begin{remark}
	Since $\univcent{g'}$ is isomorphic to a polynomial ring,
	we can compute $\Tor_j^{\univcent{g'}}(\univcent{g'}/\calI, V_i)$ by the Koszul complex.
	Hence we have
	\begin{align*}
		\Tor_{r}^{\univcent{g'}}(\univcent{g'}/\calI, V_i) \simeq V_i^{\calI},
	\end{align*}
	where $r$ is the rank of $\lie{g'}$.
\end{remark}

Let $K'$ be a reductive subgroup of $K$ such that $(\lie{g'}, K')$ is a subpair of $(\lie{g}, K)$.
Suppose that $\Ad_{\lie{g}}(K')$ is contained in $\Ad_{\lie{g}}(G')$.
The following fact is a key result to give an upper bound of multiplicities.
We will apply the fact to Harish-Chandra modules and objects in the BGG category $\calO$.

\begin{fact}[{\cite[Theorem 7.18]{Ki20}}]\label{fact:BoundedTorHom}
	Let $(V_i)_{i\in I}$ and $(V'_i)_{i\in I}$ be uniformly bounded families of $(\lie{g}, K)$-modules and $(\lie{g'}, K')$-modules, respectively.
	Then there exists some constant $C > 0$ such that
	\begin{align*}
		\Len_{\univ{g}^{G'}}(H_j(\lie{g'}, K'; V_i\otimes V'_i)) \leq C
	\end{align*}
	for any $i\in I$ and $j\in \ZZ$.
\end{fact}

\subsection{Upper bound of multiplicities and polynomial identity}

As an application of uniformly bounded families, we give an upper bound of multiplicities by the PI degree.
Let $G$ be a reductive algebraic group and $G'$ a reductive subgroup of $G$.

\begin{lemma}\label{lem:DimOfUnivG}
	Let $I$ be a two-sided ideal of $\univ{g}^{G'}$.
	Assume that $I \cap \univ{g}^{G}$ and $I \cap \univ{g'}^{G'}$
	are maximal ideals of $\univ{g}^G$ and $\univ{g'}^{G'}$, respectively.
	Then there exists some constant $C > 0$ independent of $I$ such that
	\begin{align*}
		\dim_{\CC}(\univ{g}^{G'}/I) \leq C \cdot \PIdeg(\univ{g}^{G'}/I)^2.
	\end{align*}
\end{lemma}

\begin{proof}
	By \cite[Theorem 7.29]{Ki20}, the length of $\univ{g}^{G'}/I$ as a $(\univ{g}^{G'}/I, \univ{g}^{G'}/I)$-bimodule is bounded by some constant $C$ independent of $I$.
	Although $G$ and $G'$ are assumed to be connected in the reference, it is easy to generalize the result to disconnected case.
	Therefore the assertion follows from Proposition \ref{prop:UpperBoundDimA}.
\end{proof}

The following two results give upper bounds of (cohomological) multiplicities in branching laws and induced representations by PI degree.

\begin{theorem}\label{thm:UpperBoundHomGeneral}
	Let $(V_i)_{i \in I}$ be a uniformly bounded family of $\lie{g}$-modules.
	Then there exists some constant $C > 0$ such that for any irreducible $\lie{g'}$-module $V'$, we have
	\begin{align*}
		\dim_{\CC}(\Hom_{\lie{g'}}(V_i, V')) \leq C \cdot \PIdeg(\univ{g}^{G'}/\calI(i, V')),
	\end{align*}
	where we set $\calI(i, V'):= (\Ann_{\univ{g}}(V_i) + \Ann_{\univ{g'}}(V')\univ{g}) \cap \univ{g}^{G'}$.
	We can replace $\Hom_{\lie{g'}}(V_i, V')$ with $\Hom_{\lie{g'}}(V', V_i)$.
\end{theorem}

\begin{proof}
	Put $\calA:=\univ{g}^{G'}/\calI(i, V')$.
	Clearly we can assume $\PIdeg(\calA) < \infty$.
	Then $\calA$ is finite dimensional by Lemma \ref{lem:DimOfUnivG}.
	By Fact \ref{fact:BoundedTorCenter}, we can take a constant $C > 0$ independent of $i\in I$ and $V'$ such that
	\begin{align*}
		\Len_{\calA\otimes \univ{g'}}(V_i / \Ann_{\univ{g'}}(V') V_i)
		\leq C.
	\end{align*}

	Let $M$ be an irreducible subquotient of $V_i / \Ann_{\univ{g'}}(V') V_i$.
	Since $\calA$ is finite dimensional, $M$ is isomorphic to an external tensor product of an irreducible $\calA$-module and an irreducible $\univ{g'}$-module.
	By Proposition \ref{prop:pidLowerbound}, the dimension of any irreducible $\calA$-module is less than or equal to $\PIdeg(\calA)$.
	Hence the length of $M$ as a $\lie{g'}$-module is less than or equal to $\PIdeg(\calA)$.
	This shows
	\begin{align*}
		\dim_{\CC}(\Hom_{\lie{g'}}(V_i, V'))
		=\dim_{\CC}(\Hom_{\lie{g'}}(V_i/\Ann_{\univ{g'}}(V')V_i, V'))
		\leq C\cdot \PIdeg(\calA).
	\end{align*}
	Replacing $V_i/\Ann_{\univ{g'}}(V') V_i$
	with $V_i^{\Ann_{\univ{g'}}(V')}$, we can show the second assertion.
\end{proof}

Let $(\lie{g}, K)$ be a pair and $(\lie{g'}, K')$ a subpair of $(\lie{g}, K)$.
Suppose that $K'$ is reductive and $\Ad_{\lie{g}}(K')$ is contained in $\Ad_{\lie{g}}(G')$.
The following result is an easy consequence of Proposition \ref{prop:pidLowerbound} and Fact \ref{fact:BoundedTorHom}.

\begin{lemma}\label{lem:UpperBoundTorGeneral}
	Let $(V_i)_{i\in I}$ and $(V'_j)_{j\in J}$ be uniformly bounded families of $(\lie{g}, K)$-modules and $(\lie{g'}, K')$-modules, respectively.
	Put $I_i:= \Ann_{\univ{g}}(V_i)$ and $I'_j:={}^t\Ann_{\univ{g'}}(V'_j)$.
	Then there exists some constant $C > 0$ such that
	\begin{align*}
		\dim_{\CC}(H_k(\lie{g'}, K'; V_i\otimes V'_j)) \leq C \cdot \PIdeg(\univ{g}^{G'}/(I_i + I'_j\univ{g})\cap \univ{g}^{G'})
	\end{align*}
	for any $i\in I$, $j\in J$ and $k\in \ZZ$.
\end{lemma}

\subsection{Loewy length of \texorpdfstring{$\lie{g}$}{g}-modules}\label{sect:Category}

In this subsection, we give an upper bound of the Loewy lengths of modules in $\calC_{\lie{g}, K}$ and $\calC_{G_\RR}$ as an application of uniformly bounded families.

Let $\lie{g}$ be a complex reductive Lie algebra.
Fix a Cartan subalgebra $\lie{t}$ of $\lie{g}$.
We denote by $\roots(\lie{g}, \lie{t})$ the set of all roots with respect to $\lie{t}$ and fix a set $\proots(\lie{g}, \lie{t})$ of positive roots.
We write $W_{\lie{g}}$ for the Weyl group of $\roots(\lie{g}, \lie{t})$.
Let $\calI_\chi$ denote the maximal ideal of $\univcent{g}$ corresponding to an infinitesimal character $\chi \in \lie{t}^*/W_{\lie{g}}$.

Following \cite[VII. 2]{KnVo95_cohomological_induction}, we review the primary component of a $\lie{g}$-module.
For a $\lie{g}$-module $V$ and $\chi \in \lie{t}^*/W_{\lie{g}}$, we put
\begin{align*}
\PR_{\chi}(V):= \set{v \in V: \calI_\chi^n v = 0 \text{ for some $n$}},
\end{align*}
which is called the $\chi$ primary component of $V$.
If $V$ is $\univcent{g}$-finite (e.g.\ $V$ has finite length), then $V$ is the direct sum of the primary components \cite[Proposition 7.20]{KnVo95_cohomological_induction}.
The following proposition is a consequence of Kostant's theorem \cite[Theorem 7.133]{KnVo95_cohomological_induction} and its proof.

\begin{proposition}\label{prop:generalizedInfCharTensorWithFin}
	Let $V$ be a $\lie{g}$-module with an infinitesimal character and $F$ a completely reducible finite-dimensional $\lie{g}$-module.
	Put $n := |W_{\lie{g}}|$.
	Then for any $\theta \in \lie{t}^*/W_{\lie{g}}$, we have
	\begin{align*}
		\calI_\theta^{n}\PR_\theta(V\otimes F) = 0.
	\end{align*}
\end{proposition}

\begin{lemma}\label{lem:BoundLoewy}
	Let $(V_i)_{i \in I}$ be a uniformly bounded family of $\lie{g}$-modules.
	There exists a constant $C > 0$ such that
	\begin{align*}
		\Loewy{V_i\otimes F} \leq C
	\end{align*}
	for any $F \in \Mod_{fd}(\lie{g})$ and $i \in I$.
\end{lemma}

\begin{proof}
	Let $G$ be a connected algebraic group with the Lie algebra $\lie{g}$, and $\chi \in \lie{t}^*/W_{\lie{g}}$.
	Fix $i \in I$ and $F \in \Mod_{fd}(\lie{g})$.
	By Fact \ref{fact:BoundedTorCenter}, there exists some constant $C > 0$ independent of $i \in I$, $F$ and $\chi$ such that
	\begin{align*}
		\Len_{\univ{g\oplus g}^{\Delta(G)}\otimes \univ{g}}((V_i\otimes F)/\calI_\chi(V_i\otimes F)) \leq C,
	\end{align*}
	where $\Delta\colon G\rightarrow G\times G$ is the diagonal embedding.
	Since $F$ is finite dimensional and $\univ{g}/\Ann_{\univ{g}}(V_i)$ is $G$-admissible, we have
	\begin{align*}
		&\dim_{\CC}((\univ{g\oplus g}/\Ann_{\univ{g\oplus g}}(V_i\otimes F))^{\Delta(G)})\\
		\leq &\dim_{\CC}(\Hom_G(\End_\CC(F)^*, \univ{g}/\Ann_{\univ{g}}(V_i))) < \infty.
	\end{align*}
	Hence any composition factor of the $\univ{g\oplus g}^{\Delta(G)}\otimes \univ{g}$-module $(V_i\otimes F)/\calI_\chi(V_i\otimes F)$ is completely reducible as a $\lie{g}$-module (cf.\ the proof of Theorem \ref{thm:UpperBoundHomGeneral}).
	This implies
	\begin{align*}
		\Loewy{(V_i\otimes F)/\calI_\chi(V_i\otimes F)} \leq C
	\end{align*}
	as a $\lie{g}$-module.

	For any $k > 0$, the multiplication map induces a surjective $\lie{g}$-homomorphism
	\begin{align*}
		\calI_\chi^k \otimes ((V_i\otimes F)/\calI_\chi (V_i\otimes F))
		\twoheadrightarrow \calI_\chi^k (V_i\otimes F) / \calI_\chi^{k+1} (V_i\otimes F).
	\end{align*}
	By Proposition \ref{prop:generalizedInfCharTensorWithFin}, we have
	$\calI_\chi^{|W_\lie{g}|} \PR_\chi(V_i\otimes F) = 0$.
	Therefore we obtain
	\begin{align*}
		\Loewy{\PR_\chi(V_i\otimes F)} \leq C\cdot |W_\lie{g}|.
	\end{align*}
	Since $V_i\otimes F$ is the direct sum of the primary components, we have
	\begin{align*}
		\Loewy{V_i\otimes F} \leq C\cdot |W_\lie{g}|.
	\end{align*}
	We have shown the lemma.
\end{proof}

By Fact \ref{fact:UniformlyBoundedFamilyGK} and Lemma \ref{lem:BoundLoewy}, we obtain the following theorem as we have mentioned in Subsection \ref{subsect:CategoriesGKCW}.

\begin{theorem}\label{thm:FiniteLoewyGK}
	We have $\Loewy{\calC_{\lie{g}, K}} = \Loewy{\calC_{G_\RR}} < \infty$.
\end{theorem}

Let $\lie{b}$ be a Borel subalgebra of $\lie{g}$.
We consider the BGG category $\calO_{\lie{b}}$ and the category $\calO_{\lie{g}}$ defined in Subsection \ref{subsect:CategoriesGKCW}.

\begin{proposition}\label{prop:FiniteLoewyO}
	We have $\Loewy{\calO_{\lie{b}}} = \Loewy{\calO_{\lie{g}}} < \infty$.
\end{proposition}

\begin{proof}
	By the definition of $\calO_{\lie{g}}$, it suffices to show $\Loewy{\calO_{\lie{b}}} < \infty$.
	Let $\calC$ be the full subcategory of $\calO_{\lie{b}}$ whose object is isomorphic to a subquotient of $F\otimes V$ for some $F \in \Mod_{fd}(\lie{g})$ and semisimple $V \in \calO_{\lie{b}}$.
	Then we have $\Loewy{\calC} < \infty$ by Fact \ref{fact:UniformlyBoundedFamilyGK} and Lemma \ref{lem:BoundLoewy}.
	We shall show that $\calC = \calO_{\lie{b}}$.

	Since any Verma module is isomorphic to a subquotient of a tensor product of an irreducible Verma module and an irreducible finite-dimensional $\lie{g}$-module, $\calC$ contains all Verma modules in $\calO_{\lie{b}}$.
	Any indecomposable projective object in $\calO_\lie{b}$ is isomorphic to a direct summand of a tensor product
	of a Verma module and an irreducible finite-dimensional $\lie{g}$-module by \cite[Theorems 3.8 and 3.9]{Hu08_category_o}.
	Hence any object in $\calO_\lie{b}$ is in $\calC$ since $\calO_{\lie{b}}$ has enough projectives.
\end{proof}

\section{Lower bound of multiplicities}\label{sect:LowerBound}

The main purpose in this section is to find a lower bound of the supremum of multiplicities in a general setting.
We will give concrete examples in Section \ref{sect:UniformlyBoundedTheorem}.

\subsection{\texorpdfstring{$(\lie{g}, K)$}{(g, K)}-module case}

We shall explain, using the case of $(\lie{g}, K)$-modules, how to obtain the lower bound.
Let $(\lie{g}, K)$ be a pair and $(\lie{g'}, K')$ a subpair of $(\lie{g}, K)$.
Suppose that $\lie{g}$ and $\lie{g'}$ are reductive and $K/K'$ is connected.
Here we denote by $\univ{g}^{G'}$ the algebra $\univ{g}^{\lie{g'}, K'}$ of all $(\lie{g'}, K')$-invariant elements in $\univ{g}$.
Let $V$ be an irreducible $(\lie{g}, K)$-module and set
\begin{align*}
	J := \bigcap_{V' \in \Irr(\calC_{\lie{g'},K'})}\Ann_{\univ{g}^{G'}}(\Hom_{\lie{g'}, K'}(V, V')).
\end{align*}
Then we have
\begin{align*}
	\PIdeg(\univ{g}^{G'}/J) \leq \SupDim(\Hom_{\lie{g'}, K'}(V, \calC_{\lie{g'},K'}))
\end{align*}
by Proposition \ref{prop:pidegUpperbound}.
We want to replace $J$ in the inequation with $\Ann_{\univ{g}^{G'}}(V)$.
Note that if $V$ is unitarizable, we can show $J = \Ann_{\univ{g}^{G'}}(V)$ (see Corollary \ref{cor:DirectIntegralAnnPIdeg}).
The author, however, has no proof and counterexample to the equality $J = \Ann_{\univ{g}^{G'}}(V)$ in general.

As we have seen in Theorem \ref{thm:FiniteLoewyGK}, we have $\Loewy{\calC_{\lie{g'}, K'}} < \infty$.
Hence $J^{\Loewy{\calC_{\lie{g'}, K'}}}$ annihilates $\Hom_{\lie{g'}, K'}(V, V')$ for any $V' \in \calC_{\lie{g'}, K'}$.
We have constructed $\calC_{\lie{g'}, K'}$ to be closed under tensoring with finite-dimensional modules.
This implies that
\begin{align*}
	W := \bigcap_{\substack{V' \in \calC_{\lie{g'}, K'}\\ f \in \Hom_{\lie{g'}, K'}(V, V')}} \Ker(f)
\end{align*}
is a $(\lie{g}, K')$-submodule of $V$.
Since $V$ is irreducible and $K/K'$ is connected, we have $W = 0$ or $V$.
Therefore, if $W = 0$, we have $J^{\Loewy{\calC_{\lie{g'}, K'}}}\subset \Ann_{\univ{g}^{G'}}(V)$ and hence
\begin{align*}
	\PIdeg(\univ{g}^{G'}/\Ann_{\univ{g}^{G'}}(V)) \leq \SupDim(\Hom_{\lie{g'}, K'}(V, \calC_{\lie{g'},K'}))
\end{align*}
by Proposition \ref{prop:pidegIdeal} \ref{item:pidegIdealNilpotent}.
The desired lower bound of $\SupDim(\Hom_{\lie{g'}, K'}(V, \calC_{\lie{g'},K'}))$ has been obtained.

The assumption $W = 0$ means that $V|_{\lie{g'}, K'}$ has at least one irreducible quotient.
We can prove this easily in the case of Harish-Chandra modules.
See Lemma \ref{lem:ExistenceQuotRest}.

The purpose in this section is to prove the above result in a general setting applicable to inductions and restrictions of $(\lie{g}, K)$-modules, Casselman--Wallach representations and objects in the BGG category $\calO$, and reducible cases.
There are some problems to generalize the above discussion.
For example, $W$ may not be $G_\RR$-stable if $V$ is a Casselman--Wallach representation.
We need to bridge the gaps.

\subsection{Categories with \texorpdfstring{$\Mod_{fd}(G)$}{Mod(G)}-action}\label{subsect:CategoryWithGaction}

Let $G_\RR$ be a reductive Lie group and $(\lie{g}, M)$ a pair.
Set $\calB:= \Mod(\lie{g}, M)$ or $\Mod(G_\RR)$.
We denote by $\calB_f$ the full subcategory of $\calB$ whose object is completely reducible and finite dimensional (i.e.\ $\calB_f = \Mod_{fd}(\lie{g}, M)$ or $\Mod_{fd}(G_\RR)$).

Let $\calC$ be an abelian category and $\calF^\calC \colon \calC\rightarrow \calB$ a faithful exact functor.
Suppose that there are a bifunctor $\otimes \colon \calB_f \times \calC \rightarrow \calC$ and a natural isomorphism $\calF^\calC(F \otimes V) \simeq F \otimes \calF^\calC(V)$ ($F \in \calB_f, V \in \calC$) of bifunctors.
We shall abuse notation and write $V$ for $\calF^\calC(V)$.
In many applications, $\calF^\calC$ is just an inclusion functor.
For example, we will apply results in this section to the following cases:
\begin{enumerate}
	\item $\calB = \Mod(\lie{g}, K)$, $\calC = \calC_{\lie{g}, K}$ ($K$ is the complexification of a maximal compact subgroup of $G_\RR$)
	\item $\calB = \Mod(G_\RR)$, $\calC = \calC_{G_\RR}$
	\item $\calB = \Mod(\lie{g})$, $\calC = \calO_{\lie{b}}$ (or $\calC = \calO_{\lie{g}}$).
\end{enumerate}

For simplicity, we say that a completely reducible locally finite $(\lie{g}, M)$-module (or $G_\RR$-representation) is a $G$-module
as if there is a reductive algebraic group $G$ such that $\Mod_{fl}(G) = \calB_f$.
Similarly, we use the terminology: $G$-invariant, $G$-submodule, $G$-linear and rational $G$-action.
In fact, if $\lie{g}$ is semisimple, there is a reductive algebraic group $G$ with the Lie algebra $\lie{g}$ such that $\Mod_{fl}(G)$ is canonically equivalent to $\calB_f$.

\begin{definition}\label{def:AB-module}
	Let $\calA$ be a $\CC$-algebra equipped with a rational $G$-action.
	We say that $\calA$ is a \define{$G$-algebra} if the multiplication map $\calA\otimes \calA \rightarrow \calA$ is $G$-linear.

	We say that an object $V$ in $\calB$ equipped with an $\calA$-action is an \define{$(\calA, \calB)$-module} if the multiplication map $F\otimes V \rightarrow V$
	is a morphism in $\calB$ for any finite-dimensional $G$-submodule $F$ of $\calA$.

	We denote by $\Mod(\calA, \calB)$ the category of $(\calA, \calB)$-modules.
	We define morphisms of $\Mod(\calA, \calB)$ in a natural way.	
\end{definition}

Let $\calA$ be a $G$-algebra.
It is easy to see that $\Mod(\calA, \calB)$ is abelian if $\calB = \Mod(\lie{g}, M)$.
If $\calB = \Mod(G_\RR)$, the category $\Mod(\calA, \calB)$ is, however, not abelian in general.
The following proposition is an easy consequence of the definition of $(\calA, \calB)$-modules.

\begin{proposition}
	For any $(\calA, \calB)$-module $V$, the annihilator $\Ann_{\calA}(V)$ is a $G$-submodule of $\calA$.
\end{proposition}

We consider the branching problem of $(\calA, \calB)$-modules.
We want to give a lower bound of $\SupDim(\Hom_{\calB}(V, \calC))$ or $\SupDim(\Hom_{\calB}(\calC, V))$ for $V \in \Mod(\calA, \calB)$ by the PI degree.
To do so, we need the following notation, which represents how many irreducible quotients/submodules an $(\calA, \calB)$-module has in $\calC$.

\begin{definition}\label{def:RadSocInC}
	Let $V$ be an object in $\calB$.
	We set
	\begin{align*}
		\rad(V, \calC) &:= \bigcap_{\substack{W \in \calC \\ f \in \Hom_{\calB}(V, W)}} \Ker(f), \\
		\soc(V, \calC) &:= \sum_{\substack{W \in \calC \\ f \in \Hom_{\calB}(W, V)}} \Im(f).
	\end{align*}
	When we want to specify the functor $\calF^\calC$, we write $\rad(V, \calF^\calC(\calC)) = \rad(V,\calC)$ and $\soc(V, \calF^\calC(\calC)) = \soc(V, \calC)$.
\end{definition}

\begin{proposition}\label{prop:FundamentalRadSoc}
	Let $f\colon V\rightarrow W$ be a morphism of objects in $\calB$
	and $F$ a finite-dimensional $G$-module.
	Then we have
	\begin{align*}
		f(\rad(V,\calC)) &\subset \rad(W,\calC), \\
		f(\soc(V,\calC)) &\subset \soc(W,\calC), \\
		\rad(F\otimes V,\calC) &= F\otimes \rad(V,\calC), \\
		\soc(F\otimes V,\calC) &= F\otimes \soc(V,\calC).
	\end{align*}
\end{proposition}

\begin{proof}
	The first two inclusions are clear by definition.
	Recall that $\calC$ is closed under tensoring with finite-dimensional $G$-modules.
	The last two equations follow from the natural isomorphisms
	\begin{align*}
		\Hom_{\calB}(F\otimes V, M) \simeq \Hom_{\calB}(V, F^* \otimes M), \\
		\Hom_{\calB}(M, F\otimes V) \simeq \Hom_{\calB}(F^* \otimes M, V).
	\end{align*}
	for $M \in \calC$.
\end{proof}

\begin{proposition}\label{prop:RadSoc}
	Let $V$ be an $(\calA, \calB)$-module.
	Then $\rad(V,\calC)$ and $\soc(V,\calC)$ are $(\calA,\calB)$-submodules of $V$.
\end{proposition}

\begin{proof}
	Let $F$ be a finite-dimensional $G$-submodule of $\calA$.
	Applying Proposition \ref{prop:FundamentalRadSoc} to the multiplication map $F\otimes V\rightarrow V$, we have $F\cdot \rad(V, \calC) \subset \rad(V, \calC)$ and $F\cdot \soc(V,\calC) \subset \soc(V,\calC)$.
	This shows the proposition.
\end{proof}

\begin{lemma}\label{lem:ABmodChainRadSoc}
	Let $V$ be an $(\calA, \calB)$-module.
	Set $I:=\Ann_{\calA}(V/\rad(V,\calC))$ and $J:=\Ann_\calA(\soc(V,\calC))$.
	Then we have
	\begin{align*}
		\overline{I^{k+1}V}&\subset \rad(\overline{I^k V}, \calC), \\
		\soc(V/V^{J^k},\calC) &\subset V^{J^{k+1}}/V^{J^k},
	\end{align*}
	where if $\calB = \Mod(\lie{g}, M)$, $V$ is regarded as a topological space with the discrete topology (i.e.\ $\overline{I^k V} = I^k V$).
\end{lemma}

\begin{proof}
	By Proposition \ref{prop:FundamentalRadSoc}, we have
	\begin{align*}
		I^{k+1} V \subset I^k \rad(V,\calC) \subset \rad(\overline{I^k V}, \calC).
	\end{align*}
	Hence we obtain $\overline{I^{k+1}V}\subset \rad(\overline{I^k V}, \calC)$.

	We write $W$ for the inverse image of $\soc(V/V^{J^k}, \calC)$ by the natural projection $V\rightarrow V/V^{J^k}$.
	Since the multiplication map $J^k\otimes V \rightarrow V$ factors through $J^k\otimes V/V^{J^k}$, we have
	\begin{align*}
		J^k\cdot W \subset \soc(V, \calC) \subset V^J
	\end{align*}
	by Proposition \ref{prop:FundamentalRadSoc}.
	This implies $J^{k+1} W = 0$ and hence $\soc(V/V^{J^k}, \calC) \subset V^{J^{k+1}}/V^{J^k}$.
\end{proof}

\subsection{Lower bound of multiplicities and polynomial identity}
\label{subsect:LowerBound}

Retain the notation in the previous subsection.
Let $\widetilde{\calC}$ be an abelian category and $\calF=\calF^{\widetilde{\calC}}\colon \widetilde{\calC}\rightarrow \Mod(\calA, \calB)$ a faithful exact functor.
Assume that any object in $\widetilde{\calC}$ has finite length
and assume $\Loewy{\calC} < \infty$.
See Definition \ref{def:Loewy}.

When we apply results in this subsection to the restriction case, $\calF$ is taken to be a forgetful functor.
For example, we will deal with the following settings:
\begin{enumerate}
	\item $(\calA, \calB, \calC, \widetilde{\calC}) = (\univ{\widetilde{g}}, \Mod(\lie{g}, K), \calC_{\lie{g}, K},  \Mod_{fl}(\lie{\widetilde{g}}, \widetilde{K}))$
	\item $(\calA, \calB, \calC, \widetilde{\calC}) = (\univ{\widetilde{g}}, \Mod(G_\RR), \calC_{G_\RR}, \CWCat{\widetilde{G}_\RR})$
	\item $(\calA, \calB, \calC, \widetilde{\calC}) = (\univ{\widetilde{g}}, \Mod(\lie{g}), \calO_{\lie{g}}, \calO_{\lie{\widetilde{g}}})$.
\end{enumerate}
Here $\widetilde{G}_\RR$ is a reductive Lie group containing $G_\RR$ as a reductive subgroup and we do the same for $(\lie{\widetilde{g}}, \widetilde{K})$.
For the induction case, $\calF$ is taken to be the induction functors $\Ind^{\widetilde{G}_\RR}_{G_\RR}(\cdot)$ and $\ind{\widetilde{g}}{g}(\cdot)$.
In the case, $\calA$ is an algebra of (algebraic) differential operators.
See Section \ref{sect:UniformlyBoundedTheorem} for the details.

In the results here, we assume complicated conditions.
Some of the conditions are trivial when we apply them to the branching problem and harmonic analysis.

\begin{lemma}\label{lem:AnnhilatorRad}
	Let $V$ be a non-zero object in $\widetilde{\calC}$ and set 
	\begin{align*}
		J:=\Ann_{\calA}(\calF(V)/\rad(\calF(V),\calC)).
	\end{align*}
	Assume that
	\begin{enumparen}
		\item\label{enum:AnnihilatorRadNonVanish} $\rad(\calF(W), \calC) \neq \calF(W)$ for any non-zero subobject $W$ of $V$
		\item\label{enum:AnnihilatorRadExist} for any $k\geq 0$, there exists a subobject $V^k \subset V$ such that $\calF(V^k) = \overline{J^k \calF(V)}$.
	\end{enumparen}
	Then we have $J^{\Loewy{V}} \subset \Ann_{\calA}(\calF(V))$.
\end{lemma}

\begin{proof}
	We claim that $J$ annihilates $\calF(W)$ for any composition factor $W$ of $V$.
	By the assumptions \ref{enum:AnnihilatorRadNonVanish} and \ref{enum:AnnihilatorRadExist} and Lemma \ref{lem:ABmodChainRadSoc}, we have
	\begin{align*}
		\calF(V^{k+1}) = \overline{J^{k+1} \calF(V)} \subset \rad(\overline{J^k \calF(V)}, \calC) \subsetneq \overline{J^k \calF(V)} = \calF(V^k)
	\end{align*}
	if $V^{k}\neq 0$.
	This implies $\bigcap_{k} V^k = 0$ since $V$ has finite length.
	Hence for any composition factor $W$ of $V$, we can take $k \in \NN$ such that $W$ is isomorphic to a subquotient of $V^{k}/V^{k+1}$.
	Since $\calF(V^k)/\calF(V^{k+1}) = \overline{J^k V}/\overline{J^{k+1}V}$ is annihilated by $J$, so does $\calF(W)$.
	
	By the claim, for any subobject $W$ of $V$, we have $J\calF(W) \subset \calF(\rad(W))$.
	This implies that $J^{\Loewy{V}} \subset \Ann_{\calA}(\calF(V))$.
\end{proof}

\begin{lemma}\label{lem:lowerboundQuot}
	Let $V$ be a non-zero object in $\widetilde{\calC}$ and set $I:=\Ann_\calA(\calF(V))$.
	Assume that $V$ satisfies the assumptions of Lemma \ref{lem:AnnhilatorRad}.
	Then we have
	\begin{align*}
		\PIdeg((\calA/I)^{G}) \leq \SupDim(\Hom_{\calB}(\calF(V), \calC)).
	\end{align*}
\end{lemma}

\begin{proof}
	We set
	\begin{align*}
		J := \bigcap_{V' \in \Irr(\calC)} \Ann_{\calA^{G}}(\Hom_{\calB}(\calF(V), V')).
	\end{align*}
	Then $J^{\Loewy{\calC}}$ annihilates $\Hom_{\calB}(\calF(V), V')$ for any $V' \in \calC$ since $\Loewy{V'} \leq \Loewy{\calC} < \infty$.
	In other words, we have
	\begin{align*}
		J^{\Loewy{\calC}} \subset \Ann_{\calA}(\calF(V)/\rad(\calF(V),\calC)).
	\end{align*}

	By Lemma \ref{lem:AnnhilatorRad}, we have $J^{\Loewy{\calC}\Loewy{V}} \subset I \cap \calA^G$.
	Hence we obtain
	\begin{align*}
		\PIdeg((\calA/I)^{G}) = \PIdeg(\calA^G/J) \leq \SupDim(\Hom_{\calB}(\calF(V), \calC))
	\end{align*}
	by Propositions \ref{prop:pidegUpperbound} and \ref{prop:pidegIdeal} \ref{item:pidegIdealNilpotent}.
	Note that $(\calA/I)^{G}$ is isomorphic to $\calA^{G}/I^{G}$ since $\calA$ is completely reducible as a $G$-module.
\end{proof}

\begin{lemma}\label{lem:AnnhilatorSoc}
	Let $V$ be a non-zero object in $\widetilde{\calC}$ and set $J:=\Ann_{\calA}(\soc(\calF(V),\calC))$.
	Assume that
	\begin{enumparen}
		\item\label{enum:AnnihilatorSocNonVanish} $\soc(\calF(W), \calC) \neq 0$ for any non-zero quotient $W$ of $V$
		\item\label{enum:AnnihilatorSocExist} for any $k\geq 0$, there exists a subobject $V_k \subset V$ such that $\calF(V_k) = \calF(V)^{J^k}$.
	\end{enumparen}
	Then we have $J^{\Loewy{V}} \subset \Ann_{\calA}(\calF(V))$.
\end{lemma}

\begin{proof}
	We claim that $J$ annihilates $\calF(W)$ for any composition factor $W$ of $V$.
	By the assumptions \ref{enum:AnnihilatorSocNonVanish} and \ref{enum:AnnihilatorSocExist} and Lemma \ref{lem:ABmodChainRadSoc}, we have
	\begin{align*}
		0\neq \soc(\calF(V)/\calF(V_k), \calC) \subset \calF(V_{k+1})/\calF(V_k)
	\end{align*}
	if $V_{k}\neq V$.
	This implies $\bigcup_{k} V_k = V$ since $V$ has finite length.
	Hence for any composition factor $W$ of $V$, we can take $k \in \NN$ such that $W$ is isomorphic to a subquotient of $V_{k+1}/V_{k}$.
	Since $\calF(V_{k+1})/\calF(V_{k}) = \calF(V)^{J^{k+1}}/\calF(V)^{J^k}$ is annihilated by $J$, so does $\calF(W)$.
	
	By the claim, for any subobject $W$ of $V$, we have $\calF(\soc(W))\subset \calF(W)^J$.
	This implies that $J^{\Loewy{V}} \subset \Ann_{\calA}(\calF(V))$.
\end{proof}

\begin{lemma}\label{lem:lowerboundSub}
	Let $V$ be a non-zero object in $\widetilde{\calC}$ and set $I:=\Ann_{\calA}(\calF(V))$.
	Assume that $V$ satisfies the assumptions of Lemma \ref{lem:AnnhilatorSoc}.
	Then we have
	\begin{align*}
	\PIdeg((\calA/I)^{G}) \leq \SupDim(\Hom_{\calB}(\calC, \calF(V))).
	\end{align*}
\end{lemma}

\begin{proof}
	We can show the lemma in a similar way as Lemma \ref{lem:lowerboundQuot}.
	We omit the details.
\end{proof}

We consider tensor products over $\univ{g}$ instead of $\Hom_{\calB}(\cdot, \cdot)$.
Suppose $\calB=\Mod(\lie{g}, M)$.
For an object $V \in \calB$, let $\dual{V}$ denote the subspace of all $M$-finite vectors in $V^*$.
Then $\dual{(\cdot)}\circ \calF^\calC \colon \calC \rightarrow \calB$ is a contravariant faithful exact functor.

\begin{lemma}\label{lem:lowerboundTensor}
	Let $V$ be a non-zero object in $\widetilde{\calC}$ and set $I := \Ann_{\calA}(\calF(V))$.
	Assume that
	\begin{enumparen}
		\item for any non-zero subobject $W$ of $V$, there exists an object $W' \in \calC$ such that $(\calF(W)\otimes_{\univ{g}} W')^M \neq 0$
		\item for any $k \geq 0$, there exists a subobject $V^k \subset V$ such that $\calF(V^k) = J^k\calF(V)$, where $J = \Ann_\calA(\calF(V)/\rad(\calF(V), \dual{\calF^\calC(\calC)}))$.
	\end{enumparen}
	Then we have
	\begin{align*}
		\PIdeg((\calA/I)^{G}) \leq \SupDim((\calF(V)\otimes_{\univ{g}} \calC)^M).
	\end{align*}
\end{lemma}

\begin{proof}
	Since $((\calF(V)\otimes_{\univ{g}} V')^M)^*$ is naturally isomorphic to $\Hom_{\calB}(\calF(V), \dual{V'})$ for any $V' \in \calC$, we have
	\begin{align*}
		\SupDim((\calF(V)\otimes_{\univ{g}} \calC)^M) = \SupDim(\Hom_\calB(\calF(V), \dual{\calC})).
	\end{align*}
	Hence the assertion follows from Lemma \ref{lem:lowerboundQuot}.
\end{proof}

\section{Induced representations}

To estimate the supremum of multiplicities in an induced representation,
we need some fundamental results about $\rring{G/G'}\otimes \univ{g}$-modules.

\subsection{\texorpdfstring{$\calO(G/G')\otimes \univ{g}$}{O(G/G')*U(g)}-module}

We want to consider multiplicities in $\Ind^{G_\RR}_{G'_\RR}(V')$, $\pro{\lie{g}, M}{\lie{g'}, M}(V')$ and $\ind{g}{g'}(V')$.
We shall reduce this problem to the setting in Section \ref{sect:LowerBound}.
If $G_\RR/G'_\RR$ is a real form of a homogeneous space $G/G'$ of algebraic groups, the induced representations $\Ind^{G_\RR}_{G'_\RR}(V')$, $\pro{\lie{g}, M}{\lie{g'}, M}(V')$ and $\ind{g}{g'}(V')$ are $\rring{G/G'}\otimes \univ{g}$-modules.
Here $\rring{G/G'}$ is the ring of regular functions on $G/G'$.
We shall describe this module structure.

Let $G$ be an affine algebraic group over $\CC$
and $G'$ a closed subgroup of $G$.
We denote by $\Delta\colon \univ{g}\rightarrow \univ{g}\otimes \univ{g}$ the comultiplication of $\univ{g}$, i.e.\ it is an algebra homomorphism such that $\Delta(X)=X\otimes 1 + 1\otimes X$ for any $X \in \lie{g}$.
Take an open subset $U$ of $G/G'$.
We consider $\rring{U}\otimes \univ{g}$ as a $\CC$-algebra via the multiplication
\begin{align}
(f\otimes X)\cdot (h \otimes Y) = \sum_i f L(a_i) h \otimes b_i Y
\label{eqn:DefOU}
\end{align}
for $f, h \in \rring{U}$ and $ X, Y \in \univ{g}$,
where $\Delta(X)=\sum_i a_i \otimes b_i$ as an element of $\univ{g}\otimes \univ{g}$ and $L(\cdot)$ is the differential of the left translation of $G$ on $G/G'$.

Let $V'$ be a $\lie{g'}$-module.
Suppose that $U$ is an open neighborhood of $eG'\in G/G'$.
We define a right $\rring{U}\otimes \univ{g}$-module structure
on $\ind{g}{g'}(V') = V'\otimes_{\univ{g'}}\univ{g}$ via
\begin{align}
v\otimes X \cdot (f\otimes Y) = \sum_i L(a_i) f(e) v\otimes b_i Y \label{eqn:defActionOnDelta}
\end{align}
for $v \in V', X, Y \in \univ{g}$ and $f \in \rring{U}$ with $\Delta(X)=\sum_i a_i\otimes b_i$.

Assume that $U$ is affine.
It is easy to see that $\ind{g}{g'}(V')$ is supported on $\set{eG'}$ as an $\rring{U}$-module.
Hence $\ind{g}{g'}(\cdot)$ is a functor from $\Mod(\lie{g'})$ to the category of right $\rring{U}\otimes \univ{g}$-modules supported on $\set{eG'}$.
Since $\univ{g}$ is a free left $\univ{g'}$-module, the functor $\ind{g}{g'}(\cdot)$
is faithful and exact.
The following proposition is a variation of the Kashiwara equivalence of $\ntDsheaf$-modules \cite[Theorem 1.6.1]{HTT08}.

\begin{proposition}\label{prop:irreducibleDistribution}
	Let $I$ be the maximal ideal of $\rring{U}$ at $eG' \in U$.
	The functor $\ind{g}{g'}(\cdot)$ gives an equivalence of categories
	and its quasi-inverse is given by the functor $(\cdot)^I$.
	In particular, if $V'$ is an irreducible $\lie{g'}$-module, then $\ind{g}{g'}(V')$
	is an irreducible right $\rring{U}\otimes\univ{g}$-module.
\end{proposition}

\begin{proof}
Let $V'$ be a $\lie{g'}$-module.
We shall show $\ind{g}{g'}(V')^I \simeq V'$.
As an $\rring{U}$-module, $\ind{g}{g'}(V')$ is isomorphic to $\ind{g}{g'}(\CC)^{\oplus \dim_{\CC}(V')}$.
Since $U$ is affine, $\ind{g}{g'}(\CC)^*$ is isomorphic to the $I$-adic completion of $\rring{U}$.
Hence, taking a resolution of $\rring{U}/I$ by free modules with finite rank, we have
\begin{align*}
	\Ext^i_{\rring{U}}(\rring{U}/I, \ind{g}{g'}(\CC))^* \simeq \Tor_i^{\rring{U}}(\rring{U}/I, \ind{g}{g'}(\CC)^*) \simeq \begin{cases}
		\CC & (i = 0)\\
		0 & (i\neq 0)
	\end{cases}
\end{align*}
for any $i \in \NN$.
In particular, we have $\ind{g}{g'}(\CC)^I \simeq \CC$.
This implies that the natural homomorphism $V'\rightarrow \ind{g}{g'}(V')^I$
is an isomorphism.

Let $V$ be a right $\rring{U}\otimes \univ{g}$-module supported on $\set{eG'}$.
We shall show that the natural homomorphism $\varepsilon\colon \ind{g}{g'}(V^I)\rightarrow V$ is an isomorphism.
Consider the exact sequence
\begin{align*}
	0\rightarrow \Ker(\varepsilon) \rightarrow \ind{g}{g'}(V^I)\xrightarrow{\varepsilon} V \rightarrow V/\Im(\varepsilon) \rightarrow 0.
\end{align*}
Since $(\cdot)^I$ is left exact and $\ind{g}{g'}(V^I)^I$ is isomorphic to $V^I$,
we have $\Ker(\varepsilon)^I = 0$ and hence $\Ker(\varepsilon) = 0$.
By $\Ext^1_{\rring{U}}(\rring{U}/I, \ind{g}{g'}(V^I)) = 0$, we have
\begin{align*}
	(V/\Im(\varepsilon))^I = \Hom_{\rring{U}}(\rring{U}/I, V/\Im(\varepsilon)) = 0
\end{align*}
and hence $V/\Im(\varepsilon) = 0$.
This shows that $\varepsilon$ is an isomorphism.
\end{proof}

To use Lemmas \ref{lem:AnnhilatorRad} and \ref{lem:AnnhilatorSoc}, we need to compute the annihilator of an $\rring{G/G'}\otimes \univ{g}$-module.
We provide $\rring{G/G'}\otimes \univ{g}$ with a rational $G$-action via the tensor product $L\otimes \Ad$.

\begin{lemma}\label{lem:AnnhilatorOUmod}
	Let $V'$ be a $\lie{g'}$-module, and set $I:={}^t \Ann_{\univ{g'}}(V')$ and $J:=\Ann_{\rring{G/G'}\otimes \univ{g}}(\ind{g}{g'}(V'))$.
	Then we have
	\begin{align*}
			J = \set{f \in \rring{G/G'}\otimes\univ{g}: f(gG') \in \Ad(g)I\univ{g}\text{ ($\forall g\in G_0$)}},
	\end{align*}
	where an element of $\rring{G/G'}\otimes\univ{g}$ is regarded as a $\univ{g}$-valued function on $G/G'$.
\end{lemma}

\begin{proof}
	We consider $J':=\Ann_{\rring{G/G'}\otimes \univ{g}}(V'\otimes 1)$, 
	where $V'\otimes 1$ is a subspace of $\ind{g}{g'}(V')$.
	By \eqref{eqn:defActionOnDelta}, we have
	\begin{align*}
		v \otimes 1\cdot f = v\otimes f(eG')
	\end{align*}
	for any $v \in V'$ and $f \in \rring{G/G'}\otimes \univ{g}$.
	Since $\univ{g}$ is a free left $\univ{g'}$-module, we obtain
	\begin{align*}
		J' = \set{f \in \rring{G/G'}\otimes \univ{g} : f(eG') \in I\univ{g}}.
	\end{align*}
	
	Clearly we have
	\begin{align*}
		J = \set{f \in \rring{G/G'}\otimes \univ{g}: \univ{g}f \subset J'}.
	\end{align*}
	Hence $J$ is the maximum $(\univ{g}, \univ{g})$-stable subspace in $J'$.
	In other words, $J$ is the maximum $G_0$-stable subspace in $J'$.
	This implies
	\begin{align*}
		J &= \set{f \in \rring{G/G'}\otimes \univ{g}: G_0\cdot f \subset J'} \\
		& = \set{f \in \rring{G/G'}\otimes \univ{g}: f(gG') \in \Ad(g)I\univ{g}\text{ ($\forall g\in G_0$)}}.
	\end{align*}
	We have shown the lemma.
\end{proof}

As a vector space, $(\rring{G/G'}\otimes \univ{g})^{G}$ is naturally isomorphic to $\univ{g}^{G'}$.
This isomorphism is given by
\begin{align*}
(\rring{G/G'}\otimes \univ{g})^{G} \ni \sum_i f_i \otimes X_i \mapsto \sum_i f_i(eG')X_i \in \univ{g}^{G'}.
\end{align*}
It is easy to see that the map is an anti-isomorphism of algebras.
Let $V'$ be a $\lie{g'}$-module.
Using the anti-isomorphism, we consider $\ind{g}{g'}(V')$ as a left $\univ{g}^{G'}$-module.
The left $\univ{g}^{G'}$-action on $\ind{g}{g'}(V') = V'\otimes_{\univ{g'}} \univ{g}$ is simply written as
\begin{align*}
Y\cdot (v \otimes X) = v\otimes YX
\end{align*}
for $v\otimes X \in V'\otimes_{\univ{g'}} \univ{g}$ and $Y \in \univ{g}^{G'}$.
We shall describe the annihilator of the left $\univ{g}^{G'}$-module by that of $V'$. 

\begin{lemma}\label{lem:faithfully1}
Let $V'$ be a $\lie{g'}$-module and $I:={}^t\! \Ann_{\univ{g'}}(V')$.
Then the annihilator of the left $\univ{g}^{G'}$-module $\ind{g}{g'}(V')$
is equal to $\univ{g}^{G'} \cap I\univ{g}$.
\end{lemma}

\begin{proof}
$\Ann_{\univ{g}^{G'}}(\ind{g}{g'}(V')) \supset \univ{g}^{G'} \cap I\univ{g}$ is obvious and the converse follows from Lemma \ref{lem:AnnhilatorOUmod}.
\end{proof}

\subsection{Annihilator of \texorpdfstring{$\Ind^{G_\RR}_{G'_\RR}(V')$}{Ind(V')}}

Let $G_\RR$ be a Lie group and $G'_\RR$ a closed subgroup of $G_\RR$.
Assume that there exist an algebraic group $G$ with the Lie algebra $\lie{g}$
and a closed subgroup $G'$ of $G$ with the Lie algebra $\lie{g'}\subset \lie{g}$ such that the inclusion map $(\lie{g}_\RR, \lie{g}'_\RR) \hookrightarrow (\lie{g}, \lie{g'})$ lifts to a homomorphism $(G_\RR, G'_\RR) \rightarrow (G, G')$ of pairs of Lie groups.
Then we have a natural map $G_\RR/G'_\RR\rightarrow G/G'$.

Let $V'$ be a smooth Fr\'echet representation of $G'_\RR$.
Then the ring $\rring{G/G'}$ of regular functions acts on $\Ind^{G_{\RR}}_{G'_\RR}(V')$ via the multiplication of functions.
Combining the natural $\univ{g}$-action on $\Ind^{G_\RR}_{G'_\RR}(V')$, we obtain a left $\rring{G/G'}\otimes \univ{g}$-module structure on $\Ind^{G_\RR}_{G'_\RR}(V')$.
Set $\calA:= \rring{G/G'}\otimes \univ{g}$.

Let $W$ be a dense $\lie{g'}$-submodule of the strong dual $(V')^*$ of $V'$.
Then the natural pairing $\langle\cdot, \cdot \rangle\colon W\times V'\rightarrow \CC$ is non-degenerate on $V'$, that is, if $\langle w, v \rangle = 0$ for any $w \in W$, then $v = 0$.
The pairing induces a pairing $\langle \cdot , \cdot\rangle\colon \ind{g}{g'}(W)\times \Ind^{G_\RR}_{G'_\RR}(V')$ given by
\begin{align}
	\langle \varphi \otimes X, f\rangle = \varphi(Xf(e)) \label{eqn:PairingAmod}
\end{align}
for $\varphi \in W$, $X\in \univ{g}$ and $f \in \Ind^{G_\RR}_{G'_\RR}(V')$.
The following lemma is easy by definition.

\begin{lemma}\label{lem:PairingNondegenerate}
	The pairing $\langle \cdot , \cdot\rangle\colon \ind{g}{g'}(W)\times \Ind^{G_\RR}_{G'_\RR}(V')\rightarrow \CC$ is $\calA$-invariant and non-degenerate on the left.
\end{lemma}

Let $\ev\colon \Ind^{G_\RR}_{G'_\RR}(V')\rightarrow V'$ denote the evaluation map at the identity $e\in G_\RR$.

\begin{lemma}\label{lem:AnnSubInd}
	Suppose that $G/G'$ is affine.
	Let $N$ be a $G_\RR$-stable $\calA$-submodule of $\Ind^{G_\RR}_{G'_\RR}(V')$
	and set $U:=\overline{\ev(N)} \subset V'$.
	Then we have
	\begin{align*}
		\Ann_{\calA}(N) &= \Ann_{\calA}(\Ind^{G_\RR}_{G'_\RR}(U)), \\
		\Ann_{\univ{g}^{G'}}(N) &= \univ{g}^{G'} \cap \Ann_{\univ{g'}}(U)\univ{g}
	\end{align*}
	under the isomorphism $\calA^G \simeq (\univ{g}^{G'})^{\opalg}$.
\end{lemma}

\begin{proof}
	Obviously, we can assume $U = V'$.
	It is enough to show
	\begin{align*}
		\bigcap_{g \in G_\RR} g\cdot \Ann_\calA(\ind{g}{g'}(W)) = \Ann_{\calA}(N).
	\end{align*}
	In fact, this shows the first equation easily and we have
	\begin{align*}
		\calA^G\cap \bigcap_{g \in G_\RR} g\cdot \Ann_\calA(\ind{g}{g'}(W))  &= \Ann_{\univ{g}^{G'}}(\ind{g}{g'}(W))\\
		&= \univ{g}^{G'} \cap \Ann_{\univ{g'}}(V')\univ{g}
	\end{align*}
	by Lemma \ref{lem:faithfully1}.
	Remark that $\Ann_{\univ{g'}}(W) = {}^t \Ann_{\univ{g'}}(V')$.

	By $\overline{\ev(N)} = V'$, we have $N^\perp \cap (W\otimes 1) = 0$ in $\ind{g}{g'}(W) = W\otimes_{\univ{g'}}\univ{g}$,
	where we take $(\cdot)^\perp$ with respect to the pairing $\langle \cdot, \cdot\rangle$ defined in \eqref{eqn:PairingAmod}.
	By Proposition \ref{prop:irreducibleDistribution}, $N^\perp$ is generated by 
	$N^\perp \cap (W\otimes 1)$.
	Hence we obtain $N^\perp = 0$.
	This implies that the pairing $\langle \cdot, \cdot \rangle\colon \ind{g}{g'}(W)\times N \rightarrow \CC$ is non-degenerate on the left.
	
	Since the pairing $\langle \cdot, \cdot \rangle$ is non-degenerate on the left, we have $\Ann_{\calA}(N) \subset \Ann_\calA(\ind{g}{g'}(W))$.
	Since $\Ann_{\calA}(N)$ is $G_\RR$-stable, we obtain
	\begin{align*}
		\Ann_{\calA}(N) \subset \bigcap_{g \in G_\RR} g\cdot \Ann_\calA(\ind{g}{g'}(W)) =: J.
	\end{align*}

	We shall show the converse inclusion.
	Take $f \in J N$.
	Then we have $L(g)f \in J N$ for any $g \in G_\RR$.
	This implies $L(g)f(e) = 0$ for any $g \in G_\RR$ and hence $f = 0$.
	This shows the converse inclusion.
\end{proof}

We shall consider a similar result to Lemma \ref{lem:AnnSubInd} for $\pro{\lie{g}, M}{\lie{g'}, M}(V')$.
Let $G$ and $M$ be affine algebraic groups and $G'$ a closed subgroup of $G$.
Suppose that a homomorphism $M\rightarrow G'$ of algebraic groups is given and
its differential is injective.
Then we have a pair $(\lie{g}, M)$ and its subpair $(\lie{g'}, M)$.

Let $V'$ be a $(\lie{g'}, M)$-module.
We define an $\rring{G/G'}$-action on $\pro{\lie{g}, M}{\lie{g'}, M}(V')$ via
\begin{align*}
	(f\cdot \varphi)(X) = \sum_i L(a_i)f(e)\varphi(b_i)
\end{align*}
for $f \in \rring{G/G'}$, $\varphi \in \pro{\lie{g}, M}{\lie{g'}, M}(V')$ and $X \in \univ{g}$ with $\Delta(X) = \sum_i a_i \otimes b_i$.
It is easy to see that the action is well-defined.
Combining with the natural $\univ{g}$-action, we obtain an $\rring{G/G'}\otimes \univ{g}$-action on $\pro{\lie{g}, M}{\lie{g'}, M}(V')$.
Set $\calA := \rring{G/G'}\otimes \univ{g}$.

Let $\ev\colon \Hom_{\univ{g'}}(\univ{g}, V') \rightarrow V'$ denote the evaluation map at $1 \in \univ{g'}$.
Similarly to $\Ind^{G_\RR}_{G'_\RR}(V')$, we obtain the following result.
We omit the proof.

\begin{lemma}\label{lem:AnnSubIndAlg}
	Assume that $G/G'$ is affine.
	Let $V'$ be a $(\lie{g'}, M)$-module and $N$ an $M$-stable $\calA$-submodule of $\pro{\lie{g}, M}{\lie{g'}, M}(V')$.
	Then we have
	\begin{align*}
		\Ann_{\calA}(N) &= \Ann_{\calA}(\pro{\lie{g}, M}{\lie{g'}, M}(\ev(N))), \\
		\Ann_{\univ{g}^{G'}}(N) &= \univ{g}^{G'} \cap \Ann_{\univ{g'}}(\ev(N)) \univ{g}
	\end{align*}
	under the isomorphism $\calA^G \simeq (\univ{g}^{G'})^\opalg$.
\end{lemma}

\subsection{Sphericity and multiplicity}

It is well-known that the sphericity of homogeneous varieties is strongly related to the multiplicity-freeness.
We shall relate the finiteness of the PI degree to the sphericity.

Let $G$ be a reductive algebraic group, $G'$ a reductive subgroup of $G$ and $H$ a closed subgroup of $G$.
Fix a Borel subgroup $B'$ of $G'_0$.
We denote by $\rsheaf{G/H, eH}$ the stalk of the structure sheaf $\rsheaf{G/H}$ at $eH$.

\begin{lemma}\label{lem:surjectiveMult}
For any non-zero $s \in \rsheaf{G/H, eH}$, the multiplication by $s$
\begin{align*}
\times s\colon \ind{g}{h}(\CC) \rightarrow \ind{g}{h}(\CC)
\end{align*}
defined by (\ref{eqn:defActionOnDelta}) is surjective.
\end{lemma}

\begin{proof}
As an $\rsheaf{G/H, eH}$-module, $(\ind{g}{h}(\CC))^*$ is isomorphic
to the $I$-adic completion $R$ of $\rsheaf{G/H, eH}$, where $I$ is 
the unique maximal ideal of $\rsheaf{G/H, eH}$.
Since the action of $\rsheaf{G/H, eH}$ on $R$ is torsion-free,
the multiplication by $s$ on $R$ is injective.
Hence $\times s$ is surjective on $\ind{g}{h}(\CC)$.
\end{proof}

\begin{proposition}\label{prop:multiplicityNotSpherical}
Let $V$ be a finitely generated $\lie{h}$-module.
Suppose that $G_0/H_0$ is not $G'_0$-spherical.
Then there exist some $g \in G_0$ and character $\mu$ of $\Ad(g)\lie{b'}$ such that
\begin{align*}
\dim_{\CC}(\ind{g}{h}(V) \otimes_{\calU(\Ad(g)\lie{b'})} \CC_{\mu}) = \infty.
\end{align*}
\end{proposition}

\begin{proof}
	Clearly, we can assume that $G$ and $G'$ are connected.

	We denote by $K(G/H)$ the field of all rational functions on $G/H$.
	Since $G/H$ is not $G'$-spherical, $K(G/H)^{B'}$ is infinite dimensional (see e.g.\ \cite[25.1]{Ti11}).
	Hence we can take a $B'$-stable open subset $U$ of $G/H$ such that $\dim_{\CC}(\rring{U}^{B'}) = \infty$.

	Take a Borel subgroup $B$ of $G$ containing $B'$ with unipotent radical $N$.
	Since $V$ is finitely generated, $\ind{g}{h}(V)$ is finitely generated
	and has an irreducible quotient $W$ as a $\lie{g}$-module.
	By a generalization of the Casselman subrepresentation theorem \cite[Theorem 1]{BeBe83}, there is an open subset $U'$ of $G$ such that $W/\Ad(g)(\lie{n})W\neq 0$
	for any $g \in U'$.
	If necessary, replacing $B$ and $B'$ with $\Ad(g)B$ and $\Ad(g)B'$ for some $g \in (UH)^{-1}\cap U'$, we can assume $eH \in U$ and $W/\lie{n}W \neq 0$.
	Since $W$ is irreducible, there is a character $\mu$ of $\lie{b}$ such that $W\otimes_{\univ{b}}\CC_{\mu} \neq 0$
	and hence $\ind{g}{h}(V)\otimes_{\univ{b'}}\CC_{\mu} \neq 0$.

	Put $M:=\ind{g}{h}(V) \otimes_{\univ{b'}} \CC_{\mu}$.
	Assume that $M$ is finite dimensional.
	For $s \in \rring{U}^{B'}$, the multiplication $\times s\colon \ind{g}{h}(V)\rightarrow \ind{g}{h}(V)$ by $s$ induces a linear map $\times s\colon M\rightarrow M$.
	Then the multiplication $\times s \colon M \rightarrow M$ is surjective by Lemma \ref{lem:surjectiveMult}.
	Remark that $\ind{g}{h}(V)$ is isomorphic to $(\ind{g}{h}(\CC))^{\oplus \dim_{\CC}(V)}$ as an $\rring{U}$-module.
	Since $M$ is finite dimensional, $\times s$ is bijective on $M$.

	Fix a non-zero vector $v \in M$.
	Then the map $\rring{U}^{B'}\rightarrow M$ ($s \mapsto v \cdot s$) is injective
	and hence we have $\dim_{\CC}(M) \geq \dim_{\CC}(\rring{U}^{B'}) = \infty$.
	This is a contradiction.
\end{proof}

\begin{remark}
	The proposition is an analogue of \cite[Theorem 3.7]{Bi93} and \cite[Lemme 2.1]{Br94}.
	If a quasi-projective smooth $G$-variety $X$ is not spherical,
	then for any positive integer $n$, there exist some character $\lambda$ of a Borel subgroup $B$ and $G$-equivariant line bundle $\calL$ on $X$ such that
	the dimension of $\Hom_B(\CC_\lambda, \sect(\calL))$ is greater than $n$.
	Here $\sect(\calL)$ denotes the space of all global sections of $\calL$.
\end{remark}

We consider the cases of affine homogeneous varieties and flag varieties.
We shall show that the finiteness of the PI degree implies the sphericity of $G/H$.

\begin{proposition}\label{prop:boundedAndSphericalGH}
Retain the notation as above and suppose that $H$ is reductive.
Let $I$ be a proper two-sided ideal of $\univ{h}$.
If $G_0/H_0$ is not $G_0$-spherical, then we have $\PIdeg(\univ{g}^H/I\univ{g} \cap \univ{g}^H) = \infty$.
\end{proposition}

\begin{proof}
	Take a maximal two-sided ideal $J$ containing $I$.
	Then $J$ is a primitive ideal of $\univ{h}$.
	By Duflo's theorem (Fact \ref{fact:DufloIdeal}), we can take an irreducible highest weight module $V$ of $\lie{h}$
	such that $\Ann_{\univ{h}}(V) = {}^t\! J$.

	By Proposition \ref{prop:multiplicityNotSpherical}, there are a Borel subgroup $B$ of $G_0$ and a character $\mu$ of $\lie{b}$
	such that $\dim_{\CC}(\ind{g}{h}(V)\otimes_{\univ{b}} \CC_{\mu}) = \infty$.
	By Fact \ref{fact:BoundedTorHom}, the length of the $\univ{g}^H/I\univ{g}\cap \univ{g}^H$-module $V\otimes_{\univ{h}}\univ{g}\otimes_{\univ{b}} \CC_{\mu}$
	is finite.
	We therefore obtain $\PIdeg(\univ{g}^H/I\univ{g} \cap \univ{g}^H) = \infty$ by Proposition \ref{prop:pidLowerbound}.
\end{proof}

\begin{remark}
	Suppose that $G$ is connected and $I = \lie{h}\univ{h}$.
	Then the algebra $\univ{g}^H/I\univ{g} \cap \univ{g}^H$ is isomorphic to the algebra of all $G$-invariant differential operators on $G/H$
	(see e.g.\ \cite[Theorem 4.6]{He00}).
	In this case, it is well-known that $\univ{g}^H/I\univ{g} \cap \univ{g}^H$ is commutative (i.e.\ $\PIdeg(\univ{g}^H/I\univ{g} \cap \univ{g}^H) = 1$)
	if and only if $G/H$ is $G$-spherical \cite[Theorem 25.4]{Ti11}.
\end{remark}

\begin{proposition}\label{prop:boundedAndSphericalGP}
Let $P$ be a parabolic subgroup of $G_0$ with unipotent radical $U$
and $J$ a proper two-sided ideal of $\univ{p}$ containing $\lie{u}$.
Put
\begin{align*}
	I:=\set{X \in \univ{g}: X\univ{g}\subset \univ{g}J}.
\end{align*}
If $G_0/P$ is not $G'_0$-spherical, then we have $\PIdeg(\univ{g}^{G'}/I \cap \univ{g}^{G'}) = \infty$.
\end{proposition}

\begin{remark}
	If the annihilator of a $\univ{p}$-module $V$ is equal to $J$,
	then we have $I = \Ann_{\univ{g}}(\ind{g}{p}(V))$.
	See Lemma \ref{lem:AnnhilatorOUmod}.
\end{remark}

\begin{proof}
	Take a maximal two-sided ideal $J'$ of $\univ{p}$ containing $J$,
	and put
	\begin{align*}
		I':=\set{X \in \univ{g}: X\univ{g}\subset \univ{g}J'}.
	\end{align*}
	Then $J'$ is a primitive ideal of $\univ{p}$.
	By Duflo's theorem (Fact \ref{fact:DufloIdeal}), we can take an irreducible highest weight module $V$ of $\lie{p/u}$
	such that $\Ann_{\univ{p}}(V) = J'$.
	Then we have $I' = \Ann_{\univ{g}}(\ind{g}{p}(V))$.

	By Proposition \ref{prop:multiplicityNotSpherical}, there are an element $g \in G_0$ and a character $\mu$ of $\Ad(g)\lie{b'}$
	such that $\dim_{\CC}(\ind{g}{p}(V)\otimes_{\calU(\Ad(g)\lie{b'})} \CC_{\mu}) = \infty$.
	Since $I$ is $G_0$-stable, $\Ad(g)$ induces an isomorphism
	\begin{align*}
		\univ{g}^{G'}/I \cap \univ{g}^{G'} \simeq 
		\univ{g}^{\Ad(g)G'}/I\cap \univ{g}^{\Ad(g)G'}.
	\end{align*}
	Hence if necessary, replacing $G'$ with $\Ad(g)G'$ and $B'$ with $\Ad(g)B'$, we can assume $g=e$.
	
	By Fact \ref{fact:BoundedTorHom}, the $\univ{g}^{G'}/I \cap \univ{g}^{G'}$-module
	\begin{align*}
		\ind{g}{p}(V)\otimes_{\univ{b'}} \CC_{\mu} \simeq \ind{g}{p}(V)\otimes_{\univ{g'}} \ind{g'}{b'}(\CC_\mu)
	\end{align*}
	has finite length.
	We therefore obtain $\PIdeg(\univ{g}^{G'}/I \cap \univ{g}^{G'}) = \infty$ by Proposition \ref{prop:pidLowerbound}.
\end{proof}

We have proved the results in this subsection by the representation theoretic approach.
One can show the results by a Poisson geometrical approach.
See Section \ref{sect:Coisotropic}.

\subsection{Spherical pair}\label{subsect:Spherical}

In the previous subsection, we have seen that the finiteness of the PI degree implies the sphericity of homogeneous varieties.
Hence to consider the uniform boundedness of multiplicities in induced representations, we only have to deal with spherical varieties.
We shall recall and study the local structure theorem of spherical varieties.

Let $G$ be a connected reductive algebraic group with a complex conjugate $\tau$
and $H$ a $\tau$-stable reductive subgroup of $G$.
Put $G_{\RR}:=G^{\tau}$ and $H_{\RR}:=H^{\tau}$.
We assume that $(G, H)$ is a spherical pair, that is, a Borel subgroup of $G$ has an open orbit in $G/H$.
We need the following result by Brion--Luna--Vust \cite[0.4]{BLV86}.

\begin{fact}\label{fact:BLV}
Let $B$ be a Borel subgroup of $G$ such that $BH\subset G$ is open dense.
Put $P:=\set{g \in G: g BH = BH}$.
Then $P$ is a parabolic subgroup of $G$ containing $B$
and $L:=P\cap H$ is a reductive subgroup of $H$ containing the derived group of a Levi subgroup of $P$.
\end{fact}
From the fact, we can see that $(B\cap H)_0 \subset L_0$ is a Borel subgroup of $L_0$
and $B\cap H$ meets every connected component of $L$ (see \cite[Proposition 4.3]{Ki14}).
In fact, $L/(B\cap H)$ is isomorphic to the full flag variety of a Levi subgroup of $P$.

The reductive group $L$ may not be $\tau$-stable, which depends on the choice of $B$.
The following proposition enable us to take a Borel subgroup $B$ such that $L$ is $\tau$-stable.

\begin{proposition}\label{prop:BLVreal}
	There exists a Borel subgroup $B$ such that $BH\subset G$ is open dense and the reductive group $L$ in Fact \ref{fact:BLV} is $\tau$-stable.
\end{proposition}

\begin{proof}
	If $\tau$ is a Cartan involution of $G$, the assertion is easy as follows.
	Take a Borel subgroup $B$ of $G$ such that $BH$ is open in $G$.
	Since $L$ is a reductive subgroup of $H$ and $\tau$ is a Cartan involution, there is $g \in H$ such that $\Ad(g)L$ is $\tau$-stable.
	Hence $\Ad(g)B$ satisfies the desired conditions.

	We shall show the general case.
	Since $G/H$ is a spherical $G$-variety, we can take a minimal parabolic subgroup $P_{\RR}$ of $G_{\RR}$
	such that $P_{\RR}H_{\RR} \subset G_{\RR}$ is open.
	Then the analytic subgroup $P \subset G$ with the Lie algebra $\lie{p} \subset \lie{g}$ is a parabolic subgroup of $G$ and $PH$ is open dense in $G$.
	Fix a Langlands decomposition $P_{\RR}=M_{\RR}A_{\RR}N_{\RR}$.
	Let $M$ (resp.\ $A$, $N$) be the analytic subgroup of $G$ with the Lie algebra $\lie{m}$ (resp.\ $\lie{a}$, $\lie{n}$).
	Then the unipotent radical of $P$ is $N$.

	Take a Borel subgroup $B$ of $G$ contained in $P$ such that $BH$ is open in $G$.
	Then $L$ contains $H\cap AN$ by the definition of $L$.
	Set $G':=P/AN\simeq M/(M\cap A)$, $H':=(H\cap P)/(H\cap AN)$ and $B':=B/AN$.
	Then $B'$ is a Borel subgroup of $G'$ and $B'H'$ is open in $G'$ by
	\begin{align*}
		P/B\simeq G'/B' \supset H'B'/B'\simeq (P\cap H)B/B = P/B\cap HB/B.
	\end{align*}
	Since $P$ and $AN$ are $\tau$-stable, $\tau$ induces a complex conjugate $\overline{\tau}$ on $G'$.
	Let $p$ be the natural projection $P\rightarrow P/AN$.
	For $g \in H\cap P$, we have
	\begin{align*}
		g \in L\cap P & \Leftrightarrow g BH = BH \\
		&\Leftrightarrow g (P\cap BH) = P\cap BH \\
		&\Leftrightarrow g B(P\cap H) = B(P\cap H) \\
		&\Leftrightarrow p(g)B'H'= B'H'.
	\end{align*}
	This implies that $p(L\cap P)$ is $L$ for $(G, H, B) = (G', H', B')$.
	Since $\overline{\tau}$ is a Cartan involution on $G'$, we can choose $B$ such that $p(L\cap P) = (L\cap P)/(H\cap AN)$ is $\overline{\tau}$-stable.
	Hence $L\cap P$ is $\tau$-stable.

	Since $L$ is reductive and $L/L\cap B$ is connected and projective, $L$ can be written as
	\begin{align*}
	L=&\set{g \in G: f(g)=f(e) \text{ for any }f\in \rring{G}^L}\\
	= &\set{g \in G: f(g)=f(e) \text{ for any }f\in \rring{G}^{L\cap P}}.
	\end{align*}
	(In other words, $L$ is the observable hull of $L\cap P$.)
	By $L\cap P = \tau(L\cap P)$, we obtain $L=\tau(L)$.
\end{proof}

\section{Uniformly bounded multiplicities}\label{sect:UniformlyBoundedTheorem}

In this section, we prove the main theorems in this paper.
We give necessary and sufficient conditions for the multiplicity-freeness of restrictions and inductions using the PI degree.
We have proved crucial parts of the theorems in Lemmas \ref{lem:UpperBoundTorGeneral}, \ref{lem:lowerboundQuot}, \ref{lem:lowerboundSub} and \ref{lem:lowerboundTensor}.
The remaining part is to verify the assumptions of the lemmas under concrete settings.

\subsection{Outline}\label{sect:Setting}

Let $G_\RR$ be a real reductive Lie group and $G'_\RR$ a reductive subgroup of $G_\RR$.
Fix a maximal compact subgroup $K_\RR$ of $G_\RR$ such that $K'_\RR:=G'_\RR \cap K_\RR$ is a maximal compact subgroup of $G'_\RR$.
Let $K$ and $K'$ denote the complexifications of $K_\RR$ and $K'_\RR$, respectively.
Then we have a pair $(\lie{g}, K)$ and its subpair $(\lie{g'}, K')$.

Assume that $\lie{g'}$ is algebraic in $\lie{g}$.
Then there are a reductive algebraic group $G$ with the Lie algebra $\lie{g}$ and a reductive subgroup $G'$ of $G$ with the Lie algebra $\lie{g'}\subset \lie{g}$ such that the inclusion $(\lie{g}_\RR, \lie{g}'_\RR) \hookrightarrow (\lie{g}, \lie{g'})$ of pairs of Lie algebras lifts to a homomorphism $(G_\RR, G'_\RR)\rightarrow (G, G')$ of pairs of Lie groups.
We can assume that the image of $G_\RR$ (resp.\ $G'_\RR$) is dense in $G$ (resp.\ $G'$).

In this section, we consider the supremum of multiplicities in restrictions and inductions of modules of finite length.
For example, for a $(\lie{g}, K)$-module $V$ of finite length, 
we want to compare $\SupDim(\Hom_{\lie{g}, K}(V, \calC_{\lie{g'}, K'}))$
with the PI degree $\PIdeg((\univ{g}/I)^{G'})$, where $I = \Ann_{\univ{g}}(V)$.
See Definition \ref{def:supmul} for the notation $\SupDim(\cdot)$.
A conclusion in this section is that there exists some constant $C > 0$ independent of $V$ satisfying
\begin{align*}
	\PIdeg((\univ{g}/I)^{G'})&\leq \SupDim(\Hom_{\lie{g'}, K'}(V, \calC_{\lie{g'}, K'})) \\
	&\leq C \cdot \Len_{\lie{g}, K}(V)\cdot \PIdeg((\univ{g}/I)^{G'}).
\end{align*}

In an abstract setting, we have already given an upper bound and a lower bound of the supremum of multiplicities by the PI degree in Lemmas \ref{lem:UpperBoundTorGeneral}, \ref{lem:lowerboundQuot}, \ref{lem:lowerboundSub} and \ref{lem:lowerboundTensor}.
We shall explain how to apply the lemmas to concrete settings.

We can easily verify the assumption of Lemma \ref{lem:UpperBoundTorGeneral}.
In fact, any family of $(\lie{g}, K)$-modules (or objects in $\calO_{\lie{b}}$)
with bounded lengths is a uniformly bounded family as we have seen in Fact \ref{fact:UniformlyBoundedFamilyGK}.

To use Lemmas \ref{lem:lowerboundQuot}, \ref{lem:lowerboundSub} and \ref{lem:lowerboundTensor}, we need to set up the data $\calA$, $\calB$, $\calC$, $\widetilde{\calC}$, $\calF^\calC$ and $\calF^{\widetilde{\calC}}$.
For restrictions (the branching problem), we set $\calA = \univ{g}$.
In this case, $\calF^\calC$ and $\calF^{\widetilde{\calC}}$ are just forgetful functors (or inclusion functors).

\begin{example}
	Set $\calB = \Mod(G'_\RR)$, $\calC = \calC_{G'_\RR}$ and $\widetilde{\calC} = \CWCat{G_\RR}$.
	Then we have
	\begin{align*}
		\Hom_{\calB}(\calF^{\widetilde{\calC}}(V), \calF^\calC(V')) = \Hom_{G'_\RR}(V|_{G'_\RR}, V')
	\end{align*}
	for $V \in \widetilde{\calC}$ and $V' \in \calC$.
	This is the case of the branching problem of Casselman--Wallach representations.

	Set $\calB = \Mod(\lie{g'}, K')$, $\calC = \calC_{\lie{g'}, K'}$ and $\widetilde{\calC} = \Mod_{fl}(\lie{g}, K)$.
	Then we have
	\begin{align*}
		\Hom_{\calB}(\calF^{\widetilde{\calC}}(V), \calF^\calC(V')) = \Hom_{\lie{g'}, K'}(V|_{\lie{g'}, K'}, V')
	\end{align*}
	for $V \in \widetilde{\calC}$ and $V' \in \calC$.
	This is the case of the branching problem of $(\lie{g}, K)$-modules.
\end{example}

For inductions (harmonic analysis), we set $\calA = \rring{G/G'}\otimes \univ{g}$.
In this case, $\calF^{\widetilde{\calC}}$ is more complicated than the restriction case.
$\calF^{\widetilde{\calC}}$ is one of the induction functors $\Ind^{G_\RR}_{G'_\RR}$, $\pro{\lie{g}, K'}{\lie{g'}, K'}$ and $\ind{g}{g'}$.

\begin{example}
	Set $\calB = \Mod(G_\RR)$, $\calC = \calC_{G_\RR}$, $\widetilde{\calC} = \CWCat{G'_\RR}$ and $\calF^{\widetilde{\calC}} = \Ind^{G_\RR}_{G'_\RR}$.
	Let $\calF^\calC$ be the inclusion functor.
	Then we have
	\begin{align*}
		\Hom_{\calB}(\calF^\calC(V), \calF^{\widetilde{\calC}}(V')) = \Hom_{G_\RR}(V, \Ind^{G_\RR}_{G'_\RR}(V')) \simeq \Hom_{G'_\RR}(V|_{G'_\RR}, V')
	\end{align*}
	for $V \in \calC$ and $V' \in \widetilde{\calC}$.
	This is the case of harmonic analysis of $G_\RR/G'_\RR$.

	Set $\calB = \Mod(\lie{g}, K')$, $\calC = \calC_{\lie{g}, K}$, $\widetilde{\calC} = \Mod_{fl}(\lie{g'}, K')$ and $\calF^{\widetilde{\calC}} = \pro{\lie{g}, K'}{\lie{g'}, K'}$.
	Let $\calF^\calC$ be the forgetful functor.
	Then we have
	\begin{align*}
		\Hom_{\calB}(\calF^\calC(V), \calF^{\widetilde{\calC}}(V')) &= \Hom_{\lie{g}, K'}(V|_{\lie{g}, K'}, \pro{\lie{g}, K'}{\lie{g'}, K'}(V'))\\
		&\simeq \Hom_{\lie{g'}, K'}(V|_{\lie{g'}, K'}, V')
	\end{align*}
	for $V \in \calC$ and $V' \in \widetilde{\calC}$.
\end{example}

To apply Lemmas \ref{lem:lowerboundQuot}, \ref{lem:lowerboundSub} and \ref{lem:lowerboundTensor}, we need to verify two assumptions.
We shall explain this using the restriction and the induction of Casselman--Wallach representations as above.
If $\calC = \calC_{G'_\RR}$ and $\widetilde{\calC} = \CWCat{G_\RR}$ (the restriction case),
the two assumptions of Lemma \ref{lem:lowerboundQuot} can be written as
\begin{enumerate}
	\item $\rad(V, \calC) \neq V$
	\item $\overline{I^k V}$ is $G_\RR$-stable for $I = \Ann_{\univ{g}}(V/\rad(V, \calC))$
\end{enumerate}
for any non-zero $V \in \CWCat{G_\RR}$.

The first condition says that $V|_{G'_\RR}$ has at least one irreducible quotient.
We will show this in the next subsection.
Remark that $\rad(V, \calC) = 0$ if $V$ is unitarizable (see Theorem \ref{thm:PointWiseGAmodule}).
The second condition is trivial if $G_\RR = G'_\RR (G_\RR)_0$.
In fact, since $I$ is a $G'_\RR$-stable two-sided ideal of $\univ{g}$,
$I^k V$ is $G_\RR$-stable and hence so is $\overline{I^k V}$.
Obviously, to show the lower bound, we can replace $G_\RR$ with $G'_\RR (G_\RR)_0$.

Note that the second condition is trivial if we work with categories of algebraic representations, e.g.\ $\Mod(\lie{g}, K)$ and $\Mod(\lie{g})$.
We should replace $K$ with $K_0K'$ as we have remarked for Casselman--Wallach representations.

If $\calC = \calC_{G_\RR}$, $\widetilde{\calC} = \CWCat{G'_\RR}$ and $\calF^{\widetilde{\calC}} = \Ind^{G_\RR}_{G'_\RR}$ (the induction case), the two assumptions of Lemma \ref{lem:lowerboundSub} can be written as
\begin{enumerate}
	\item $\soc(\Ind^{G_\RR}_{G'_\RR}(V'), \calC) \neq 0$
	\item $\Ind^{G_\RR}_{G'_\RR}(V')^{I^k} = \Ind^{G_\RR}_{G'_\RR}(V'_k)$ for some closed subrepresentation $V'_k \subset V$
\end{enumerate}
for any non-zero $V' \in \CWCat{G'_\RR}$,
where $I = \Ann_{\rring{G/G'}\otimes \univ{g}}(\soc(\Ind^{G_\RR}_{G'_\RR}(V'), \calC))$.
The first condition says that the induced representation $\Ind^{G_\RR}_{G'_\RR}(V')$ has at least one irreducible subrepresentation.
We will see this in the next subsection.
The second condition has been proved in Lemma \ref{lem:AnnSubInd}.

In cases to which Lemmas \ref{lem:lowerboundQuot}, \ref{lem:lowerboundSub} and \ref{lem:lowerboundTensor} will be applied, the second conditions are trivial or proved in Lemmas \ref{lem:AnnSubInd} and \ref{lem:AnnSubIndAlg}.
Hence what we have to show is the existence of an irreducible subrepresentation or quotient, which is the first conditions of the lemmas.

\subsection{Non-vanishing of multiplicity}

We shall show the existence of an irreducible subrepresentation or quotient as we have explained in the previous subsection.

\subsubsection{Restriction}

\begin{lemma}\label{lem:ExistenceQuotRest}
	Let $V$ be a non-zero Casselman--Wallach representation of $G_\RR$.
	Fix a Borel subalgebra $\lie{b'}$ of $\lie{g'}$ such that $\lie{b'} + \lie{k'} = \lie{g'}$.
	\begin{enumparen}
		\item\label{item:ExistenceQuotRestCW} There exists some $V' \in \CWCat{G'_\RR}$ such that $\Hom_{G'_\RR}(V, V') \neq 0$.
		\item\label{item:ExistenceQuotRestGK} There exists some $V' \in \Mod_{fl}(\lie{g'}, K')$ such that $\Hom_{\lie{g'}, K'}(V_K, V') \neq 0$.
		\item\label{item:ExistenceQuotRestGKO} There exists some $V' \in \calO_{\lie{b'}}$ such that $V_K\otimes_{\univ{g'}} V' \neq 0$.
	\end{enumparen}
\end{lemma}

\begin{proof}
	We shall show \ref{item:ExistenceQuotRestGK} and \ref{item:ExistenceQuotRestGKO} from \ref{item:ExistenceQuotRestCW}.
	Take $V' \in \CWCat{G'_\RR}$ such that $\Hom_{G_\RR}(V, V')\neq 0$.
	Then we have $\Hom_{\lie{g'}, K'}(V_K, V'_{K'}) \neq 0$ since $V_K$ is dense in $V$.
	This implies that $V_K|_{\lie{g'}, K'}$ has an irreducible quotient $W$.
	By the Casselman subrepresentation theorem (see \cite[Theorem 3.8.3]{Wa88_real_reductive_I}), there is a character $\CC_\lambda$ of $\lie{b'}$ such that $W\otimes_{\univ{b'}} \CC_\lambda \neq 0$.
	Hence we have
	\begin{align*}
		0\neq V_K\otimes_{\univ{b'}} \CC_{\lambda} \simeq V_K\otimes_{\univ{g'}} \ind{g'}{b'}(\CC_{\lambda}).
	\end{align*}
	This shows \ref{item:ExistenceQuotRestGKO}.

	We shall show the property \ref{item:ExistenceQuotRestCW}.
	By the Frobenius reciprocity, we have
	\begin{align*}
		\Hom_{(G'_\RR)_0}(V, V') \simeq \Hom_{G'_\RR}(V, \Ind^{G'_\RR}_{(G'_\RR)_0}(V'))
	\end{align*}
	for any Casselman--Wallach representation $V'$ of $(G'_\RR)_0$.
	Note that $\Ind^{G'_\RR}_{(G'_\RR)_0}(V')$ is in $\CWCat{G'_\RR}$.
	Hence we can assume that $G_\RR$ and $G'_\RR$ are connected.
	It is enough to show the assertion for irreducible $V$.

	Fix a minimal parabolic subgroup $P_{\RR}=M_{\RR}A_{\RR}N_{\RR}$ of $G_{\RR}$
	and a minimal parabolic subgroup $P'_{\RR}=M'_{\RR}A'_{\RR}N'_{\RR}$ of $G'_{\RR}$,
	where $M_{\RR}$ and $M'_{\RR}$ are compact, $A_{\RR}$ and $A'_{\RR}$ are split abelian, and $N_{\RR}$ and $N'_{\RR}$ are the nilradicals.
	We can assume $A'_{\RR} \subset A_{\RR}$ and $N'_{\RR} \subset N_{\RR}$.
	(In general, $M'_{\RR} \subset M_{\RR}$ does not hold.)
	
	By the Casselman subrepresentation theorem, we can assume that $V$ is a subrepresentation of $\Ind^{G_\RR}_{P_\RR}(F)$ for some  finite-dimensional irreducible representation $F$ of $P_{\RR}$.

	Consider the inclusion map $G'_\RR/G'_\RR\cap P_\RR \rightarrow G_\RR/P_\RR$.
	Then we have the restriction map
	\begin{align*}
		r\colon \Ind^{G_\RR}_{P_\RR}(F) \rightarrow \Ind^{G'_\RR}_{G'_\RR \cap P_\RR}(F)
		\simeq \Ind^{G'_\RR}_{P'_\RR}(\Ind^{P'_\RR}_{G'_\RR \cap P_\RR}(F)).
	\end{align*}
	Since $V$ is $G_\RR$-stable, $r(V)$ is non-zero.

	Recall that $G'_\RR \cap P_\RR$ contains $A'_\RR N'_\RR$.
	Hence any $M'_\RR$-subrepresentation of $\Ind^{P'_\RR}_{G'_\RR \cap P_\RR}(F)$ is $P'_\RR$-stable.
	Since $r(V)$ is non-zero and $M'_\RR$ is compact, there is an irreducible quotient $p\colon \Ind^{P'_\RR}_{G'_\RR \cap P_\RR}(F)\rightarrow F'$ such that
	the composition
	\begin{align*}
		V\xrightarrow{r} \Ind^{G'_\RR}_{P'_\RR}(\Ind^{P'_\RR}_{G'_\RR \cap P_\RR}(F)) \xrightarrow{\Ind^{G'_\RR}_{P'_\RR}(p)} \Ind^{G'_\RR}_{P'_\RR}(F')
	\end{align*}
	is non-zero.
	This shows the lemma.
\end{proof}

\subsubsection{Induction}

\begin{lemma}\label{lem:ExistenceQuotInduction}
	Let $V'$ be a non-zero Casselman--Wallach representation of $G'_\RR$.
	\begin{enumparen}
		\item \label{item:ExistenceQuotInductionCW} $\Hom_{G_\RR}(V, \Ind^{G_\RR}_{G'_\RR}(V')) \neq 0$ for some $V \in \CWCat{G_\RR}$.
		\item \label{item:ExistenceQuotInductionGK} $\Hom_{\lie{g}, K'}(V, \pro{\lie{g}, K'}{\lie{g'}, K'}(V'_{K'})) \neq 0$ for some $V \in \Mod_{fl}(\lie{g}, K)$.
	\end{enumparen}
\end{lemma}

\begin{proof}
	By the Frobenius reciprocity, we have
	\begin{align*}
		\Hom_{G_\RR}(V, \Ind^{G_\RR}_{G'_\RR}(V')) &\simeq \Hom_{G'_\RR}(V, V'),\\
		\Hom_{\lie{g}, K'}(V, \pro{\lie{g}, K'}{\lie{g'}, K'}(V'_{K'})) &\simeq \Hom_{\lie{g'}, K'}(V, V'_{K'})
	\end{align*}
	in \ref{item:ExistenceQuotInductionCW} and \ref{item:ExistenceQuotInductionGK}.
	As in the proof of Lemma \ref{lem:ExistenceQuotRest}, \ref{item:ExistenceQuotInductionGK} follows from \ref{item:ExistenceQuotInductionCW}, and we can assume that $G_\RR$ and $G'_\RR$ are connected and $V'$ is irreducible.
	
	Fix a generic element $H$ of $(\lieR{k'})^{\perp}$ in $\lieR{g'}$.
	Let $\lieR{p}$ (resp.\ $\lieR{p'}$) be the sum of $\ad_{\lieR{g}}(H)$-eigenspaces (resp.\ $\ad_{\lieR{g'}}(H)$-eigenspaces) with non-negative eigenvalues.
	Then the normalizer $P_{\RR} \subset G_{\RR}$ (resp.\ $P'_{\RR} \subset G'_{\RR}$) of $\lieR{p}$ (resp.\ $\lieR{p'}$)
	is a parabolic subgroup of $G_{\RR}$ (resp.\ $G'_{\RR}$).
	By construction, $P'_{\RR}$ is a minimal parabolic subgroup of $G'_{\RR}$ and $P_{\RR} \cap G'_{\RR} = P'_{\RR}$ holds.
	
	By the Casselman subrepresentation theorem, there is a finite-dimensional irreducible representation $F'$ of $P'_{\RR}$
	such that $V'$ is an irreducible quotient of $\Ind^{G'_{\RR}}_{P'_{\RR}}(F')$.
	Fix Langlands decompositions $P'_\RR = M'_\RR A'_\RR N'_\RR$ and $P_\RR = M_\RR A_\RR N_\RR$ such that $M'_\RR \subset M_\RR$, $A'_\RR\subset A_\RR$ and $N'_\RR \subset N_\RR$.
	Here $A_\RR$ and $A'_\RR$ are split abelian, $M'_\RR$ is compact and $M_\RR$ is reductive.
	Hence we can take an irreducible Casselman--Wallach representation $F$ of $M_{\RR}A_\RR$
	such that $F'$ is an irreducible quotient of $F|_{M'_\RR A'_\RR}$.
	
	Since the restriction map $\Ind^{G_{\RR}}_{P_{\RR}}(F) \rightarrow \Ind^{G'_{\RR}}_{P'_{\RR}}(F|_{P'_{\RR}})$ is surjective, we obtain a surjective $G'_\RR$-linear map defined by the composition
	\begin{align*}
		\Ind^{G_{\RR}}_{P_{\RR}}(F) \rightarrow \Ind^{G'_{\RR}}_{P'_{\RR}}(F|_{P'_{\RR}})\rightarrow \Ind^{G'_{\RR}}_{P'_{\RR}}(F')
		\rightarrow V'.
	\end{align*}
	This shows the lemma.
\end{proof}

\subsubsection{BGG category \texorpdfstring{$\calO$}{O}}

We consider non-vanishing results for the branching problem in the BGG category $\calO$.
Let $\lie{b}$ and $\lie{b'}$ be Borel subalgebras of $\lie{g}$ and $\lie{g'}$, respectively.
Recall that $\calO_{\lie{b}}$ is the BGG category $\calO$ of $\lie{g}$ with respect to $\lie{b}$.
See Subsection \ref{subsect:CategoriesGKCW}.

\begin{fact}[{\cite[Definition 3.1 and Proposition 3.5]{Ko12_generalized_Verma}}]\label{fact:DiscreteO}
	Assume $\lie{b} \cap \lie{g'} = \lie{b'}$.
	Then for any object $V \in \calO_{\lie{b}}$, the restriction $V|_{\lie{g'}}$ is discretely decomposable in $\calO_{\lie{b'}}$, i.e.\ any finitely generated $\lie{g'}$-submodule belongs to $\calO_{\lie{b'}}$.
\end{fact}

\begin{corollary}\label{cor:NonVanishingOO}
	Assume $\lie{b} \cap \lie{g'} = \lie{b'}$.
	Then for any object $V \in \calO_{\lie{b}}$,
	we have
	\begin{align*}
		\soc(V, \calO_{\lie{b'}}) &= V,\\
		\rad(V, \calO_{\lie{b'}}) &= 0.
	\end{align*}
	See Definition \ref{def:RadSocInC} for $\soc$ and $\rad$.
\end{corollary}

\begin{proof}
	The assertion of Fact \ref{fact:DiscreteO} can be rewritten as $\soc(V, \calO_{\lie{b'}}) = V$.
	We shall show the second equation.

	Let $M$ be a finitely generated non-zero $\lie{g'}$-submodule of $V$.
	Since $M$ belongs to $\calO_{\lie{b'}}$, there is an injective hull $M'$ of $M$ in $\calO_{\lie{b'}}$.
	Since any finitely generated $\lie{g'}$-submodule of $V$ is in $\calO_{\lie{b'}}$, the inclusion $M\hookrightarrow M'$ can be extended to a homomorphism $\varphi\colon V \rightarrow M'$.
	Hence we have $M \cap \rad(V, \calO_{\lie{b'}}) \subset M \cap \Ker(\varphi) = 0$.
	This shows $\rad(V, \calO_{\lie{b'}}) = 0$.
\end{proof}

\subsubsection{Spherical pair}\label{subsubsect:SphericalPair}

Here we assume that $(G, G')$ is a spherical pair, i.e.\ we can take a Borel subgroup $B$ of $G_0$ such that $BG'_0$ is open dense in $G_0$.
We set
\begin{align*}
	P &:= \set{g \in G_0 : gBG'_0 = BG'_0}, \\
	L &:= (P \cap G')_0.
\end{align*}
By Fact \ref{fact:BLV}, $P$ is a parabolic subgroup of $G_0$ and $L$ is a reductive subgroup of $G'$
containing the derived group of a Levi subgroup of $P$.
Then $(B\cap G')_0$ is a Borel subgroup of $L$.

By Proposition \ref{prop:BLVreal}, we can choose the Borel subgroup $B$ such that
$\lieR{l}:=\lie{l} \cap \lieR{g}$ is a real form of $\lie{l}$.
Let $L_\RR$ be the analytic subgroup of $G'_\RR$ with the Lie algebra $\lieR{l}\subset \lieR{g'}$, which is closed in $G'_\RR$.
If necessary, replacing $B$ with $gBg^{-1}$ for some $g \in G'_\RR$, we can assume that $K_{L,\RR}:= L_\RR \cap K'_\RR$ is a maximal compact subgroup of $L_\RR$.
Let $K_L$ be the analytic subgroup of $K'$ with the Lie algebra $\lie{k}_L \subset \lie{k'}$, which is closed in $K'$.
Fix a Levi subgroup $\widetilde{L}$ of $P$ containing $L$.

\begin{lemma}\label{lem:ReductionToFiber}
	Let $M$ be a closed subgroup of $K_L$.
	Then for any $(\lie{g'}, M)$-module $V'$ and $(\widetilde{\lie{l}}, M)$-module $V_L$, we have natural isomorphisms
	\begin{align*}
		\Hom_{\lie{g'}, M}(\ind{g}{p}(V_L), V') &\simeq \Hom_{\lie{l}, M}(V_L, V') \\
		(\ind{g}{p}(V_L)\otimes_{\univ{g'}} V')^{M} &\simeq (V_L \otimes_{\univ{l}} V')^M
	\end{align*}
	of vector spaces.
\end{lemma}

\begin{proof}
	By $\lie{g'} + \lie{p} = \lie{g}$ and $\lie{g'}\cap \lie{p} = \lie{l}$,
	we have $\ind{g}{p}(V_L)|_{\lie{g'}, M} \simeq \ind{g'}{l}(V_L)$.
	Hence we obtain the desired isomorphisms.
\end{proof}

The subalgebras $\widetilde{\lie{b}}_{L} := \widetilde{\lie{l}} \cap \lie{b}$
and $\lie{b}_L := \lie{l} \cap \lie{b}$ are Borel subalgebras of $\widetilde{\lie{l}}$ and $\lie{l}$, respectively.
Since $\lie{l}$ is the direct sum of its center and $[\widetilde{\lie{l}}, \widetilde{\lie{l}}]$, the restriction functors
\begin{align*}
	(\cdot)|_{\lie{l}, K_L}\colon &\calC_{\widetilde{\lie{l}}, K_L} \rightarrow \calC_{\lie{l}, K_L}, \\
	(\cdot)|_{\lie{l}}\colon &\calO_{\widetilde{\lie{b}}_{L}}\rightarrow \calO_{\lie{b}_L}
\end{align*}
are essentially surjective.
Here we say that a functor $F\colon \calC\rightarrow \calD$
is essentially surjective if for any object $V \in \calD$, there exists some object $W \in \calC$ such that $F(W) \simeq V$.

\begin{lemma}\label{lem:NonVanishingSpherical}
	Let $\lie{b'}$ be a Borel subalgebra of $\lie{g'}$.
	\begin{enumparen}
		\item For any non-zero $V' \in \Mod_{fl}(\lie{g'}, K')$, we have
		\begin{align*}
			\SupDim((\ind{g}{p}(\calC_{\lie{\widetilde{l}}, K_L})\otimes_{\univ{g}}\ind{g}{g'}(V'))^{K_L}) &= \SupDim((\calC_{\lie{l}, K_L}\otimes_{\univ{l}} V')^{K_L}) \neq 0.
		\end{align*}
		\item For any non-zero $V' \in \Mod_{fl}(\lie{g'}, K')$, we have
		\begin{align*}
			\SupDim(\ind{g}{p}(\calO_{\lie{\widetilde{b}}_L})\otimes_{\univ{g}} \ind{g}{g'}(V')) &= \SupDim(\calO_{\lie{b}_L}\otimes_{\univ{l}} V').
		\end{align*}
		If $\lie{b}_L + \lie{k}_L = \lie{l}$, the both sides are non-zero.
		\item For any non-zero $V' \in \calO_{\lie{b'}}$, we have
		\begin{align*}
			\SupDim(\Hom_{\lie{g}}(\ind{g}{p}(\calO_{\widetilde{\lie{b}}_L}), \pro{\lie{g}}{\lie{g'}}(V'))) = \SupDim(\Hom_{\lie{l}}(\calO_{\lie{b}_L}, V')).
		\end{align*}
		If $\lie{b'} \cap \lie{l} = \lie{b}_L$, the both sides are non-zero.
	\end{enumparen}
\end{lemma}

\begin{proof}
	By Lemma \ref{lem:ReductionToFiber}, for any $\lie{\widetilde{l}}$-module $V_L$, we have
	\begin{align*}
		\ind{g}{p}(V_L)\otimes_{\univ{g}} \ind{g}{g'}(V') &\simeq \ind{g}{p}(V_L)\otimes_{\univ{g'}} V' \\ 
		&\simeq V_L\otimes_{\univ{l}} V', \\
		\Hom_{\lie{g}}(\ind{g}{p}(V_L), \pro{\lie{g}}{\lie{g'}}(V')) & \simeq \Hom_{\lie{g'}}(\ind{g}{p}(V_L), V') \\
		&\simeq \Hom_{\lie{l}}(V_L, V').
	\end{align*}
	This shows the equations in the assertion.
	The non-vanishing parts follow from Lemma \ref{lem:ExistenceQuotRest} and Corollary \ref{cor:NonVanishingOO}.
\end{proof}

$(\calC, \calF^\calC) = (\calC_{\widetilde{\lie{l}}, K_L}, \ind{g}{p}), (\calO_{\widetilde{\lie{b}}_L}, \ind{g}{p})$ do not satisfy the assumption in Subsection \ref{subsect:CategoryWithGaction}.
In fact, tensoring with finite-dimensional $G$-modules does not commutes with the functor $\calF^\calC$.
We prepare good categories satisfying this condition.

\begin{definition}\label{def:OInd}
	Let $M$ be a closed subgroup of $K_L$ and $\calC'$ a full subcategory of $\Mod(\lie{\widetilde{l}}, M)$.
	We denote by $\Oind{g}{p}(\calC')$ the full subcategory of $\Mod(\lie{g}, M)$ whose object is isomorphic to a subquotient of $F\otimes \ind{g}{p}(V)$ for some $F \in \Mod_{fd}(\lie{g}, M)$ and $V \in \calC'$.
\end{definition}

Note that we can see $\Oind{g}{p}(\calO_{\widetilde{\lie{b}}_L}) = \calO_{\lie{b}}$ by the proof of Lemma \ref{prop:FiniteLoewyO}.

\begin{lemma}\label{lem:OInd}
	Set $(M, \calC') := (K_L, \calC_{\widetilde{l}, K_L})$ or $(1, \calO_{\widetilde{\lie{b}}_L})$.
	Then
	\begin{enumparen}
		\item\label{enum:OIndClosedTensor} the category $\Oind{g}{p}(\calC')$ is an abelian category closed under tensoring with objects in $\Mod_{fd}(\lie{g}, M)$
		\item\label{enum:OIndUniformlyBounded} any family of objects in $\Oind{g}{p}(\calC')$ with bounded lengths is uniformly bounded
		\item\label{enum:OIndLoewy} $\Loewy{\Oind{g}{p}(\calC')} < \infty$
		\item\label{enum:OIndandInd} there exists some constant $C > 0$ such that
		\begin{align*}
			\SupDim((\Oind{g}{p}(\calC')\otimes_{\univ{g}} V)^M) &\leq \SupDim((\ind{g}{p}(\calC')\otimes_{\univ{g}} V)^M) \\
			&\leq C\cdot \SupDim((\Oind{g}{p}(\calC')\otimes_{\univ{g}} V)^M),\\
			\SupDim(\Hom_{\lie{g}, M}(\Oind{g}{p}(\calC'), V)) &\leq \SupDim(\Hom_{\lie{g}, M}(\ind{g}{p}(\calC'), V) \\
			& \leq C\cdot \SupDim(\Hom_{\lie{g}, M}(\Oind{g}{p}(\calC'), V))
		\end{align*}
		for any $V \in \Mod(\lie{g}, M)$.
	\end{enumparen}
\end{lemma}

\begin{proof}
	\ref{enum:OIndClosedTensor} is clear by definition.
	By Facts \ref{fact:BoundedGeneralizedVerma} and \ref{fact:UniformlyBoundedFamilyGK}, for any family $(V_i)_{i \in I}$ of objects in $\calC'$ with bounded lengths, $(\ind{g}{p}(V_i))_{i \in I}$ is uniformly bounded and hence has bounded lengths.
	This shows \ref{enum:OIndUniformlyBounded}, \ref{enum:OIndLoewy} and the upper bounds in \ref{enum:OIndandInd}.
	See Lemma \ref{lem:BoundLoewy} for \ref{enum:OIndLoewy}.
	By the same argument as for irreducible highest weight modules, we can see that any irreducible object in $\Oind{g}{p}(\calC')$ is isomorphic to a quotient of $\ind{g}{p}(V)$ for some irreducible $V \in \calC'$.
	This shows the lower bounds in \ref{enum:OIndandInd}.
\end{proof}

\subsection{Uniform boundedness and polynomial identity}

In this subsection, we shall give necessary and sufficient conditions for the uniform boundedness of multiplicities in the branching problem and harmonic analysis.

Retain the notation in Subsection \ref{sect:Setting}.
Let $\lie{b}$ and $\lie{b'}$ be Borel subalgebras of $\lie{g}$ and $\lie{g'}$, respectively.
We have shown
\begin{align*}
	\Loewy{\calC_{G_\RR}}, \Loewy{\calC_{G'_\RR}}, \Loewy{\calC_{\lie{g}, K}},
	\Loewy{\calC_{\lie{g'}, K'}}, \Loewy{\calO_{\lie{b}}}, \Loewy{\calO_{\lie{b'}}} < \infty
\end{align*}
in Propositions \ref{thm:FiniteLoewyGK} and \ref{prop:FiniteLoewyO},
and we have shown in Fact \ref{fact:UniformlyBoundedFamilyGK} that any family of objects in one of the above categories (without $\calC_{G_\RR}$ and $\calC_{G'_\RR}$) with bounded lengths is a uniformly bounded family.

We say that a $(\lie{g'}, K')$-module $V$ is discretely decomposable if any finitely generated submodule has finite length (see \cite[Definition 1.1]{Ko98_discrete_decomposable_3}).
We denote by $\Mod_{fl}(\lie{g}, K; \lie{g'}, K')$ the full subcategory of $\Mod_{fl}(\lie{g}, K)$ whose objects are discretely decomposable as $(\lie{g'}, K')$-modules.

\begin{theorem}\label{thm:BoundedRestriction}
	Let $(\widetilde{\calC}, \calC, F)$ be one of the tuples in Table \ref{table:BoundedRestriction}.
	Let $V$ be a non-zero object in $\widetilde{\calC}$ with annihilator $I \subset \univ{g}$.
	Then there exists some constant $C > 0$ independent of $V$ such that
	\begin{align*}
		\PIdeg(\univ{g}^{G'}/I \cap \univ{g}^{G'})
		&\leq \SupDim(F(V, \calC)) \\
		&\leq C\cdot \Len_{\widetilde{C}}(V)\cdot \PIdeg(\univ{g}^{G'}/I \cap \univ{g}^{G'}).
	\end{align*}
\end{theorem}

\newcounter{tenum}[table]
\newcommand{\titem}{\refstepcounter{tenum} \thetenum}

\begin{table}[htb]		
	\begin{center}
		\begin{tabular}{|c|l|l|l|l|}
			\hline
			&$\widetilde{\calC}$ & $\calC$ & $F(V, V')$ & condition \\ \hline \hline
			\titem \label{enum:BoundedRestrictionG}&$\CWCat{G_\RR}$ & $\calC_{G'_\RR}$ & $\Hom_{G'_\RR}(V, V')$ & \\
			\titem \label{enum:BoundedRestrictionGK}&$\Mod_{fl}(\lie{g}, K)$ & $\calC_{\lie{g'}, K'}$ & $\Hom_{\lie{g'}, K'}(V, V')$ & \\
			\titem &$\Mod_{fl}(\lie{g}, K; \lie{g'}, K')$ & $\calC_{\lie{g'}, K'}$ & $\Hom_{\lie{g'}, K'}(V', V)$ & \\
			\titem &$\Mod_{fl}(\lie{g}, K)$ & $\calO_{\lie{b'}}$ & $V\otimes_{\univ{g'}}V'$ & $\lie{b'} + \lie{k'} = \lie{g'}$ \\
			\titem \label{enum:BoundedRestrictionOO}&$\calO_{\lie{b}}$ & $\calO_{\lie{b'}}$ & $\Hom_{\lie{g'}}(V, V')$ & $\lie{b}\cap \lie{g'} = \lie{b'}$ \\
			\titem \label{enum:BoundedRestrictionOOSub}&$\calO_{\lie{b}}$ & $\calO_{\lie{b'}}$ & $\Hom_{\lie{g'}}(V', V)$ & $\lie{b}\cap \lie{g'} = \lie{b'}$ \\
			\titem \label{enum:BoundedRestrictionOOTensor}&$\calO_{\lie{b}}$ & $\calO_{\lie{g'}}$ & $V\otimes_{\univ{g'}}V'$ & $\lie{b}\cap \lie{g'} = \lie{b'}$ \\ \hline
		\end{tabular}
		\caption{the restriction case} \label{table:BoundedRestriction}
	\end{center}
\end{table}

\begin{proof}
	We have explained in Subsection \ref{sect:Setting} why the estimate holds, and proved the non-vanishing results in the previous subsection.

	Remark that the upper bound for the case \ref{enum:BoundedRestrictionG} does not follow from Lemma \ref{lem:UpperBoundTorGeneral} directly.
	For $V' \in \CWCat{G'_\RR}$ and $V \in \CWCat{G_\RR}$, the restriction map
	\begin{align*}
		\Hom_{G'_\RR}(V, V') \rightarrow \Hom_{\lie{g'}, K'}(V_K, V'_{K'})
	\end{align*}
	is injective.
	Hence the upper bound for the case \ref{enum:BoundedRestrictionG}
	follows from that for the case \ref{enum:BoundedRestrictionGK}.

	The last assertion follows from Corollary \ref{cor:NonVanishingOO}.
	See also the proofs of Lemmas \ref{lem:AnnhilatorRad} and \ref{lem:AnnhilatorSoc}.
\end{proof}

We will show an analogue of Theorem \ref{thm:BoundedRestriction} for unitary representations in Section \ref{sect:Unitary}.
The following theorem is a version of Theorem \ref{thm:BoundedRestriction} for inductions.

\begin{theorem}\label{thm:BoundedInduction}
	Let $(\widetilde{\calC}, \calC, F)$ be one of the tuples in Table \ref{table:BoundedInduction} or in Table \ref{table:BoundedInductionSpherical} if $(G, G')$ is a spherical pair.
	In the latter case, we use the notation $\lie{p}, \widetilde{\lie{l}}, K_L$ and $\widetilde{\lie{b}}_L$ in Lemma \ref{lem:NonVanishingSpherical}.
	Let $V'$ be a non-zero object in $\widetilde{\calC}$ with annihilator $I \subset \univ{g'}$.
	Then there exists some constant $C > 0$ independent of $V'$ such that
	\begin{align*}
		\PIdeg(\univ{g}^{G'}/I')
		&\leq \SupDim(F(\calC, V')) \\
		&\leq C\cdot \Len_{\widetilde{\calC}}(V')\cdot \PIdeg(\univ{g}^{G'}/I'),
	\end{align*}
	where we set $I':= I \univ{g} \cap \univ{g}^{G'}$.
\end{theorem}

\begin{table}[htb]		
	\begin{center}
		\begin{tabular}{|c|l|l|l|}
			\hline
			&$\widetilde{\calC}$ & $\calC$ & $F(V, V')$ \\ \hline \hline
			\titem \label{enum:BoundedInductionG}& $\CWCat{G'_\RR}$ & $\calC_{G_\RR}$ & $\Hom_{G_\RR}(V, \Ind^{G_\RR}_{G'_\RR}(V'))$ \\
			\titem \label{enum:BoundedInductionGK}& $\Mod_{fl}(\lie{g'}, K')$ & $\calC_{\lie{g}, K}$ & $\Hom_{\lie{g}, K'}(V, \pro{\lie{g}, K'}{\lie{g'}, K'}(V'))$ \\
			\hline
		\end{tabular}
		\caption{the induction case} \label{table:BoundedInduction}
	\end{center}
\end{table}

\begin{table}[htb]		
	\begin{center}
		\begin{tabular}{|c|l|l|l|l|} 
			\hline
			\setcounter{tenum}{2}&$\widetilde{\calC}$ & $\calC$ & $F(V, V')$ & condition \\ \hline \hline
			\titem & $\Mod_{fl}(\lie{g'}, K')$ &$\Oind{g}{p}(\calC_{\widetilde{\lie{l}}, K_L})$ &  $(V\otimes_{\univ{g}} \ind{g}{g'}(V'))^{K_L}$ & \\
			\titem & $\Mod_{fl}(\lie{g'}, K')$ & $\calO_{\lie{b}}$ & $V\otimes_{\univ{g}} \ind{g}{g'}(V')$ & $\lie{k'} + \lie{b}_L = \lie{g'}$ \\
			\titem \label{enum:BoundedInductionOO} & $\calO_{\lie{b'}}$ & $\calO_{\lie{b}}$ & $\Hom_{\lie{g}}(V, \pro{\lie{g}}{\lie{g'}}(V'))$ & $\lie{b'} \cap \lie{l} = \lie{b}_L$\\ \hline
		\end{tabular}
		\caption{$(G, G')$ is a spherical pair} \label{table:BoundedInductionSpherical}
	\end{center}
\end{table}

\begin{remark}
	Although ${}^tI$ is more natural than $I$ to use in the inequation in some cases,
	there is no difference since we have already proved
	\begin{align*}
		\PIdeg(\univ{g}^{G'}/I\univ{g}\cap \univ{g}^{G'}) = \PIdeg(\univ{g}^{G'}/{}^t I\univ{g}\cap \univ{g}^{G'})
	\end{align*}
	in Lemma \ref{lem:CommuteTwoSidedIdeal}.
\end{remark}

\begin{proof}
	The outline of the proof is the same as Theorem \ref{thm:BoundedRestriction}.
	We shall explain the difference.
	
	To obtain the upper bound, we use the Frobenius reciprocity and its variations, e.g.
	\begin{align*}
		\Hom_{\lie{g}, K'}(V, \pro{\lie{g}, K'}{\lie{g'}, K'}(V'))
		\simeq \Hom_{\lie{g'}, K'}(V, V').
	\end{align*}
	Then the upper bound follows from Lemma \ref{lem:UpperBoundTorGeneral}.
	
	To obtain the lower bound, we let the algebra $\calA$ be $\rring{G/G'}\otimes \univ{g}$.
	See Subsection \ref{sect:Setting} for the notation.
	The annihilators of the $\calA$-modules $\Ind^{G_\RR}_{G'_\RR}(V')$, $\ind{g}{g'}(V')$ and $\pro{\lie{g}}{\lie{g'}}(V')$ are computed in Lemmas \ref{lem:AnnhilatorOUmod}, \ref{lem:AnnSubInd} and \ref{lem:AnnSubIndAlg}.
\end{proof}

\begin{corollary}\label{cor:ReductionToFiber}
	Let $V'$ be a non-zero $(\lie{g'}, K')$-module of finite length with annihilator $I \subset \univ{g'}$.
	If $\SupDim(\Hom_{\lie{g'}, K'}(\calC_{\lie{g}, K}, V')) < \infty$, then $(G, G')$ is a spherical pair.
	If $(G, G')$ is a spherical pair, there exists some constant $C > 0$ independent of $V'$ such that
	\begin{align*}
		\frac{\PIdeg((\univ{g'}/I)^{L})}{C\cdot \Len_{\lie{g'}, K'}(V')} &\leq \PIdeg((\univ{g}/I\univ{g})^{G'})\\
		&\leq C\cdot \Len_{\lie{g'}, K'}(V')\cdot \PIdeg((\univ{g'}/I)^{L}), \\
		\frac{\SupDim(\Hom_{\lie{l}, K_L}(V', \calC_{\lie{l}, K_L}))}{C\cdot \Len_{\lie{g'}, K'}(V')} &\leq \SupDim(\Hom_{\lie{g'}, K'}(\calC_{\lie{g}, K}, V')) \\
		&\leq C\cdot \Len_{\lie{g'}, K'}(V') \cdot \SupDim(\Hom_{\lie{l}, K_L}(V', \calC_{\lie{l}, K_L})).
	\end{align*}
	and use the notation $L, \lie{l}$ and $K_L$ in Subsubsection \ref{subsubsect:SphericalPair}.
	A similar inequation holds for $V'\in \calC_{G'_\RR}$ or $V'\in \calO_{\lie{b'}}$.
\end{corollary}

\begin{remark}
	We can relate $\univ{g'}^L$ to $\univ{g}^{G'}$ directly by a generalization of the Lepowsky homomorphism.
	Let $N$ be the unipotent radical of $P$.
	By $\lie{g} = \lie{g'} + \lie{p}$ and $\lie{g'}\cap \lie{p} = \lie{l}$, we can take an abelian subspace $\lie{a}$ of the center $\lieCent(\widetilde{\lie{l}})$ such that $\lie{g} = \lie{g'}\oplus \lie{a} \oplus \lie{n}$.
	Hence we obtain an $L$-module homomorphism by
	\begin{align*}
	\univ{g} \simeq \univ{a} \otimes \univ{g'} \oplus \lie{n}\univ{g} \rightarrow \univ{a} \otimes \univ{g'}.
	\end{align*}
	The homomorphism induces an anti-homomorphism $\univ{g}^{G'}\rightarrow \univ{a} \otimes \univ{g'}^L$ of algebras (see e.g.\ \cite[3.5.6]{Wa88_real_reductive_I}).
	This homomorphism is compatible with the isomorphism in Lemma \ref{lem:ReductionToFiber}.
\end{remark}

\begin{proof}
	By Lemma \ref{prop:boundedAndSphericalGH}, $\PIdeg((\univ{g}/I\univ{g})^{G'})$ is finite only if $G/G'$ is spherical.
	Hence if $\SupDim(\Hom_{\lie{g'}, K'}(\calC_{\lie{g}, K}, V')) < \infty$,
	then $G/G'$ is spherical by Theorem \ref{thm:BoundedInduction}.

	Suppose that $G/G'$ is $G$-spherical.
	By Theorems \ref{thm:BoundedRestriction} and \ref{thm:BoundedInduction}, we can take a constant $C > 0$ independent of $V'$ such that
	\begin{align}
		\PIdeg((\univ{g}/I\univ{g})^{G'})
		&\leq \SupDim(\Hom_{\lie{g'}, K'}(\calC_{\lie{g}, K}, V')) \\
		&\leq C\cdot \Len_{\lie{g'}, K'}(V') \cdot \PIdeg((\univ{g}/I\univ{g})^{G'}), \\ \label{eqn:ReductionToFiber1}\\
		\PIdeg((\univ{g}/I\univ{g})^{G'})
		&\leq \SupDim((\ind{g}{p}(\calC_{\widetilde{\lie{l}}, K_L})\otimes_{\univ{g'}}V')^{K_L}) \\
		&\leq C\cdot \Len_{\lie{g'}, K'}(V') \cdot \PIdeg((\univ{g}/I\univ{g})^{G'}), \\ \label{eqn:ReductionToFiber2}\\
		\PIdeg((\univ{g'}/I)^{L})
		&\leq \SupDim(\Hom_{\lie{l}, K_L}(V', \calC_{\lie{l}, K_L})) \\
		&\leq C\cdot \Len_{\lie{g'}, K'}(V') \cdot \PIdeg((\univ{g'}/I)^{L}). \label{eqn:ReductionToFiber3}
	\end{align}
	Here we use Lemma \ref{lem:OInd} in the second inequation.
	By Lemma \ref{lem:NonVanishingSpherical}, we have
	\begin{align*}
		\SupDim((\ind{g}{p}(\calC_{\widetilde{\lie{l}}, K_L})\otimes_{\univ{g'}}V')^{K_L}) &= \SupDim((\calC_{\lie{l}, K_L}\otimes_{\univ{l}}V')^{K_L}) \\
		&= \SupDim(\Hom_{\lie{l}, K_L}(V', \calC_{\lie{l}, K_L})).
	\end{align*}
	Hence by \eqref{eqn:ReductionToFiber1} and \eqref{eqn:ReductionToFiber2}, we obtain the desired inequation
	\begin{align*}
		\frac{\SupDim(\Hom_{\lie{l}, K_L}(V', \calC_{\lie{l}, K_L}))}{C \cdot \Len_{\lie{g'}, K'}(V')}
		&\leq \SupDim(\Hom_{\lie{g'}, K'}(\calC_{\lie{g}, K}, V')) \\
		&\leq C \cdot \Len_{\lie{g'}, K'}(V') \cdot \SupDim(\Hom_{\lie{l}, K_L}(V', \calC_{\lie{l}, K_L})).
	\end{align*}
	Similarly, the first one follows from \eqref{eqn:ReductionToFiber2} and \eqref{eqn:ReductionToFiber3}.
\end{proof}

We may sometimes consider reducible modules such as principal series representations instead of irreducible modules.
The following corollary is useful to show $\SupDim(F(W, \calC')) < \infty$ for an irreducible subquotient $W$ of a reducible module.

\begin{corollary}\label{cor:BoundedSubquotient}
	Under the notation in Theorem \ref{thm:BoundedRestriction}, we have
	\begin{align*}
		\SupDim(F(V, \calC)) \leq \sum_{W} \SupDim(F(W, \calC)) \leq C \cdot \Len_{\calC}(V) \cdot \PIdeg((\univ{g}/I)^{G'}),
	\end{align*}
	where the sum is over all composition factors $W$ of $V$ with multiplicity.
	A similar inequation holds for the setting of Theorem \ref{thm:BoundedInduction}.
\end{corollary}

\begin{proof}
	The first inequality follows from the left or right exactness of the functor $F$.
	Applying Theorem \ref{thm:BoundedRestriction} to a subquotient $W$ of $V$, we obtain
	\begin{align*}
		\SupDim(F(W, \calC)) \leq C\cdot \PIdeg((\univ{g}/I)^{G'})
	\end{align*}
	since the annihilator of $W$ contains $I$.
	This shows the second inequality.
\end{proof}

We have seen in Lemma \ref{lem:UpperBoundTorGeneral} that $\dim_{\CC}(\Tor_i^{\univ{g'}}(V, V'))$ can be bounded by the PI degree.
The following corollary gives sufficient conditions for the uniform boundedness of the cohomological multiplicities.

\begin{corollary}\label{cor:BoundedHomology}
	If $V \in \Mod_{fl}(\lie{g}, K)$ satisfies $\SupDim (\Hom_{\lie{g'}, K'}(V, \calC_{\lie{g'}, K'})) < \infty$, then we have
	\begin{align*}
		\SupDim(H_i(\lie{g'}, K'; V\otimes \calC_{\lie{g'}, K'})) &< \infty, \\
		\SupDim(\Ext_{\lie{g'}, K'}^i(V, \calC_{\lie{g'}, K'})) &< \infty
	\end{align*}
	for any $i \geq 0$.
	If $V' \in \Mod_{fl}(\lie{g'}, K')$ satisfies $\SupDim(\Hom_{\lie{g'}, K'}(\calC_{\lie{g}, K}, V')) < \infty$, then we have
	\begin{align*}
		\SupDim(H_i(\lie{g'}, K'; \calC_{\lie{g}, K} \otimes V')) &< \infty, \\
		\SupDim(\Ext_{\lie{g'}, K'}^i(\calC_{\lie{g}, K}, V')) &< \infty
	\end{align*}
	for any $i \geq 0$.
\end{corollary}

\begin{proof}
	Remark that there is a natural isomorphism
	\begin{align*}
		H_i(\lie{g'}, K'; V\otimes V')^* \simeq \Ext_{\lie{g'}, K'}^i(V, (V')^*_{K'})
	\end{align*}
	for $V \in \Mod_{fl}(\lie{g}, K)$ and $V' \in \Mod_{fl}(\lie{g'}, K')$.
	The corollary follows easily from Lemma \ref{lem:UpperBoundTorGeneral}, Theorems \ref{thm:BoundedRestriction} and \ref{thm:BoundedInduction}.
	Note that $\SupDim(\Hom_{\lie{g'}, K'}(\calC_{\lie{g}, K}, V')) < \infty$ if and only if $\SupDim(\Hom_{\lie{g'}, K'}(\calC_{\lie{g}, K},  (V')^*_{K'})) < \infty$ by Lemma \ref{lem:CommuteTwoSidedIdeal} and Theorem \ref{thm:BoundedInduction}.
\end{proof}

Let $\widetilde{\calC}$ be a full subcategory of $\Mod_{fl}(\lie{g'})$
closed under taking subquotients in $\Mod_{fl}(\lie{g'})$.
Assume that any family of objects in $\widetilde{\calC}$ with bounded lengths is a uniformly bounded family of $\lie{g'}$-modules.
The following theorem is a version of Theorem \ref{thm:BoundedInduction} for the category $\widetilde{\calC}$.

\begin{theorem}\label{thm:BoundedInductionGeneral}
	Let $V'$ be a non-zero object in $\widetilde{\calC}$ with annihilator $I \subset \univ{g'}$.
	Then there exists some constant $C > 0$ independent of $V'$ such that
	\begin{align*}
		\PIdeg(\univ{g}^{G'}/I')
		&\leq \SupDim(V' \otimes_{\univ{g'}} \calO_{\lie{g}}) \\
		&\leq C\cdot \Len_{\widetilde{\calC}}(V') \cdot \PIdeg(\univ{g}^{G'}/I'),
	\end{align*}
	where we set $I':= I \univ{g} \cap \univ{g}^{G'}$.
	See Subsection \ref{subsect:CategoriesGKCW} for the category $\calO_{\lie{g}}$.
\end{theorem}

\begin{proof}
	The outline of the proof is the same as Theorem \ref{thm:BoundedInduction}.
	We need to show that $\SupDim(V' \otimes_{\univ{g'}} \calO_{\lie{g}})$ is non-zero.

	Since $\ind{g}{g'}(V')$ is a finitely generated $\lie{g}$-module,
	we can take a Borel subalgebra $\lie{b}$ of $\lie{g}$ and a character $\lambda$ of $\lie{b}$ such that
	\begin{align*}
		\ind{g}{g'}(V') \otimes_{\univ{g'}}\ind{g}{b}(\CC_{\lambda}) \simeq  \ind{g}{g'}(V')\otimes_{\univ{b}}\CC_{\lambda}\neq 0
	\end{align*}
	as we have seen in the proof of Proposition \ref{prop:multiplicityNotSpherical}.
	This implies $\SupDim(V' \otimes_{\univ{g'}} \calO_{\lie{g}}) \neq 0$.
\end{proof}

\section{Parabolic induction}\label{section:parabolicInduction}

In this section, we consider the branching problem of parabolically induced representations.
In Corollary \ref{cor:ReductionToFiber}, we have shown that the supremum of multiplicities in $\Ind^{G_\RR}_{G'_\RR}(V')$ can be bounded by that of the fiber $V'$.
We shall show a similar result to Corollary \ref{cor:ReductionToFiber} for branching laws of parabolically induced representations.

\subsection{Annihilator}

To study the annihilator of a parabolically induced representation, we shall prove several lemmas about the annihilator of a generalized Verma module.

Let $G$ be a connected reductive algebraic group over $\CC$ and $G'$ a connected reductive subgroup of $G$.
Take a parabolic subgroup $P$ of $G$.

We assume that $G/P$ is a $G'$-spherical variety.
If necessary, replacing $G'$ with its conjugate, we can assume that there is a Borel subgroup $B'$ of $G'$ such that $B'P$ is open in $G$.
Put
\begin{align*}
P' &:= \set{g \in G': gB'P = B'P},\\
L' &:= (P'\cap P)_0.
\end{align*}
By Fact \ref{fact:BLV}, $L'$ is a reductive subgroup of $G'$ containing
the derived group of a Levi subgroup of $P'$.
Let $M'$ be a Levi subgroup of $P'$ containing $L'$.
Since $L' \subset P$ is reductive, we can take a Levi subgroup $M$ of $P$ containing $L'$.
Then we have
\begin{align*}
	\lie{g} = \lie{p} + \lie{p'}, \quad \lie{p}\cap \lie{p'} = \lie{l'}, \quad [\lie{m'}, \lie{m'}] \subset \lie{l'} \subset \lie{m'}.
\end{align*}

\begin{lemma}\label{lem:ReductionToFiberParabolic}
	Let $V$ be an $\lie{m}$-module and $V'$ an $\lie{m'}$-module.
	Then there exists a natural isomorphism
	\begin{align*}
		\ind{g}{p}(V)\otimes_{\univ{g'}} \ind{g'}{p'}(V') \simeq V\otimes_{\lie{l'}} V'
	\end{align*}
	of vector spaces.
\end{lemma}

\begin{proof}
	By $\lie{g} = \lie{p} + \lie{p'}$ and $\lie{p}\cap \lie{p'} = \lie{l'}$,
	we have a natural isomorphism $\ind{g}{p}(V)|_{\lie{p'}} \simeq \ind{p'}{l'}(V)$.
	This shows the lemma.
\end{proof}

\begin{lemma}\label{lem:NonVanishingParabolicInd}
	Let $V$ be an $\lie{m}$-module.
	If any non-zero $\lie{l'}$-submodule of $V$ has an irreducible quotient,
	then for any non-zero $\lie{g}$-submodule $W$ of $\ind{g}{p}(V)$, there exists some object $V' \in \calO_{\lie{g'}}$ such that $W\otimes_{\univ{g'}} V' \neq 0$.
\end{lemma}

\begin{proof}
	Using the isomorphism $\ind{g}{p}(V)|_{\lie{p'}}\simeq \ind{p'}{l'}(V)$ in the proof of Lemma \ref{lem:ReductionToFiberParabolic}, we regard $W$ as a $\lie{p'}$-submodule of $\ind{p'}{l'}(V)$.
	Let $U$ be the intersection of all $\lie{l'}$-submodules $X$ of $V$ such that $W \subset \ind{p'}{l'}(X)$.
	By assumption, the $\lie{l'}$-module $U$ has an irreducible quotient $U'$.
	Then we obtain a non-zero $\lie{p'}$-module homomorphism $\varphi \colon W\rightarrow \ind{p'}{l'}(U')$.

	By $[\lie{m'}, \lie{m'}] = [\lie{l'}, \lie{l'}]$, we can take a filtration $\ind{p'}{l'}(U') = F_0 \supset F_1 \supset \cdots$ such that $\bigcap_i F_i = 0$ and $F_i/F_{i+1}$ is an irreducible $\lie{p'}$-module on which the nilpotent radical of $\lie{p'}$ acts trivially for any $i \geq 0$.
	Take $i\geq 0$ such that $\varphi(W)\subset F_i$ and $\varphi(W) \not\subset F_{i+1}$.
	Then $F_i/F_{i+1}$ is an irreducible $\lie{p'}$-module quotient of $W$.

	By a generalization of the Casselman subrepresentation theorem \cite[Theorem 1]{BeBe83}, there exists a Borel subalgebra $\lie{b}_{M'}$ of $\lie{m'}$ and a character $\lambda$ of $\lie{b}_{M'}$ such that
	$F_i/F_{i+1}\otimes_{\univ{b_{\mathit{M'}}}} \CC_{\lambda} \neq 0$.
	Therefore $V':= \ind{g'}{p'}(\ind{m'}{b_{\mathit{M'}}}(\CC_{\lambda}))$ satisfies the desired condition:
	\begin{equation*}
		W \otimes_{\univ{g'}} \ind{g'}{p'}(\ind{m'}{b_{\mathit{M'}}}(\CC_{\lambda})) \simeq W\otimes_{\univ{b_{\mathit{M'}}}} \CC_{\lambda} \neq 0. \qedhere
	\end{equation*}
\end{proof}

Let $B_M$ be a Borel subgroup of $M$ such that $B_M \cap L'$ is a Borel subgroup of $L'$.
By Fact \ref{fact:DiscreteO}, for any $V \in \calO_{\lie{b}_M}$, 
the restriction $V|_{\lie{l'}}$ is discretely decomposable.
Hence any non-zero $\lie{l'}$-submodule of $V$ has an irreducible quotient (see also Corollary \ref{cor:NonVanishingOO}).
By Lemma \ref{lem:NonVanishingParabolicInd}, we can apply Lemma \ref{lem:lowerboundTensor} to
\begin{align*}
	(\calA, \calB, \calC, \widetilde{\calC}, V) = (\univ{g}, \Mod(\lie{g'}), \calO_{\lie{g'}}, \Mod_{fl}(\lie{g}), \ind{g}{p}(V)).
\end{align*}
The following theorem is an analogue of Corollary \ref{cor:ReductionToFiber} for parabolic inductions.

\begin{theorem}\label{thm:pidegAnnihilatorVerma}
	Let $I$ be a proper two-sided ideal of $\univ{m}$ such that $\univcent{m}\cap I$ has finite codimension in $\univcent{m}$.
	Fix primitive ideals $I_1, I_2, \ldots, I_n$ containing $I$ and satisfying $I_1 I_2 \cdots I_n \subset I$ as in Lemma \ref{lem:TwosidedPrimitive}.
	Set $J:=\Ann_{\univ{g}}(\ind{g}{p}(\univ{m}/I))$, where $\univ{m}/I$ is regarded as an $\lie{m}$-module via the left multiplication.
	Then there exists some constant $C > 0$ independent of $I$ such that
	\begin{align*}
		C^{-1}\cdot \PIdeg((\univ{m}/I)^{L'}) \leq \PIdeg((\univ{g}/J)^{G'})
		\leq C\cdot \PIdeg((\univ{m}/I)^{L'}).
	\end{align*}
\end{theorem}

\begin{proof}
	Set $J_i:=\Ann_{\univ{g}}(\ind{g}{p}(\univ{m}/I_i))$.
	Then we have $J_1 J_2\cdots J_n \subset J$ by the short exact sequence
	\begin{align*}
		0\rightarrow \ind{g}{p}(A_{i+1}/A_i)
		\rightarrow \ind{g}{p}(\univ{m}/A_i) \rightarrow \ind{g}{p}(\univ{m}/A_{i+1})\rightarrow 0,
	\end{align*}
	where $A_i := I_i I_{i+1}\cdots I_n$.
	If the theorem holds for any primitive ideal of $\univ{m}$, we have
	\begin{align*}
		\PIdeg((\univ{m}/I)^{L'}) &= \max_i\set{\PIdeg((\univ{m}/I_i)^{L'})} \\
		&\leq C\cdot \max_i\set{ \PIdeg((\univ{g}/J_i)^{G'})}\\
		&= C\cdot \PIdeg((\univ{g}/J)^{G'}), \\
		\PIdeg((\univ{g}/J)^{G'}) &= \max_i\set{\PIdeg((\univ{g}/J_i)^{G'})} \\
		&\leq C\cdot \max_i\set{\PIdeg((\univ{m}/I_i)^{L'})}\\
		&= C\cdot \PIdeg((\univ{m}/I)^{L'})
	\end{align*}
	by Proposition \ref{prop:pidegIdeal} \ref{item:pidegIdealProduct}.
	Hence it is enough to show the assertion for primitive $I$.

	By Duflo's theorem (Fact \ref{fact:DufloIdeal}), we can take an irreducible highest weight module $V$ of $\lie{m}$ with respect to $\lie{b}_M$ such that $\Ann_{\univ{m}}(V) = I$.
	Then we have $J= \Ann_{\univ{g}}(\ind{g}{p}(V))$ by Lemma \ref{lem:AnnhilatorOUmod}.

	By Fact \ref{fact:BoundedGeneralizedVerma}, the uniform boundedness of a family of modules is preserved by parabolic inductions.
	Hence, by Lemmas \ref{lem:UpperBoundTorGeneral} and \ref{lem:lowerboundTensor},
	there exists some constant $C' > 0$ independent of (primitive) $I$ such that
	\begin{align*}
		\PIdeg((\univ{g}/J)^{G'}) \leq \SupDim(\ind{g}{p}(V)\otimes_{\univ{g'}}\calO_{\lie{g'}}) \leq C'\cdot \PIdeg((\univ{g}/J)^{G'}).
	\end{align*}
	By Lemma \ref{lem:OInd} \ref{enum:OIndandInd}, we can replace $\calO_{\lie{g'}}$ with $\ind{g'}{p'}(\calO_{\lie{m'}})$.

	By Lemma \ref{lem:ReductionToFiberParabolic}, we have
	\begin{align*}
		\SupDim(\ind{g}{p}(V)\otimes_{\univ{g'}}\ind{g'}{p'}(\calO_{\lie{m'}})) = \SupDim(V\otimes_{\univ{l'}}\calO_{\lie{l'}}),
	\end{align*}
	and by Theorem \ref{thm:BoundedRestriction}, we can take a constant $C'' > 0$ independent of $V$ such that
	\begin{align*}
		\PIdeg((\univ{m}/I)^{L'}) \leq \SupDim(V\otimes_{\univ{l'}}\calO_{\lie{l'}}) \leq C''\cdot \PIdeg((\univ{m}/I)^{L'}).
	\end{align*}
	Therefore we obtain
	\begin{equation*}
		\frac{\PIdeg((\univ{m}/I)^{L'})}{C'} \leq \PIdeg((\univ{g}/J)^{G'})
		\leq C''\cdot \PIdeg((\univ{m}/I)^{L'}).
		\qedhere
	\end{equation*}
\end{proof}

\begin{remark}
We can relate $\univ{g}^{G'}$ to $\univ{m}^{L'}$ by an analogue of the Harish-Chandra homomorphism as follows.
Let $U'$ (resp.\ $U$) be the unipotent radical of $P'$ (resp.\ $P$)
and $\lie{a'}$ a complement of $\lieCent(\lie{l'})$ in $\lieCent(\lie{m'})$.
Then since $\lie{p'}+\lie{p} = \lie{g}$ and $\lie{p'}\cap \lie{p} = \lie{l'}$,
we have
\begin{align*}
\lie{g} = \lie{u'}\oplus \lie{a'} \oplus \lie{m} \oplus \lie{u}.
\end{align*}
Since $\univ{g}\simeq \univ{a'}\otimes \univ{m}\oplus (\lie{u'}\univ{g}+\univ{g}\lie{u})$ by the Poincar\'e--Birkhoff--Witt theorem,
the projection onto $\univ{a'}\otimes \univ{m}$ gives a linear map
\begin{align*}
\univ{g}^{G'} \rightarrow \univ{a'}\otimes \univ{m}^{L'}.
\end{align*}
It is easy to see that this is an algebra homomorphism.
The homomorphism is compatible with the isomorphism in Lemma \ref{lem:ReductionToFiberParabolic}.
\end{remark}

\subsection{Real parabolic induction}

In this subsection, we study branching laws of parabolically induced representations using Theorem \ref{thm:pidegAnnihilatorVerma}.
We deal only with Casselman--Wallach representations.
Similar results hold for $(\lie{g}, K)$-modules.

Let $G_\RR$ be a connected reductive Lie group and $G'_\RR$ a connected reductive subgroup of $G_\RR$.
Assume that $\lieR{g'}$ is algebraic in $\lieR{g}$
and fix connected reductive algebraic groups $G$ and $G'$ as in Subsection \ref{sect:Setting}.

We take a parabolic subgroup $P_{\RR}$ of $G_{\RR}$ and a Levi subgroup $M_{\RR}$ of $P_{\RR}$.
Assume that $G/P$ is a spherical $G'$-variety, where $P$ is the analytic subgroup of $G$
with the Lie algebra $\lie{p}\subset \lie{g}$.
If necessary, replacing $P_{\RR}$ with its conjugate in $G_{\RR}$, we can assume that $G'_{\RR}P_{\RR}$
is open in $G_{\RR}$.
Then by Proposition \ref{prop:BLVreal}, there exists a Borel subgroup $B' \subset G'$ satisfying the following conditions:
\begin{enumerate}
	\item $B'P$ is open dense in $G$
	\item the identity component $L'$ of $\set{g \in G'\cap P: g B'P = B'P}$ is reductive
	\item $\lieR{l'}:=\lie{l'}\cap \lieR{g'}$ is a real form of $\lie{l'}$.
\end{enumerate}

Let $L'_\RR$ denote the analytic subgroup of $G'_\RR$ with the Lie algebra $\lieR{l'}\subset \lieR{g'}$, which is closed in $G'_\RR$ by construction.
For a Casselman--Wallach representation $V$ of $M_\RR$, we consider the parabolically induced representation $\Ind^{G_{\RR}}_{P_{\RR}}(V)$.
Then $\Ind^{G_{\RR}}_{P_{\RR}}(V)$ is a Casselman--Wallach representation of $G_\RR$.
Let $K_{M,\RR}$ be a maximal compact subgroup of $M_\RR$ and $K_M$ the complexification of $K_{M, \RR}$.

\begin{lemma}\label{lem:annihilatorParabolicInduction}
Let $V$ be a Casselman--Wallach representation $V$ of $M_\RR$.
Set $W:=(V^*)_{K_{M}}$.
\begin{enumparen}
	\item We have $\Ann_{\univ{g}}(\Ind^{G_{\RR}}_{P_{\RR}}(V)) = {}^t \Ann_{\univ{g}}(\ind{g}{p}(W))$.
	\item If $\ind{g}{p}(W)$ is irreducible as a $(\lie{g}, K_{M})$-module
	and $X$ is a non-zero closed subrepresentation of $\Ind^{G_{\RR}}_{P_{\RR}}(V)$, then
	$\Ann_{\univ{g}}(X) = {}^t \Ann_{\univ{g}}(\ind{g}{p}(W))$.
\end{enumparen}
\end{lemma}

\begin{proof}
	We shall show the first assertion.
	The second assertion follows easily from the first assertion and its proof.
	Note that there is a canonical $(\lie{g}, K_{M})$-invariant pairing
	\begin{align*}
		(\cdot, \cdot)\colon \ind{g}{p}(W)\times \Ind^{G_{\RR}}_{P_{\RR}}(V) \rightarrow \CC
	\end{align*}
	defined by $(X\otimes \varphi, f) \mapsto \varphi({}^tXf(e))$.
	Since the pairing is non-degenerate on the left by Lemma \ref{lem:PairingNondegenerate}, we have
	\begin{align*}
		\Ann_{\univ{g}}(\Ind^{G_{\RR}}_{P_{\RR}}(V)) \subset {}^t \Ann_{\univ{g}}(\ind{g}{p}(W)).
	\end{align*}

	Take $f \in {}^t \Ann_{\univ{g}}(\ind{g}{p}(W))\Ind^{G_{\RR}}_{P_{\RR}}(V)$.
	Since $\Ann_{\univ{g}}(\ind{g}{p}(W))$ is $G_\RR$-stable, we have $\varphi((g\cdot f)(e)) = 0$ for any $g \in G_\RR$ and $\varphi \in W$.
	This implies $f = 0$ and hence
	\begin{equation*}
		\Ann_{\univ{g}}(\Ind^{G_{\RR}}_{P_{\RR}}(V)) \supset {}^t \Ann_{\univ{g}}(\ind{g}{p}(W)). \qedhere
	\end{equation*}
\end{proof}

\begin{theorem}\label{thm:ReductionToFiberParabolicInd}
	Let $V$ be a non-zero Casselman--Wallach representation of $M_\RR$
	and $X$ a non-zero closed subrepresentation of $\Ind^{G_{\RR}}_{P_{\RR}}(V)$.
	Then there exists some constant $C$ independent of $V$ and $X$ such that
	\begin{align*}
	\SupDim(\Hom_{G'_\RR}(X, \calC_{G'_\RR})) \leq C\cdot \Len_{M_\RR}(V) \cdot \SupDim(\Hom_{L'_\RR}(V, \calC_{L'_\RR})).
	\end{align*}
	If $\ind{g}{p}((V^*)_{K_M})$ is irreducible as a $(\lie{g}, K_M)$-module or $X = \Ind^{G_{\RR}}_{P_{\RR}}(V)$, then we can take the constant $C$ to satisfy
	\begin{align*}
		\frac{\SupDim(\Hom_{L'_\RR}(V, \calC_{L'_\RR}))}{C\cdot \Len_{M_\RR}(V)} &\leq \SupDim(\Hom_{G'_\RR}(X, \calC_{G'_\RR})) \\
		&\leq C\cdot \Len_{M_\RR}(V) \cdot \SupDim(\Hom_{L'_\RR}(V, \calC_{L'_\RR})).
	\end{align*}
\end{theorem}

\begin{remark}
	Suppose that $X$ satisfies the assumption of the last assertion.
	By Proposition \ref{prop:boundedAndSphericalGP} and Theorem \ref{thm:BoundedRestriction}, $\SupDim(\Hom_{G'_\RR}(X, \calC_{G'_\RR}))$ is finite only if $G/P$ is $G'$-spherical.
\end{remark}

\begin{proof}
	The theorem follows from Theorems \ref{thm:BoundedRestriction} and \ref{thm:pidegAnnihilatorVerma} and Lemma \ref{lem:annihilatorParabolicInduction}.
	Remark that for any irreducible representation $V$, the length of $\Ind^{G_\RR}_{P_\RR}(V)$ is bounded by a constant independent of $V$ by Fact \ref{fact:BoundedGeneralizedVerma}.
\end{proof}

Theorem \ref{thm:ReductionToFiberParabolicInd} is useful to classify branching laws of Casselman--Wallach representations with uniformly bounded multiplicities.
We can obtain a large family of such branching laws using the Langlands classification.
Similarly, we can reduce the classification of such branching laws of unitary representations to that of unitary representations with real infinitesimal characters (see \cite[Theorem 16.10]{Kn86}).

\subsection{Cohomological parabolic induction}

In this subsection, we consider branching laws of cohomological parabolic inductions.

Let $G$ be a connected reductive algebraic group and $G'$ a connected reductive subgroup of $G$.
Let $(\lie{g}, K)$ be a pair and $(\lie{g'}, K')$ a subpair of $(\lie{g'}, K')$.
Assume that $(\lie{g}, \lie{k})$ and $(\lie{g'}, \lie{k'})$ are symmetric pair,
and $K$ and $K'$ are connected.

Take a parabolic subalgebra $\lie{p}$ of $\lie{g}$ and a Levi subalgebra $\lie{l}$ of $\lie{p}$.
Fix a reductive subgroup $K_L$ of $N_K(\lie{p}) \cap N_K(\lie{l})$ whose Lie algebra is contained in $\lie{l}$.
Then $(\lie{l}, K_L)$ is a pair.
Suppose that $(\lie{l}, \lie{k}_L)$ is symmetric pair.
We do not need this assumption if we consider $K_L$-admissible $(\lie{l}, K_L)$-modules (or more generally `holonomic' $(\lie{l}, K_L)$-modules) instead of $(\lie{l}, K_L)$-modules below.
See \cite[Thereom 7.25]{Ki20}.

Let $P$ be the parabolic subgroup of $G$ with the Lie algebra $\lie{p}\subset \lie{g}$.
Assume that $G/P$ is $G'$-spherical, that is,
$B'xP$ is open dense in $G$ for some Borel subgroup $B' \subset G'$ and $x \in G$.
We put
\begin{align*}
	L':=(x^{-1}\set{g \in G': g B'xP = B'xP} x \cap P)_0.
\end{align*}
If necessary, replacing $x$ with an element in $xP$, we can assume $\lie{l'}\subset \lie{l}$.

We consider cohomologically induced modules $\Dzuck{K}{K_L}{i}(\ind{g}{p}(V))$
for $(\lie{l}, K_L)$-modules $V$.
See Definition\ref{def:Zuckerman} for the Zuckerman derived functors $\Dzuck{K}{K_L}{i}$, and \cite{KnVo95_cohomological_induction} for the theory of cohomologically induced modules.
We have shown in Fact \ref{fact:BoundedGeneralizedVerma} that if $\set{V_j}_{j\in J}$ is a uniformly bounded family of $(\lie{l}, K_L)$-modules,
then the family $\set{\Dzuck{K}{K_L}{i}(\ind{g}{p}(V_j))}_{i \in \NN, j \in J}$ is uniformly bounded.
In particular, if $V$ is irreducible, the length of $\Dzuck{K}{K_L}{i}(\ind{g}{p}(V))$ is bounded by a constant independent of $V$ and $i$.

It is hard to determine the annihilator of $\Dzuck{K}{K_L}{i}(\ind{g}{p}(V))$ in general.
We can give a lower bound of the annihilator by the following fact (see e.g.\ \cite[Proposition 3.77]{KnVo95_cohomological_induction}
and \cite[Lemma 6.3.3]{Wa88_real_reductive_I}).

\begin{fact}\label{fact:LowerBoundAnnihilatorCohInd}
Let $V$ be an $(\lie{l}, K_L)$-module.
Then we have
\begin{align*}
\Ann_{\univ{g}}(\Dzuck{K}{K_L}{i}(\ind{g}{p}(V))) \supset \Ann_{\univ{g}}(\ind{g}{p}(V)).
\end{align*}
\end{fact}

By Fact \ref{fact:LowerBoundAnnihilatorCohInd} and Theorems \ref{thm:BoundedRestriction} and \ref{thm:pidegAnnihilatorVerma}, we obtain the following theorem.

\begin{theorem}\label{thm:ReductionToFiberCohInd}
Let $V$ be a non-zero $(\lie{l}, K_L)$-module of finite length with annihilator $I \subset \univ{l}$ and fix $i \in \NN$.
Set $\calR^i(V):=\Dzuck{K}{K_L}{i}(\ind{g}{p}(V))$.
Then there exists some constant $C > 0$ independent of $V$ and $i$ such that
\begin{align*}
\SupDim(\Hom_{\lie{g'}, K'}(\calR^i(V), \calC_{\lie{g'}, K'}) \leq C\cdot \Len_{\lie{l}, K_L}(V) \cdot \PIdeg((\univ{l}/I)^{L'}).
\end{align*}
If, in addition, $\Ann_{\univ{g}}(\calR^i(V)) = \Ann_{\univ{g}}(\ind{g}{p}(V))$ holds, we can choose the constant $C$ to satisfy
\begin{align*}
C^{-1}\cdot \PIdeg((\univ{l}/I)^{L'}) \leq \SupDim(\Hom_{\lie{g'}, K'}(\calR^i(V), \calC_{\lie{g'}, K'})).
\end{align*}
\end{theorem}

In general, the assumption $\Ann_{\univ{g}}(\calR^i(V)) = \Ann_{\univ{g}}(\ind{g}{p}(V))$ does not hold
and hence the lower bound in the theorem may not hold.
The Borel--Weil--Bott construction of finite-dimensional representations gives an example.

We can show $\Ann_{\univ{g}}(\calR^i(V)) = \Ann_{\univ{g}}(\ind{g}{p}(V))$ in some special cases.
If $K=K_L$, the module $\calR^0(V)$ is $\ind{g}{p}(V)$ itself and hence the equality for the annihilators automatically holds.
The module is the underlying Harish-Chandra module of a holomorphic discrete series representation.
Theorem \ref{thm:ReductionToFiberCohInd} says that $\SupDim(\Hom_{\lie{g'}, K'}(\ind{g}{p}(V), \calC_{\lie{g'}, K'}))$ is finite if and only if $G/P$ is $G'$-spherical.
If $(G, G')$ is a symmetric pair, more precise results are known (see \cite{Ko05,Ko08,Ko12_generalized_Verma} and \cite{Ki14,Ki16}).

The equation for the annihilators holds also for quaternionic discrete series representations induced from suitable parabolic subalgebras $\lie{p}$.
In the case, it is known in \cite[Proposition 5.7]{GrWa96} that the Gelfand--Kirillov dimension of $\calR^i(V)$ is equal to $\dim_{\CC}(G/P)$.
This implies that $\Ann_{\univ{g}}(\calR^i(V))$ coincides with $\Ann_{\univ{g}}(\ind{g}{p}(V))$ by \cite[Korollar 3.7]{BoHa76}.

The following corollary supports Kobayashi's conjecture for multiplicity-freeness of branching laws of $A_{\lie{q}}(\lambda)$ with sufficiently regular $\lambda$
\cite[Conjecture 4.2]{Ko11}
and gives an affirmative answer to \cite[Conjecture 4.3]{Ko11}.

\begin{corollary}\label{cor:kobayashiConj}
	Let $V$ be a non-zero finite-dimensional irreducible $(\lie{l}, K_L)$-module.
	Fix $i \in \NN$.
	Suppose that the nilpotent radical of $\lie{p}$ is abelian and $(G, G')$ is a symmetric pair.
	Then there exists some constant $C > 0$ independent of $V$ and $i$ such that
	\begin{align*}
		\SupDim(\Hom_{\lie{g'}, K'}(\calR^i(V), \calC_{\lie{g'}, K'}))\leq C \cdot \SupDim(\Hom_{\lie{l'}}(V, \Mod_{fd}(\lie{l'}))) < \infty.
	\end{align*}
\end{corollary}

\begin{proof}
Under the assumption, $G/P$ is $G'$-spherical by \cite[Corollary 15]{Ko05} (see also \cite[Proposition 6.12]{Ko07}).
Then the assertion follows from Theorems \ref{thm:BoundedRestriction} and \ref{thm:ReductionToFiberCohInd}.
\end{proof}

We refer the reader to \cite{AvPe20} and references therein for classification results of spherical actions on flag varieties.

\section{Coisotropic action}\label{sect:Coisotropic}

We have shown in Corollary \ref{cor:PidegPoissonCommutativeUg} that if $\univ{g}^{G'}/I$ has finite PI degree, then $S(\lie{g})^{G'}/\sqrt{\gr(I)}$ is Poisson-commutative.
In this section, we relate the Poisson-commutativity to coisotropic actions on coadjoint orbits.
As a result, we can characterize the uniform boundedness of multiplicities by the coisotropic actions.

\subsection{Poisson variety and symplectic leaves}

We recall the notion of Poisson varieties following \cite{Lo09}.
See also \cite{LaPiVa12}.

We denote by $\rsheaf{X}$ the structure sheaf of an algebraic variety $X$ over $\CC$ and denote by $\rsheaf{X,x}$ the stalk of $\rsheaf{X}$ at $x \in X$.
If $X$ is irreducible, let $K(X)$ denote the field of all rational functions on $X$, otherwise the direct product $\prod_{Y\subset X} K(Y)$ taken over all irreducible components.

\begin{definition}
	We say that an algebraic variety $X$ is a \define{Poisson variety} if $\rsheaf{X}$ is equipped with a structure of a sheaf of Poisson algebras.
	Let $X$ and $Y$ be Poisson varieties.
	A morphism $f\colon X\rightarrow Y$ of algebraic varieties is called \define{Poisson} if $f^*\colon \rsheaf{Y,f(x)}\rightarrow \rsheaf{X,x}$ is a morphism of Poisson algebras for any $x \in X$.
	If $Y$ is a (locally closed) subvariety of $X$ with a Poisson structure and its inclusion is Poisson, then $Y$ is called a \define{Poisson subvariety} of $X$.
\end{definition}

If $X$ is an affine Poisson variety, then $\rring{X}$ is a Poisson algebra.
The converse is also true.
In fact, if $\calA$ is a Poisson algebra, its localization $S^{-1}\calA$ by a multiplicatively closed subset $S$ admits a unique Poisson structure such that the natural homomorphism $\calA\rightarrow S^{-1}\calA$ is a morphism of Poisson algebras.

For any subvariety $Y$ of a Poisson variety $X$, the ideal sheaf $\calI_Y \subset \rsheaf{X}$ of functions vanishing on $Y$ is a sheaf of Poisson ideals if and only if $Y$ is a Poisson subvariety of $X$.
In other words, the Poisson structure on $Y$ such that $Y$ is a Poisson subvariety of $X$ is unique if it exists.

Symplectic varieties are typical examples of Poisson varieties.
We relate Poisson structure to symplectic structure.
Let $X$ be a smooth Poisson variety.
Let $TX$ (resp.\ $T^*X$) denote the tangent (resp.\ cotangent) bundle of $X$.
The Poisson bracket of $\rsheaf{X}$ can be written as
\begin{align*}
	\{f, g\} = \langle P, df\wedge dg \rangle \quad (f, g \in \rsheaf{X})
\end{align*}
for some bivector $P \in \sect(X, \wedge^2 TX)$, which is uniquely determined by the Poisson bracket.
Then the bivector $P_x \in \wedge^2 T_x X$ defines a linear map $p_x\colon T^*_x X \rightarrow T_x X$.
The image of $p_x$ can be written as
\begin{align*}
	T^P_x X:=p_x(T^*_x X) = \set{(X_f)_x : f \in \rsheaf{X, x}},
\end{align*}
where $X_f$ is the vector field corresponding to the derivation $\{f, \cdot\}$.

The following fact is well-known (see e.g.\ \cite[Proposition 2.2.1]{Lo09}).

\begin{fact}\label{fact:BasicPoissonVariety}
	Let $Y$ be a smooth subvariety of $X$.
	\begin{enumparen}
		\item $Y$ is Poisson if and only if $T^P_x X \subset T_x Y$ for any $x \in Y$.
		\item If $T_x Y = T^P_x X$ for any $x \in Y$, then the bivector $P$ defines a symplectic structure on $Y$.
	\end{enumparen}
\end{fact}

If a subvariety $Y$ of $X$ is Poisson and admits a symplectic structure by Fact \ref{fact:BasicPoissonVariety} (2), we say that $Y$ is a \define{symplectic subvariety} of $X$.

\begin{example}\label{ex:CoadjointOrbits}
	Let $G$ be a connected reductive algebraic group.
	Then $S(\lie{g})$ is a Poisson algebra since $S(\lie{g})$ is the associated graded algebra of $\univ{g}$.
	Hence $\lie{g}^*$ is a Poisson variety.
	The map $\lie{g}\ni v \mapsto X_v = \{v, \cdot\}$ coincides with the differential of the action of $G$ on $S(\lie{g})$.
	This implies
	\begin{align*}
		(X_v)_{\lambda} &= -\ad^*(v)\lambda \in \lie{g}^* \simeq T_{\lambda}\lie{g}^*, \\
		T^P_\lambda \lie{g}^* &= \ad^*(\lie{g})\lambda
	\end{align*}
	for $\lambda \in\lie{g}^*$.
	Therefore any coadjoint orbit is a symplectic subvariety of $\lie{g}^*$.
	The symplectic form on the orbit is called the Kirillov--Kostant--Souriau form.
\end{example}

It is well-known that there are finitely many nilpotent coadjoint orbits in $\lie{g}^*$ (see \cite[Theorem 3.5.4]{CoMc93}).
In particular, the closure of a nilpotent coadjoint orbit has finitely many $G$-orbits.

\begin{fact}\label{fact:FiniteNilpotentOrbits}
	Let $\OO$ be a nilpotent coadjoint orbit in $\lie{g}^*$.
	Then the closure $\overline{\OO}$ is a finite disjoint union of $G$-stable symplectic subvarieties, which are also nilpotent coadjoint orbits.
\end{fact}

We recall the definition of Hamiltonian actions.

\begin{definition}
	Let $G$ be an affine algebraic group and $X$ a Poisson variety equipped with a $G$-action and a $G$-linear map $H \colon \lie{g} \rightarrow \rring{X}$.
	We say that $X$ is a \define{Hamiltonian $G$-variety} or the $G$-action is \define{Hamiltonian} if $\{H(v), f\} = L(v)f$ for any $f \in \rsheaf{X}$ and $v \in \lie{g}$.
	Here $L(\cdot)$ is the differential of the $G$-action on $\rsheaf{X}$.
\end{definition}

Coadjoint orbits and cotangent bundles of $G$-varieties are typical examples of Hamiltonian $G$-varieties.
By definition, any $G$-stable Poisson subvariety of a Hamiltonian $G$-variety is also Hamiltonian.
We give an example of Hamiltonian $G$-varieties related to induced representations.

Let $G$ be a connected reductive algebraic group and $G'$ a connected reductive subgroup of $G$.
We have given a $\CC$-algebra structure on $\rring{G/G'}\otimes \univ{g}$ in \eqref{eqn:DefOU}.
The standard filtration of $\univ{g}$ induces a filtration on $\rring{G/G'}\otimes \univ{g}$.
Then $\rring{G/G'}\otimes \univ{g}$ is a filtered $\CC$-algebra whose associated graded algebra $\rring{G/G'}\otimes S(\lie{g})$ is a Poisson algebra.
We regard $G/G'\times \lie{g}^*$ as a $G$-variety via the diagonal action.
By definition, $G/G'\times \lie{g}^*$ is a Hamiltonian $G$-variety.
We shall show a similar result to Fact \ref{fact:FiniteNilpotentOrbits} for $G/G'\times \lie{g}^*$.

\begin{definition}
	For a $G'$-stable subvariety $S\subset (\lie{g'})^*$, we set
	\begin{align*}
		\Ind^{G}_{G'}(S) := \set{(gG', \lambda) \in G/G'\times \lie{g}^* : (\Ad^*(g^{-1})\lambda)|_{\lie{g'}} \in S},
	\end{align*}
	which is a $G$-stable subvariety of $G/G'\times \lie{g}^*$.
\end{definition}

Under the canonical isomorphism $G/G'\times \lie{g}^* \simeq G \times_{G'} \lie{g}^*$ given by $(gG', \lambda) \mapsto [g, \Ad^*(g^{-1})\lambda]$,
the $G$-variety $\Ind^G_{G'}(S)$ is isomorphic to $G\times_{G'}((\lie{g}/\lie{g'})^*\times S)$.
Here $[g, \lambda]$ is the image of $(g, \lambda) \in G\times \lie{g}^*$ by the natural surjection $G\times \lie{g}^*\rightarrow G\times_{G'} \lie{g}^*$.
Hence if $S$ is smooth, so is $\Ind^G_{G'}(S)$.

\begin{lemma}\label{lem:InducedOrbit}
	Let $\OO$ be a coadjoint orbit in $(\lie{g'})^*$.
	Then $\Ind^{G}_{G'}(\OO)$ is a symplectic subvariety of $G/G'\times \lie{g}^*$.
\end{lemma}

\begin{proof}
	Set $X := G/G'\times \lie{g}^*$ and take $x = (gG', \lambda) \in \Ind^G_{G'}(\OO)$.
	We shall compute $T^P_{x} X$.

	Fix a basis $\set{v_1, v_2, \ldots, v_n}$ of $\lie{g}$ and its dual basis $\set{\mu_1, \mu_2, \ldots, \mu_n} \subset \lie{g}^*$.
	By \eqref{eqn:DefOU}, we have
	\begin{align*}
		\{1\otimes v, f\otimes w\} = L(v)f \otimes w + f\otimes [v, w] \quad (f \in \rsheaf{X}, v, w \in \lie{g}),
	\end{align*}
	where $L(\cdot)$ is the differential of the left translation of $G$ on $G/G'$.
	Hence for $f \in \rsheaf{X, gG'}$ and $v \in \lie{g}$, we have
	\begin{align}
		(X_{1\otimes v})_x &= L(v)_{gG'} - \ad^*(v)\lambda \in T_x(Gx), \label{eqn:IndCoadjoint1} \\
		(X_{f\otimes 1})_x &= \sum_{i} \{f\otimes 1, 1\otimes v_i\}(x) \mu_i \\
		&= - \sum_{i} L(v_i)f(gG') \mu_i \in \Ad^*(g)(\lie{g}/\lie{g'})^*. \label{eqn:IndCoadjoint2}
	\end{align}
	Here we identify $T_\lambda\lie{g}^*$ with $\lie{g}^*$.
	By \eqref{eqn:IndCoadjoint1} and \eqref{eqn:IndCoadjoint2}, we have
	\begin{align*}
		T^P_{x}X = \Ad^*(g)(\lie{g}/\lie{g'})^* + T_x(Gx) \subset T_x(\Ind^G_{G'}(\OO))
	\end{align*}
	since $\Ind^{G}_{G'}(\OO)$ is $G$-stable.
	Again, by \eqref{eqn:IndCoadjoint1} and \eqref{eqn:IndCoadjoint2},
	we have
	\begin{align*}
		\dim_{\CC}(T^P_x X) &= \dim_{\CC}(G/G') + \dim_{\CC}(\Ad^*(g)(\lie{g}/\lie{g'})^* + \ad^*(\Ad(g)\lie{g'})\lambda) \\
		& = 2\dim_{\CC}(G/G') + \dim_{\CC}(\OO) = \dim_{\CC}(\Ind^{G}_{G'}(\OO)).
	\end{align*}
	Therefore we obtain $T^P_x X = T_x(\Ind^G_{G'}(\OO))$,
	and we have proved the lemma by Fact \ref{fact:BasicPoissonVariety}.
\end{proof}

\begin{corollary}\label{cor:InductionOrbit}
	Let $\OO$ be a nilpotent coadjoint orbit in $(\lie{g'})^*$.
	Then $\Ind^{G}_{G'}(\overline{\OO})$ is a closed Poisson subvariety of $G/G'\times \lie{g}^*$ and a finite disjoint union of $G$-stable symplectic subvarieties.
\end{corollary}

For an affine variety $Y$ and an ideal $I$ of $\rring{Y}$, we denote by $\calV(I)$ the closed subset of $Y$ defined by $I$.
The following lemma is the reason why we study $\Ind^G_{G'}(\OO)$.

\begin{lemma}\label{lem:InductionAndOrbit}
	Let $I$ be a two-sided ideal of $\univ{g'}$ and set
	\begin{align*}
		J:=\Ann_{\rring{G/G'}\otimes \univ{g}}(\ind{g}{g'}(\univ{g'}/I)),
	\end{align*}
	where $\univ{g'}/I$ is regarded as a $\lie{g'}$-module via the right action.
	Then we have $\calV(\gr(J)) = \Ind^G_{G'}(\calV(\gr(I)))$.
\end{lemma}

\begin{proof}
	By Lemma \ref{lem:AnnhilatorOUmod}, we have
	\begin{align*}
		J = \set{f \in \rring{G/G'}\otimes\univ{g}: f(gG') \in \Ad(g)I\univ{g}\ (\forall g\in G)}.
	\end{align*}
	Under the canonical isomorphism $\rring{G/G'}\otimes \univ{g} \simeq (\rring{G}\otimes \univ{g})^{G'}$ of filtered $\CC$-algebras, the image of the two-sided ideal $J$ coincides with $(\rring{G}\otimes I\univ{g})^{G'}$.
	Since $G'$ is reductive, we have
	\begin{align*}
		\gr((\rring{G}\otimes I\univ{g})^{G'}) = (\rring{G}\otimes \gr(I) S(\lie{g}))^{G'}.
	\end{align*}
	Hence we have
	\begin{align*}
		\gr(J) = \set{f \in \rring{G/G'}\otimes S(\lie{g}): f(gG') \in \Ad(g)\gr (I) S(\lie{g})\ (\forall g \in G)}.
	\end{align*}
	This shows $\calV(\gr(J)) = \Ind^G_{G'}(\calV(\gr(I)))$.
\end{proof}

\subsection{Losev's result}

Here we recall Losev's results \cite{Lo09} about Poisson centers of Poisson algebras of $G$-invariant functions.
To do so, we need the definition of conical Hamiltonian $G$-varieties in the sense of \cite[Definition 1.2.4]{Lo09}.
We describe the definition without using categorical quotients.
Let $G$ be a connected reductive algebraic group.

\begin{definition}
	Let $X$ be an affine Hamiltonian $G$-variety.
	We say that $X$ is \define{conical} if $\rring{X}$ has a $G$-stable $\ZZ$-grading (i.e.\ $\rring{X} = \bigoplus_{k \in \ZZ} \rring{X}_k$ as $G$-modules), the $G$-invariants $\rring{X}_k^G$ vanishes for any $k < 0$ and $H(\lie{g})$ is contained in $\rring{X}_k$ for some $k > 0$.
\end{definition}

In \cite{Lo09}, several results are stated for normal varieties.
We shall see that several properties of Hamiltonian $G$-varieties are preserved by normalization.
Although such results are in the paper \cite{Lo09}, we shall state them for readers' convenience.
The following lemma is standard.

\begin{lemma}\label{lem:NormalizationGvariety}
	Let $X$ be an affine irreducible $G$-variety and $\widetilde{X}$ the normalization of $X$.
	Then $\widetilde{X}$ admits a unique $G$-action such that the canonical morphism $\widetilde{X}\rightarrow X$ is $G$-equivariant, and $\rring{\widetilde{X}}^G$ is the integral closure of $\rring{X}^G$ in $K(X)^G$.
\end{lemma}

\begin{proof}
	By the universality of the normalization, the composition $G\times \widetilde{X}\rightarrow G\times X \rightarrow X$ lifts uniquely to a morphism $G\times \widetilde{X}\rightarrow \widetilde{X}$, which defines a $G$-action on $\widetilde{X}$.

	It is obvious that the integral closure of $\rring{X}^G$ in $K(X)^G$ is contained in $\rring{\widetilde{X}}^G$.
	To show the converse inclusion, take $x \in \rring{\widetilde{X}}^G$ and let $f \in \rring{X}[t]$ be the minimal polynomial of $x$.
	Since $G$ is reductive, there is a $G$-equivariant projection $R$ of $\rring{\widetilde{X}}[t]$ onto $\rring{\widetilde{X}}^G[t]$ called the Reynolds operator.
	Then $R(f)(x) = R(f(x)) = 0$ holds and hence $x$ is integral over $\rring{X}^G$.
\end{proof}

\begin{proposition}\label{prop:NormalizationHamiltonian}
	Let $X$ be an affine irreducible $G$-variety and $\widetilde{X}$ the normalization of $X$.
	\begin{enumparen}
		\item\label{enum:NormalizationHamiltonian} If $X$ is Hamiltonian and conical, then so is $\widetilde{X}$.
		\item\label{enum:NormalizationHamiltonianSymplectic} If $X$ is decomposed into a disjoint union of irreducible symplectic subvarieties, then so is $\widetilde{X}$.
	\end{enumparen}
\end{proposition}

\begin{proof}
	The assertion \ref{enum:NormalizationHamiltonian} is in \cite[Example 3.3.5]{Lo09}.
	Note that the conicality follows from Lemma \ref{lem:NormalizationGvariety} as follows.
	The $\CC^\times$-action on $X$ determined by the $\ZZ$-grading on $\rring{X}$ lifts to a $\CC^\times$-action on $\widetilde{X}$, which defines a $G$-stable $\ZZ$-grading on $\rring{\widetilde{X}}$.
	If $x$ is a non-zero homogeneous element in $\rring{\widetilde{X}}^G$ with negative degree, then $x$ is not a root of any monic $f \in \rring{X}^G[t]$ since $\rring{X}^G_k$ vanishes for any $k < 0$.
	Hence $\widetilde{X}$ is conical.

	Since the canonical morphism $\widetilde{X}\rightarrow X$ is finite and dominant, the assertion \ref{enum:NormalizationHamiltonianSymplectic} is a special case of \cite[Lemma 3.4.3]{Lo09}.
\end{proof}

For a Poisson algebra $\calA$, we denote by $\mathfrak{z}(\calA)$ the Poisson center of $\calA$, i.e.\ $\mathfrak{z}(\calA) = \set{a \in \calA: \{a, \cdot\} = 0}$.

\begin{proposition}\label{prop:IntegralExtensionPoisson}
	Let $\calA$ be a Poisson algebra which is an integral domain, and $\calB$ an integral extension of $\calA$.
	An extension of the Poisson bracket of $\calA$ to $\calB$ is unique if exists.
	If the extension exists, $\mathfrak{z}(\calA)$ is contained in $\mathfrak{z}(\calB)$.
\end{proposition}

\begin{proof}
	See \cite[Proposition 2.1.1]{Lo09} for the first assertion.
	For the last assertion, let $a \in \mathfrak{z}(\calA)$ and $b \in \calB$ with minimal polynomial $f \in \calA[t]$.
	Then we have $0 = \{a, f(b)\} = \{a, b\}f'(b)$.
	Since $f'(b)\neq 0$, we obtain $\{a, b\} = 0$.
\end{proof}

For an affine Hamiltonian $G$-variety $X$, we extend $H\colon \lie{g}\rightarrow \rring{X}$ to a homomorphism $H\colon S(\lie{g})\rightarrow \rring{X}$ of Poisson algebras.
Then $H(S(\lie{g})^G)$ is contained in $\mathfrak{z}(\rring{X}^G)$.
We state Losev's results in our setting.

\begin{theorem}\label{thm:FinitePoissonCenter}
	Let $X$ be an affine irreducible conical Hamiltonian $G$-variety
	and $\calA_X$ the integral closure of $H(S(\lie{g})^G)$ in $K(X)^G$.
	Assume that $X$ is a finite disjoint union of symplectic subvarieties.
	\begin{enumparen}
		\item\label{enum:FinitePoissonCenterExtension} $K(X)^G$ is a finite extension of the quotient field of $\rring{X}^G$.
		\item\label{enum:FinitePoissonCenterIntegral} $\calA_X$ coincides with the algebra of all integral elements over $\rring{X}^G$ in $\mathfrak{z}(K(X)^G)$.
		\item\label{enum:FinitePoissonCenterFinite} The Poisson center $\mathfrak{z}(\rring{X}^G)$ is finitely generated as an $S(\lie{g})^G$-module.
	\end{enumparen}
\end{theorem}

\begin{proof}
	We shall show \ref{enum:FinitePoissonCenterFinite} from \ref{enum:FinitePoissonCenterExtension} and \ref{enum:FinitePoissonCenterIntegral}.
	Let $L$ be the relative algebraic closure of the quotient field $M$ of $H(S(\lie{g})^G)$ in $K(X)^G$.
	Since $K(X)^G$ is finitely generated as a field over $M$, $L$ is a finite extension of $M$ by \cite[\S 14, Corollary 7.1]{Bo03}.
	Since $\calA_X$ is the integral closure of $H(S(\lie{g}))^G$ in $L$, the $\CC$-algebra $\calA_X$ is finitely generated by \cite[Chapter I, Theorem 3.9A]{Ha77_alggeo}.
	Hence $\calA_X$ is finitely generated as an $S(\lie{g})^G$-module.
	By Proposition \ref{prop:IntegralExtensionPoisson} and \ref{enum:FinitePoissonCenterExtension}, we have $\mathfrak{z}(\rring{X}^G) \subset \mathfrak{z}(K(X)^G)$.
	Hence we have $\mathfrak{z}(\rring{X}^G) \subset \calA_X$ by \ref{enum:FinitePoissonCenterIntegral}. 
	This shows \ref{enum:FinitePoissonCenterFinite}.

	Let $\widetilde{X}$ be the normalization of $X$.
	By Proposition \ref{prop:NormalizationHamiltonian}, $\widetilde{X}$ satisfies the assumptions of Theorem \ref{thm:FinitePoissonCenter}.
	Since $K(X) = K(\widetilde{X})$, we have $\calA_X = \calA_{\widetilde{X}}$.
	By Lemma \ref{lem:NormalizationGvariety}, the quotient field of $\rring{\widetilde{X}}^G$ is a finite extension of that of $\rring{X}^G$.
	Therefore, replacing $X$ with $\widetilde{X}$ we can assume that $X$ is normal.

	For normal $X$, \ref{enum:FinitePoissonCenterExtension} has been proved in \cite[Theorem 1.2.4]{Lo09}, and \ref{enum:FinitePoissonCenterIntegral} in \cite[Theorem 1.2.3]{Lo09}.
	Note that the strong equidefectinality follows from \cite[Lemma 3.4.2]{Lo09}, and $C_{G,X}$ in the Losev's paper is $\Spec(\calA_X)$ in our notation.
\end{proof}

\subsection{Coisotropic action and polynomial identity}

As we have seen in Corollary \ref{cor:PidegPoissonCommutativeUg},
the finiteness of the PI degree of $\univ{g}^{G'}/I$ implies the Poisson-commutativity of $S(\lie{g})^{G'}/\sqrt{\gr(I)}$.
Using Theorem \ref{thm:FinitePoissonCenter}, we shall relate the finiteness of the PI degree to a geometric condition of coadjoint orbits called coisotropic actions.
Let $G$ be a connected reductive algebraic group.

\begin{definition}
	Let $X$ be a Hamiltonian $G$-variety.
	We say that the $G$-action on $X$ is \define{coisotropic} if there is a $G$-stable open dense symplectic subvariety $U$ of $X$ such that for any $x \in U$, $L(\lie{g})_x$ is coisotropic in $T_x U$, that is,
	$L(\lie{g})_x \supset \set{v \in T_x U: \omega_x(v, L(\lie{g})_x) = 0}$.
	Here $\omega$ is the symplectic form on $U$.
\end{definition}

The following fact is the reason why we consider coisotropic actions.
See \cite[Chapter II, Proposition \S 3.5]{Vi01}.

\begin{fact}\label{fact:CoisotropicAndPoisson}
	Let $X$ be a Hamiltonian $G$-variety with an open dense $G$-stable symplectic subvariety.
	Then the $G$-action on $X$ is coisotropic if and only if $K(X)^G$ is Poisson-commutative.
\end{fact}

The following corollary is an easy consequence of Theorem \ref{thm:FinitePoissonCenter} \ref{enum:FinitePoissonCenterExtension} and Proposition \ref{prop:IntegralExtensionPoisson}.

\begin{corollary}\label{cor:CoisotropicPoisson}
	Under the assumption of Theorem \ref{thm:FinitePoissonCenter}, the following conditions are equivalent:
	\begin{enumparen}
		\item $X$ is coisotropic
		\item $\rring{X}^G$ is Poisson-commutative
		\item $K(X)^G$ is Poisson-commutative.
	\end{enumparen}
\end{corollary}

Let $G'$ be a connected reductive subgroup of $G$.
Note that, as we have seen in Proposition \ref{prop:PIdegStableConnected}, the finiteness of $\PIdeg(\univ{g}^{G'}/I)$ does not depend on the connectedness of $G'$.
The following two theorems are the main results in this section.

\begin{theorem}\label{thm:PIdegAndCoisotropicRest}
	Let $I$ be a two-sided ideal of $\univ{g}$ such that $I\cap \univcent{g}$ has finite codimension in $\univcent{g}$.
	Then the following conditions are equivalent:
	\begin{enumparen}
		\item\label{enum:PIdegAndCoisotropicRestPIdeg} $\PIdeg((\univ{g}/I)^{G'}) < \infty$
		\item\label{enum:PIdegAndCoisotropicRestPoisson} $(S(\lie{g})/\sqrt{\gr(I)})^{G'}$ is Poisson-commutative.
		\item\label{enum:PIdegAndCoisotropicRestCoisotropic} the $G'$-action on $\calV(\gr(I))$ is coisotropic
		\item\label{enum:PIdegAndCoisotropicRestFinite} $(\univ{g}/I)^{G'}$ is finitely generated as a $\univcent{g'}$-module.
	\end{enumparen}
\end{theorem}

\begin{proof}
	Obviously, we can assume $I \neq \univ{g}$.

	By Lemma \ref{lem:TwosidedPrimitive}, we can take primitive ideals $J_1, J_2, \ldots, J_n \subset \univ{g}$ containing $I$ such that $J_1J_2\cdots J_n \subset I$.
	Hence it is enough to show the theorem for primitive $I$.
	By \cite[Theorem 3.10]{Jo85}, the associated variety $\calV(\gr(I))$ is the Zariski closure of a nilpotent coadjoint orbit in $\lie{g}^*$.
	By Fact \ref{fact:FiniteNilpotentOrbits}, we can apply Theorem \ref{thm:FinitePoissonCenter} to the affine irreducible conical Hamiltonian $G'$-variety $\calV(\gr(I))$.

	\ref{enum:PIdegAndCoisotropicRestPIdeg} $\Rightarrow$ \ref{enum:PIdegAndCoisotropicRestPoisson} has been proved in Corollary \ref{cor:PidegPoissonCommutativeUg}.
	
	\ref{enum:PIdegAndCoisotropicRestPoisson} $\Leftrightarrow$ \ref{enum:PIdegAndCoisotropicRestCoisotropic} has been shown in Corollary \ref{cor:CoisotropicPoisson}.

	\ref{enum:PIdegAndCoisotropicRestCoisotropic} $\Rightarrow$ \ref{enum:PIdegAndCoisotropicRestFinite}:
	Assume \ref{enum:PIdegAndCoisotropicRestCoisotropic}.
	Then $K(\calV(\gr(I)))^{G'}$ is Poisson-commutative by Fact \ref{fact:CoisotropicAndPoisson}.
	In particular, the Poisson center $\mathfrak{z}(\rring{\calV(\gr(I))}^{G'})$ is equal to $\rring{\calV(\gr(I))}^{G'}$ itself.
	By Theorem \ref{thm:FinitePoissonCenter} \ref{enum:FinitePoissonCenterFinite}, $\rring{\calV(\gr(I))}^{G'}$ is finitely generated as an $S(\lie{g'})^{G'}$-module.
	Set $V_i := (\sqrt{\gr(I)}^{G'})^i$ for $i \geq 0$.
	Then $V_r$ is contained in $\gr(I)^{G'}$ for some $r > 0$, and $V_i / V_{i+1}$ is finitely generated as an $\rring{\calV(\gr(I))}^{G'}$-module and hence as an $S(\lie{g'})^{G'}$-module for any $i\geq 0$.
	This shows that $(S(\lie{g})/\gr(I))^{G'}$ is finitely generated as an $S(\lie{g'})^{G'}$-module.
	Therefore $(\univ{g}/I)^{G'}$ is finitely generated as a $\univcent{g'}$-module.

	\ref{enum:PIdegAndCoisotropicRestFinite} $\Rightarrow$ \ref{enum:PIdegAndCoisotropicRestPIdeg}:
	Take a finite generating subset $S \subset (\univ{g}/I)^{G'}$ as a $\univcent{g'}$-module.
	Then any irreducible $(\univ{g}/I)^{G'}$-module $V$ has at most $|S|$ dimension since $\univcent{g'}$ acts on $V$ by scalars.
	By Proposition \ref{prop:PIdegSemiprimitive}, we have
	\begin{equation*}
		\PIdeg((\univ{g}/I)^{G'}) \leq |S| < \infty. \qedhere
	\end{equation*}
\end{proof}

\begin{theorem}\label{thm:PIdegAndCoisotropicInd}
	Let $I$ be a two-sided ideal of $\univ{g'}$ such that $I\cap \univcent{g'}$ has finite codimension in $\univcent{g'}$.
	Then the following conditions are equivalent:
	\begin{enumparen}
		\item $\PIdeg((\univ{g}/I\univ{g})^{G'}) < \infty$
		\item $(S(\lie{g})/\sqrt{\gr(I)}S(\lie{g}))^{G'}$ is Poisson-commutative
		\item the $G$-action on $\Ind^G_{G'}(\calV(\gr(I)))$ is coisotropic
		\item $(\univ{g}/I\univ{g})^{G'}$ is finitely generated as a $\univcent{g}$-module.
	\end{enumparen}
\end{theorem}

\begin{proof}
	The proof of the theorem is essentially the same as that of Theorem \ref{thm:PIdegAndCoisotropicRest}.
	Use Corollary \ref{cor:InductionOrbit}, Lemma \ref{lem:InductionAndOrbit}
	and the canonical isomorphism $(\rring{G/G'}\otimes \univ{g})^{G} \simeq (\univ{g}^{G'})^{\opalg}$.
	See also Lemmas \ref{lem:AnnhilatorOUmod} and \ref{lem:faithfully1}.
\end{proof}

The following two results are useful to classify coisotropic actions on nilpotent coadjoint orbits and $\Ind^G_{G'}(\OO)$.

\begin{proposition}
	Let $X$ be an irreducible affine Hamiltonian $G$-variety satisfying the assumption of Theorem \ref{thm:FinitePoissonCenter}, e.g.\ the Zariski closure of a nilpotent coadjoint orbit and $\Ind^G_{G'}(\overline{\OO})$ for a nilpotent coadjoint orbit $\OO$ in $(\lie{g'})^*$.
	If the $G$-action on $X$ is coisotropic, then the $G$-action on any $G$-stable Poisson subvariety $Y$ of $X$ is coisotropic.
\end{proposition}

\begin{proof}
	By the definition of coisotropic actions, we can assume that $Y$ is closed and irreducible.
	Hence $Y$ satisfies the assumption of Theorem \ref{thm:FinitePoissonCenter}.
	By Corollary \ref{cor:CoisotropicPoisson}, $\rring{X}^G$ is Poisson-commutative.
	Since $G$ is reductive, $\rring{Y}^G$ is a quotient of $\rring{X}^G$.
	Therefore $\rring{Y}^G$ is also Poisson-commutative.
	By Corollary \ref{cor:CoisotropicPoisson}, $Y$ is coisotropic.
\end{proof}

\begin{remark}
	The proposition has been proved in \cite[Proposition 2.7]{AvPe14} for nilpotent coadjoint orbits.
	Our proof is essentially the same as theirs.
\end{remark}

The following fact is proved in \cite[Theorem 2.6]{AvPe14}.
One can give an alternative proof of the fact using Proposition \ref{prop:boundedAndSphericalGP} and Theorems \ref{thm:pidegAnnihilatorVerma} and \ref{thm:PIdegAndCoisotropicRest}.

\begin{fact}\label{fact:ParabolicSpherical}
	Let $P$ be a parabolic subgroup of $G$.
	Set $X:=\Ad^*(G)(\lie{g}/\lie{p})^*$.
	Then $G/P$ is $G'$-spherical if and only if the $G'$-action on $X$ is coisotropic.
\end{fact}

\begin{remark}
	It is well-known that $X$ has a unique open $G$-orbit called a Richardson orbit.
	The variety $X$ coincides with the image of $T^*(G/P)$ by the moment map
	and depends only on the conjugacy class in $G$ of Levi subgroups of $P$.
	See e.g.\ \cite[Chapter 7]{CoMc93}.
\end{remark}

Spherical actions on flag varieties are well studied by many mathematicians.
We refer to \cite{AvPe20} and references therein for classification results.
We note the difference between Theorem \ref{thm:PIdegAndCoisotropicRest} and the following result of A. Petukhov.
A $\lie{g}$-module $V$ is called a $(\lie{g}, \lie{g'})$-module if the $\lie{g'}$-action on $V$ is locally finite.

\begin{fact}[{\cite[Theorem 3.1 and its proof]{Pe18}}]\label{fact:Petukhov}
	Let $V$ be an irreducible $(\lie{g}, \lie{g'})$-module with annihilator $I\subset \univ{g}$.
	Set $\AV(V):=\calV(\gr(\Ann_{\univ{g}}(v)))$ for non-zero $v\in V$.
	(Note that $\AV(V)$ is independent of $v$.)
	Then the following conditions are equivalent:
	\begin{enumparen}
		\item $\SupDim(\Hom_{\lie{g'}}(\Mod_{fd}(\lie{g'}), V)) < \infty$
		\item $\AV(V)$ is $G'$-spherical
		\item the $G'$-action on $\calV(\gr(I))$ is coisotropic.
	\end{enumparen}
	Here $\Mod_{fd}(\lie{g'})$ is the category of completely reducible finite-dimensional $\lie{g'}$-modules.
\end{fact}

\begin{remark}
	The $\lie{g'}$-action on $V$ is completely reducible since $V$ is irreducible.
	By the Jacobson density theorem, the $\univ{g}^{G'}$-action on $\Hom_{\lie{g'}}(F, V)$ is irreducible for any irreducible $F \in \Mod_{fd}(\lie{g'})$.
	Hence we have
	\begin{align*}
		\PIdeg((\univ{g}/I)^{G'}) = \SupDim(\Hom_{\lie{g'}}(\Mod_{fd}(\lie{g'}), V))
	\end{align*}
	by Propositions \ref{prop:pidLowerbound} and \ref{prop:pidegUpperbound}.
	This equation is known in \cite[Theorem 4.3]{PeSe12}.
\end{remark}

Theorem \ref{thm:PIdegAndCoisotropicRest} (combined with Theorem \ref{thm:BoundedRestriction}) is a generalization of Fact \ref{fact:Petukhov}.
For some primitive ideal $I\subset \univ{g}$, there can be no irreducible $(\lie{g}, \lie{g'})$-module with the annihilator $I$.
For example, the Gelfand--Kirillov dimension of any irreducible $(\so(n,\CC), \so(n-1, \CC))$-module is less than or equal to $n-2$, and hence there is no irreducible $(\so(n,\CC), \so(n-1, \CC))$-module whose annihilator is a minimal primitive ideal if $n \geq 5$.
This means that Theorem \ref{thm:PIdegAndCoisotropicRest} can not be deduced from Fact \ref{fact:Petukhov} in general.

If $(G, G')$ is a symmetric pair, a classification of spherical nilpotent $G'$-orbits in $(\lie{g'})^\perp$ is given by D. R. King \cite{Ki04}.
Using this classification and Fact \ref{fact:Petukhov}, we can obtain many coisotropic actions on nilpotent coadjoint orbits.

Let $V$ be a $\lie{g}$-module.
Then $\calV(V):=\calV(\gr(\Ann_{\univ{g}}(V)))$ is a $G$-stable Poisson subvariety of $\lie{g}^*$.
The variety $\calV(V)$ is an invariant of a $\lie{g}$-module and well-behaved in the translation principle.
More precisely, $\calV(V\otimes F)$ coincides with $\calV(V)$ for any finite-dimensional $\lie{g}$-module $F$.
This and the results in this subsection suggest that the PI degree and the supremum of multiplicities are also well-behaved in the translation principle.
In fact, the following proposition holds.
Although one can show similar results for e.g.\ $\PIdeg$, $(\lie{g}, K)$-modules and inductions, we omit such variations.
See \cite[Section 5]{Ko08} for the case of $(\lie{g}, K)$-modules.

\begin{proposition}
	Let $V$ be a $\lie{g'}$-module and $F$ a finite-dimensional $\lie{g'}$-module.
	Then there exists a constant $C > 0$ independent of $V$ and $F$ such that
	\begin{align*}
		\SupDim(V\otimes_{\univ{g'}}\calO_{\lie{g'}}) &\leq \SupDim((V\otimes F)\otimes_{\univ{g'}}\calO_{\lie{g'}}) \\
		&\leq C\cdot \dim_{\CC}(F) \cdot \SupDim(V\otimes_{\univ{g'}}\calO_{\lie{g'}}).
	\end{align*}
\end{proposition}
		
\begin{proof}
	For any irreducible object $V'$ in $\calO_{\lie{g'}}$, there is an irreducible object $W'$ in $\calO_{\lie{g'}}$ such that $V'$ is isomorphic to a quotient of $W'\otimes F$.
	This implies
	\begin{align*}
		\SupDim((V\otimes F)\otimes_{\univ{g'}}\calO_{\lie{g'}}) &= \SupDim((V\otimes_{\univ{g'}} (F \otimes\calO_{\lie{g'}})) \\
		&\geq \SupDim(V\otimes_{\univ{g'}}\calO_{\lie{g'}}).
	\end{align*}
	This is the first inequality in the assertion.

	Let $V'$ be an irreducible object in $\calO_{\lie{g'}}$.
	We shall show that there exists a constant $C > 0$ independent of $V'$ and $F$ such that
	\begin{align*}
		\Len_{\lie{g'}}(V' \otimes F) \leq C\cdot \dim_{\CC}(F).
	\end{align*}
	The second inequality in the assertion follows from this claim since the functor $V\otimes_{\univ{g'}}(\cdot)$ is right exact.

	Take a Borel subalgebra $\lie{b'}$ of $\lie{g'}$ and a character $\lambda$ of $\lie{b'}$ such that $V'$ is isomorphic to a quotient of $\univ{g'}\otimes_{\univ{b'}}\CC_{\lambda}$.
	Then we have
	\begin{align*}
		(\univ{g'}\otimes_{\univ{b'}} \CC_{\lambda})\otimes F \simeq \univ{g'} \otimes_{\univ{b'}} (\CC_{\lambda}\otimes F).
	\end{align*}
	$X:=\univ{g'} \otimes_{\univ{b'}} (\CC_{\lambda}\otimes F)$ has a filtration $0 = X_0 \subset X_1 \subset X_2 \subset \cdots \subset X_{\dim_{\CC}(F)} = X$ such that $X_i/X_{i-1}$ is isomorphic to a Verma module for any $i$.
	By Fact \ref{fact:BoundedGeneralizedVerma}, we can take a constant $C>0$ such that the length of any Verma module of $\lie{g'}$ is less than or equal to $C$.
	Therefore we obtain
	\begin{equation*}
		\begin{split}
			\Len_{\lie{g'}}(V\otimes F) &\leq \Len_{\lie{g'}}(X) \\
			&\leq \sum_{i=1}^{\dim_{\CC}(F)} \Len_{\lie{g'}}(X_i/X_{i-1}) \\
			&\leq C\cdot \dim_{\CC}(F). \qedhere
		\end{split}
	\end{equation*}
\end{proof}

\subsection{Example}\label{subsect:MultiplicityOne}

As an application of our characterization of the uniform boundedness of multiplicities, we consider the case of so called the multiplicity one theorem.
We use the notation $G_\RR, G'_\RR, K, K', G$ and $G'$ in Subsection \ref{sect:Setting}.

\begin{theorem}\label{thm:StrongGelfand}
	The following conditions for the pair $(G_\RR, G'_\RR)$ is equivalent:
	\begin{enumparen}
		\item\label{item:StrongGelfandCW} $\SupDim(\Hom_{G'_\RR}(\calC_{G_\RR}, \calC_{G'_\RR})) < \infty$
		\item\label{item:StrongGelfandGK} $\SupDim(\Hom_{\lie{g'}, K'}(\calC_{\lie{g}, K}, \calC_{\lie{g'}, K'})) < \infty$
		\item\label{item:StrongGelfandSpherical} $(G_0\times G'_0)/\Delta(G'_0)$ is $(G_0\times G'_0)$-spherical
		\item\label{item:StrongGelfandAlgebra} $\univ{g}^{G'}$ is commutative.
	\end{enumparen}
\end{theorem}

\begin{proof}
	Note that we have
	\begin{align*}
		\SupDim(\Hom_{G'_\RR}(\calC_{G_\RR}, \calC_{G'_\RR})) &= \SupDim(\Hom_{\Delta(G'_\RR)}(\calC_{G_\RR\times G'_\RR}, \CC)), \\
		\SupDim(\Hom_{\lie{g'}, K'}(\calC_{\lie{g}, K}, \calC_{\lie{g'}, K'})) &= \SupDim(\Hom_{\Delta(\lie{g'}), \Delta(K')}(\calC_{\lie{g\oplus g'}, K\times K'}, \CC)).
	\end{align*}
	See Fact \ref{fact:TensorIsom}, Proposition \ref{prop:IsomorphismTensor} and Lemma \ref{lem:HomIsomSmooth}.
	Therefore the equivalence \ref{item:StrongGelfandCW} $\Leftrightarrow$ \ref{item:StrongGelfandGK} $\Leftrightarrow$ \ref{item:StrongGelfandSpherical} follows from Theorem \ref{thm:BoundedInduction} and Corollary \ref{cor:ReductionToFiber} for $V' = \CC$.
	The implication \ref{item:StrongGelfandAlgebra} $\Rightarrow$ \ref{item:StrongGelfandCW} follows from Theorem \ref{thm:BoundedRestriction}.

	We shall show \ref{item:StrongGelfandSpherical} $\Rightarrow$ \ref{item:StrongGelfandAlgebra}.
	The homomorphism $\univ{g}^{G'_0} \rightarrow \univ{g\oplus g'}^{\Delta(G'_0)}$ ($X \mapsto X\otimes 1$) induces an isomorphism
	\begin{align*}
		\univ{g}^{G'_0} \simeq (\univ{g\oplus g'}/\Delta(\lie{g'})\univ{g\oplus g'})^{\Delta(G'_0)},
	\end{align*}
	where we identify $\univ{g\oplus g'}$ with $\univ{g}\otimes \univ{g'}$.
	Hence $\univ{g}^{G'_0}$ is isomorphic to the algebra of $(G_0\times G'_0)$-invariant differential operators on $(G_0\times G'_0)/\Delta(G'_0)$.
	By \cite[Theorem 25.4]{Ti11}, \ref{item:StrongGelfandSpherical} implies that $\univ{g}^{G'_0}$ is commutative and hence so is $\univ{g}^{G'}$.
\end{proof}

Theorem \ref{thm:StrongGelfand} without \ref{item:StrongGelfandGK} is known by Kobayashi--Oshima \cite[Theorem D]{KoOs13}.
Note that the author of the present paper has no direct proof of the equivalence \ref{item:ExistenceQuotInductionCW} $\Leftrightarrow$ \ref{item:StrongGelfandGK}.
The pair $(\lie{g}, \lie{g'})$ satisfying the equivalent conditions in Theorem \ref{thm:StrongGelfand} is classified by M. Kr\"amer \cite{Kr76}, and is isomorphic to a direct sum of some copies of $(\sl(n,\CC), \gl(n-1, \CC))$ and $(\so(n,\CC), \so(n-1,\CC))$.
If $(G_\RR, G'_\RR)$ is one of $(\GL(n+1,\CC), \GL(n,\CC))$, $(\upO(n+1,\CC), \upO(n,\CC))$, $(\SO(n + 1, \CC), \SO(n,\CC))$, $(\GL(n+1,\RR), \GL(n,\RR))$, $(\U(p,q+1), \U(p,q))$, $(\upO(p, q+1), \upO(p, q))$, $(\SO(p, q + 1), \SO(p, q))$, then it is known that $\SupDim(\Hom_{G'_\RR}(\calC_{G_\RR}, \calC_{G'_\RR}))$ is just one \cite{AiGo09,SuZh12}, which is called the multiplicity one theorem.

The multiplicity one theorem has been proved for Jacobi groups in \cite{SuZh12}.
We shall show an analogue of Theorem \ref{thm:StrongGelfand} for the Jacobi group case.
For simplicity, suppose that $G_\RR$ is connected non-compact simple.
Then $\lie{g}$ is simple or a direct sum of two mutually isomorphic simple factors.
We denote by $\calO_m$ the closure of the set of nilpotent elements in $\lie{g}^* \simeq \lie{g}$ whose projections onto any simple ideal belong to the minimal nilpotent orbit of the ideal.
Clearly, if $\lie{g}$ is simple, then $\calO_m$ is the closure of the minimal nilpotent orbit in $\lie{g}^*$.

\begin{lemma}\label{lem:MinimalNilpotentCoisotropic}
	Let $\calO$ be a nilpotent coadjoint orbit in $\lie{g}^*$
	and $H$ a maximal torus in $G$.
	Suppose $\calO \neq \set{0}$.
	Then the $H$-action on $\calO$ is coisotropic if and only if $\calO$ is contained in $\calO_m$ and $\lie{g}$ is isomorphic to $\sl(n,\CC)$, $\sl(n,\CC)\oplus \sl(n,\CC)$, $\sp(n,\CC)$ or $\sp(n,\CC)\oplus \sp(n,\CC)$.
\end{lemma}

\begin{proof}
	It is enough to show the assertion for simple $\lie{g}$.
	Assume that the $H$-action on $\calO$ is coisotropic.
	By the definition of coisotropic actions, we have $\dim_{\CC}(\calO) \leq 2 \dim_{\CC}(H)$.
	The dimension of the minimal nilpotent orbit can be computed by the root system of $\lie{g}$ (see \cite[Lemma 4.3.5 and Subsection 8.4]{CoMc93}).
	By a case by case analysis, we can see that $\calO$ is $\set{0}$ or the minimal nilpotent orbit, and $\lie{g}$ is isomorphic to $\sl(n,\CC)$ or $\sp(n,\CC)$.
	Note that $\dim_{\CC}(\calO_m) = 2 \dim_{\CC}(H)$ in the cases.

	We shall show the converse.
	It is enough to show that the $H$-action on $\calO_m$ is coisotropic if $\lie{g} = \sl(n,\CC)$ or $\sp(n,\CC)$.
	We can take an irreducible highest weight module $V$ of $\lie{g}$ such that $V|_{\lie{h}}$ is completely reducible and multiplicity-free, and $\calV(\gr(\Ann_{\univ{g}}(V))) = \calO_m$.
	By Theorems \ref{thm:BoundedRestriction} and \ref{thm:PIdegAndCoisotropicRest}, the $H$-action on $\calO_m$ is coisotropic.

	Note that as such a representation $V$, we can take an irreducible component of the Segal--Shale--Weil representation if $\lie{g} = \sp(n,\CC)$,
	and an irreducible generalized Verma module $\univ{g}\otimes_{\univ{p}}\CC_{\lambda}$ if $\lie{g} = \sl(n,\CC)$.
	Here $\lie{p}$ is a maximal parabolic subalgebra of $\sl(n,\CC)$
	containing $\sl(n-1,\CC)\oplus \CC$ as a Levi subalgebra.
\end{proof}

\begin{theorem}\label{thm:FourierJacobi}
	Let $V$ be an irreducible Casselman--Wallach representation of $G_\RR$ with annihilator $I \subset \univ{g}$.
	Suppose that $V$ is infinite dimensional.
	Fix a Cartan subgroup $H_\RR$ of $G_\RR$.
	The following conditions for $V$ is equivalent:
	\begin{enumparen}
		\item\label{item:FourierJacobiCW} $\SupDim(\Hom_{G_\RR}(\calC_{G_\RR}\toptensor[\pi] V, \calC_{G_\RR})) < \infty$
		\item\label{item:FourierJacobiGK} $\SupDim(\Hom_{\lie{g}, K}(\calC_{\lie{g}, K}\otimes V_K, \calC_{\lie{g}, K})) < \infty$
		\item\label{item:FourierJacobiFiber} $\SupDim(\Hom_{H_\RR}(V, \calC_{H_\RR})) < \infty$
		\item\label{item:FourierJacobiOrbit} $\calV(\gr(I)) \subset \calO_m$, and $\lie{g}$ is isomorphic to $\sl(n,\CC)$, $\sl(n,\CC)\oplus \sl(n,\CC)$, $\sp(n,\CC)$ or $\sp(n,\CC)\oplus \sp(n,\CC)$,
	\end{enumparen}
	where $\toptensor[\pi]$ is the projective tensor product (see Subsection \ref{subsect:LocallyConvex}).
\end{theorem}

\begin{proof}
	Let $\calH$ be a Hilbert globalization of $V$ and set $W:=(\calH^*)^\infty$.
	As in the proof of Theorem \ref{thm:StrongGelfand}, we have
	\begin{align*}
		\SupDim(\Hom_{G_\RR}(\calC_{G_\RR}\toptensor[\pi] V, \calC_{G_\RR})) &= \SupDim(\Hom_{\Delta(G_\RR)}(\calC_{G_\RR\times G_\RR}, W)), \\
		\SupDim(\Hom_{\lie{g}, K}(\calC_{\lie{g}, K}\otimes V_K, \calC_{\lie{g}, K})) &= \SupDim(\Hom_{\Delta(\lie{g}), \Delta(K)}(\calC_{\lie{g\oplus g}, K\times K}, W_K)).
	\end{align*}
	We shall apply Corollary \ref{cor:ReductionToFiber} to $(G_\RR, G'_\RR, V') = (G_\RR\times G_\RR, \Delta(G_\RR), W)$.
	Let $B$ be a Borel subgroup of $G$ and $\overline{B}$ an opposite Borel subgroup of $G$ such that $H:=B\cap \overline{B}$ is the analytic subgroup of $G$ with the Lie algebra $\lie{h}$.
	Then $(B\times \overline{B})\Delta(G)$ is open dense in $G\times G$ and $(B\times \overline{B}) \cap \Delta(G) = \Delta(H)$.
	This implies that the subgroup $L$ in Corollary \ref{cor:ReductionToFiber} is $\Delta(H)$ in this case.
	Therefore the equivalence \ref{item:FourierJacobiCW} $\Leftrightarrow$ \ref{item:FourierJacobiGK} $\Leftrightarrow$ \ref{item:FourierJacobiFiber} follows from Theorem \ref{thm:BoundedInduction} and Corollary \ref{cor:ReductionToFiber}.
	
	By Theorems \ref{thm:BoundedRestriction} and \ref{thm:PIdegAndCoisotropicRest}, the property \ref{item:FourierJacobiFiber} holds if and only if the $H$-action on $\calV(\gr(I))$ is coisotropic.
	Hence we have \ref{item:FourierJacobiFiber} $\Leftrightarrow$ \ref{item:FourierJacobiOrbit} by Lemma \ref{lem:MinimalNilpotentCoisotropic}.
	Note that the variety $\calV(\gr(I))$ has positive dimension since $V$ is infinite dimensional.
\end{proof}

\section{Unitary representation}\label{sect:Unitary}

In this section, we prove an analogue of Theorems \ref{thm:BoundedRestriction} and \ref{thm:BoundedInduction} for unitary representations.

\subsection{Locally convex space}\label{subsect:LocallyConvex}

We summarize the notions of nuclear spaces, barreled spaces and topological tensor products, and their properties.
We refer the reader to \cite{Tr67}.
In this section, any locally convex space is over $\CC$ and Hausdorff.
We denote by $\topdual{V}$ the strong dual of a locally convex space $V$.

We recall two kinds of topological tensor products.
Let $V$ and $W$ be locally convex spaces.
Then the projective topology on $V\otimes W$ is the strongest locally convex topology such that the canonical bilinear map $V\times W\rightarrow V\otimes W$ is continuous.
Similarly, the inductive topology on $V\otimes W$ is the strongest locally convex topology such that the canonical bilinear map $V\times W\rightarrow V\otimes W$ is separately continuous.
We denote by $V\toptensor[\pi] W$ (resp.\ $V\toptensor[i] W$) the completion of $V\otimes W$ with respect to the projective (resp.\ inductive) topology, and the obtained space is called the projective (resp.\ inductive) tensor product of $V$ and $W$.

The following properties are fundamental.
See \cite[Theorems A 2.2.3 and A 2.2.4]{Wa72} and \cite[Corollary of Theorem 34.1]{Tr67}.

\begin{fact}
	Let $B\colon V\times W \rightarrow X$ be a bilinear map to a locally convex space $X$.
	Then the unique linear extension $V\otimes W\rightarrow X$ of $B$ is continuous with respect to the projective (resp.\ inductive) topology if and only if $B$ is continuous (resp.\ separately continuous).
\end{fact}

\begin{fact}\label{fact:TensorIsom}
	If $V$ and $W$ are Fr\'echet, then any separately continuous bilinear map $B\colon V\times W \rightarrow X$ to a locally convex space is continuous.
	In particular, we have a canonical isomorphism
	\begin{align*}
		V\toptensor[i] W &\simeq V\toptensor[\pi] W.
	\end{align*}
\end{fact}

The projective topology on $V\otimes W$ is given by all semi-norms of the form $p\otimes q$ for continuous semi-norms $p$ on $V$ and $q$ on $W$.
Hence if $V$ and $W$ are Fr\'echet, then $V\toptensor[\pi] W$ is also Fr\'echet.

The notion of barreled spaces is useful to relate the inductive tensor product to the space of continuous linear maps.
A locally convex space $V$ is called barreled if any closed absorbing absolutely convex subset (called barrel) in $V$ is a neighborhood of $0\in V$.

Many spaces used in the representation theory are barreled.
For example, Fr\'echet spaces and strong duals of nuclear Fr\'echet spaces are barreled,
and strict inductive limits of Fr\'echet spaces are barreled.
See \cite[Section 33]{Tr67}.
In particular, Casselman--Wallach representations and their strong duals are barreled.

One of important properties of barreled spaces is the Banach--Steinhaus theorem:
if $V$ is a barreled space and $W$ is a locally convex space,
then a subset $S \subset \Hom_{\CC}(V, W)$ is bounded with respect to the weak operator topology if and only if $S$ is equicontinuous.
See \cite[Corollary 33.1]{Tr67}.
The following propositions are easy consequences of the Banach--Steinhaus theorem.

\begin{proposition}\label{prop:IsomorphismTensor}
	If $V$ is barreled, then there exists a natural isomorphism
	\begin{align*}
		\topdual{(V\toptensor[i] W)} &\simeq \Hom_{\CC}(V, \topdual{W})
	\end{align*}
	of vector spaces.
\end{proposition}

\begin{proof}
	Let $S$ be a bounded subset of $W$ and $B\colon V\times W \rightarrow \CC$ a separately continuous bilinear map.
	Then the subset $\set{B(\cdot, w): w \in S}$ of $V^*$ is equicontinuous by the Banach--Steinhaus theorem.
	In other words, the linear map $T_B\colon V\rightarrow W^*$ ($T_B(v) = B(v,\cdot)$) is continuous.
	It is easy to see that the obtained linear map $T\colon \topdual{(V\toptensor[i] W)}\rightarrow \Hom_{\CC}(V, \topdual{W})$ is bijective.
\end{proof}

\begin{lemma}\label{lem:HomBarrelled}
	Let $V$ and $W_1$ are locally convex spaces and $W_2$ a subspace of $W_1$ provided with a locally convex topology such that the inclusion map $W_2\hookrightarrow W_1$ is continuous.
	Assume that $V$ is barreled and the image of the restriction map $\topdual{W_1}\rightarrow \topdual{W_2}$ is sequentially dense in $\topdual{W_2}$.
	If the image of an operator $T \in \Hom_{\CC}(V, W_1)$ is contained in $W_2$, then we have $T \in \Hom_{\CC}(V, W_2)$.
\end{lemma}

\begin{proof}
	Take a sequence $(\varphi_i)_{i\in \NN}$ in $\topdual{W_1}$ converging in $\topdual{W_2}$ to a functional $\varphi \in \topdual{W_2}$.
	Then $\varphi_i \circ T$ is continuous for any $i \in \NN$
	and $(\varphi_i \circ T)_{i \in \NN}$ converges pointwise to $\varphi\circ T$.
	By the Banach--Steinhaus theorem, we have $\varphi\circ T \in \topdual{V}$.
	This implies that $T$ is continuous for the weak topologies of $V$ and $W_2$.
	Since $V$ is barreled, any continuous linear map for the weak topology is also continuous for the initial topology (see \cite[Proposition IV.\S 1.3.7]{Bo03_TVS}).
	Hence $T\colon V\rightarrow W_2$ is continuous.
\end{proof}

Let $V$ be a locally convex space.
For a continuous semi-norm $p$ on $V$, we denote by $V_p$ the completion of the normed space $V/p^{-1}(0)$ with the norm $\overline{p}$ induced from $p$.
If $\overline{p}$ is obtained from an inner product on $V_p$, we say that the semi-norm $p$ is Hilbert.

We recall the notion of nuclear spaces.
In this paper, we do not need the definition of nuclear spaces itself.
For a smooth manifold $M$, it is well-known that the function spaces $C_c^\infty(M)$ and $C^\infty(M)$ are nuclear.
Casselman--Wallach representations of real reductive Lie groups are nuclear.
The following characterization of the nuclearity is important to the disintegration of unitary representations.
See \cite[III.7.2 and Corollary III.7.3.2]{ScWo99} for the characterization.

\begin{fact}\label{fact:NuclearCharacterization}
	A locally convex space $V$ is nuclear if and only if $V$ satisfies the following two conditions:
	\begin{enumparen}
		\item for any continuous semi-norm $p$ on $V$, there exists a continuous Hilbert semi-norm $q$ on $V$ such that $p(v) \leq q(v)$ for any $v \in V$
		\item for any continuous linear map $\varphi\colon V\rightarrow \calH$ to a Hilbert space $\calH$, there exists a continuous Hilbert semi-norm $p$ on $V$ such that $\varphi$ lifts to a Hilbert--Schmidt operator $\overline{\varphi}\colon V_p\rightarrow \calH$.
	\end{enumparen}
\end{fact}

The nuclearity is preserved by several operations of locally convex spaces.
In this section, the following properties are important.
See \cite[Proposition 50.1]{Tr67}.

\begin{fact}\label{fact:Nuclear}
	\begin{enumparen}
		\item A subspace of a nuclear space is nuclear.
		\item A Hausdorff quotient of a nuclear space is nuclear.
		\item The projective tensor product of two nuclear spaces is nuclear.
		\item A countable inductive limit of nuclear spaces is nuclear.
	\end{enumparen}
\end{fact}

\subsection{Frobenius reciprocity}

We review a reciprocity theorem of induced representations.
Let $X$ be a smooth manifold and $V$ a locally convex space.
Let $C^\infty(X, V)$ denote the space of all $V$-valued $C^\infty$-functions
and put
\begin{align*}
	C^\infty_Y(X, V) &:= \set{f \in C^\infty(X, V): \supp(f) \subset Y} \quad (\text{closed }Y\subset X), \\
	C^\infty_c(X, V) &:= \set{f \in C^\infty(X, V): \supp(f) \text{ is compact}}.
\end{align*}
We provide $C^\infty(X, V)$ with the topology of uniform convergence on compact subsets of functions and their derivatives.
We consider $C^\infty_Y(X, V)$ as a closed subspace of $C^\infty(X, V)$,
and $C^\infty_c(X, V)$ as the strict inductive limit of the subspaces $C^\infty_K(X, V)$ over all compact subsets $K\subset X$.
The spaces can be described by topological tensor products.
See \cite[Appendices 2.2]{Wa72} and \cite[Proposition 51.6]{Tr67}.

\begin{fact}\label{fact:FunctionSpace}
	Let $V$ be a complete locally convex space and $K$ a compact subset of $X$.
	Then there exist natural isomorphisms
	\begin{align*}
		C^\infty(X, V) &\simeq C^\infty(X)\toptensor[\pi] V, \\
		C^\infty_K(X, V) &\simeq C^\infty_K(X)\toptensor[\pi] V (\simeq C^\infty_K(X) \toptensor[i] V \quad \text{if $V$ is Fr\'echet}), \\
		C^\infty_c(X, V) &\simeq C^\infty_c(X)\toptensor[i] V \quad \text{if $V$ is Fr\'echet}.
	\end{align*}
	In particular, if $V$ is nuclear (resp.\ separable), then the three function spaces are nuclear (resp.\ separable), and if $V$ is Fr\'echet, then the spaces are barreled.
\end{fact}

Let $G_\RR$ be a Lie group and $G'_\RR$ a closed subgroup of $G_\RR$.
For simplicity, we assume that $G_\RR$ and $G'_\RR$ are unimodular.
Fix Haar measures $\mu$ and $\mu'$ on $G_\RR$ and $G'_\RR$, respectively, and a $G_\RR$-invariant Radon measure $\nu$ on $G_\RR/G'_\RR$.
For a continuous representation $V$ of $G_\RR$, let $V^\infty$ denote the subspace of all smooth vectors in $V$.

In this section, we use two kinds of induced representations.
See \cite[5.1.1 and Lemma 5.1.1.5]{Wa72}.

\begin{definition}\label{def:CompactInd}
	Let $(\pi, V)$ be a smooth Fr\'echet representation of $G'_\RR$.
	We set
	\begin{align*}
		\ccInd^{G_\RR}_{G'_\RR}(V) := \set{f \in C^\infty(G_\RR, V):
		\begin{array}{l}
			f(gh) = \pi(h^{-1})f(g) \text{ ($\forall h \in G'_\RR$)},\\
			\supp(f)/G'_\RR \subset G_\RR/G'_\RR \text{ is compact}
		\end{array}}.
	\end{align*}
	Let $(\pi, \calH)$ be a unitary representation of $G'_\RR$ on a Hilbert space $\calH$ with inner product $\langle \cdot, \cdot \rangle$.
	We denote by $\uInd^{G_\RR}_{G'_\RR}(\calH)$ the Hilbert completion of the pre-Hilbert space $\ccInd^{G_\RR}_{G'_\RR}(\calH^\infty)$ with the inner product
	\begin{align*}
		\langle f, h \rangle = \int_{G_\RR/G'_\RR} \langle f(x), h(x)\rangle d\nu(x).
	\end{align*}
\end{definition}

The subspace $\ccInd^{G_\RR}_{G'_\RR}(V) \cap C^\infty_{KG'_\RR}(G_\RR, V)$ of $C^\infty(G_\RR, V)$ is closed for any compact subset $K\subset G_\RR/G'_\RR$.
We consider $\ccInd^{G_\RR}_{G'_\RR}(V)$ as the strict inductive limit of the subspaces $\ccInd^{G_\RR}_{G'_\RR}(V) \cap C^\infty_{KG'_\RR}(G_\RR, V)$.
The following proposition is clear from Facts \ref{fact:Nuclear} and \ref{fact:FunctionSpace}.

\begin{proposition}
	$\ccInd^{G_\RR}_{G'_\RR}(V)$
	is barreled.
	If $V$ is nuclear (resp.\ separable), then so is $\ccInd^{G_\RR}_{G'_\RR}(V)$.
\end{proposition}

We regard $\ccInd^{G_\RR}_{G'_\RR}(V)$ and $\uInd^{G_\RR}_{G'_\RR}(\calH)$ as representations of $G_\RR$ via the left translation.
It is well-known that $\ccInd^{G_\RR}_{G'_\RR}(V)$ is a smooth representation and $\uInd^{G_\RR}_{G'_\RR}(\calH)$ is a unitary representation.

\begin{fact}[See e.g.\ {\cite[Lemma 5.1.1.4]{Wa72}}]\label{fact:CompactInd}
	Let $(\pi, V)$ be a smooth Fr\'echet representation of $G'_\RR$.
	Then the map $I \colon C^\infty_c(G_\RR, V) \rightarrow \ccInd^{G_\RR}_{G'_\RR}(V)$ defined by $I(f)(x) = \int_{G'_\RR} \pi(g) f(xg) d\mu'(g)$ is continuous and surjective, and the map induces an isomorphism
	\begin{align*}
		C^\infty_c(G_\RR, V) / \Ker(I) \simeq \ccInd^{G_\RR}_{G'_\RR}(V)
	\end{align*}
	of smooth representations.
\end{fact}

We recall the action of $C^\infty_c(G_\RR)$ on a continuous representation $(\pi, V)$ of $G_\RR$ on a complete locally convex space $V$.
For $f \in C^\infty_c(G_\RR)$ and $v \in V$, we set
\begin{align*}
	\pi(f)v = \int_{G_\RR} f(g)\pi(g)v d\mu(g) \in V^\infty.
\end{align*}
Then we obtain a separately continuous bilinear map $C^\infty_c(G_\RR)\times V \rightarrow V^\infty$ and the map extends to a continuous linear map $C^\infty_c(G_\RR)\toptensor[i] V \rightarrow V^\infty$.

Let $(\pi, V)$ be a Hilbert representation of $G_\RR$.
Then $V^\infty$ is a Fr\'echet space with the semi-norms of the form $\|\pi(X)v\|$ ($X \in \univ{g}$).
In other words, $V^\infty$ is the projective limit of countably many Hilbert spaces and hence reflexive by \cite[IV.5.6 and IV.5.8]{ScWo99}.
We set $V^{-\infty} := \topdual{((\topdual{V})^\infty)}$, which is the space of distribution vectors of $V$.
Since $(V^*)^\infty$ is reflexive, $V^{-\infty}$ is also reflexive and hence barreled.
Then $\pi$ extends to a continuous representation on $V^{-\infty}$.

In the case, the image of $\pi(f)$ ($f \in C^\infty_c(G_\RR)$)
is contained in $V^\infty$, that is,
\begin{align*}
	\pi(C^\infty_c(G_\RR)) V^{-\infty} \subset V^\infty.
\end{align*}
In fact, $V^{-\infty} = \pi(\univ{g})V$ holds and hence $\pi(C^\infty_c(G_\RR)) V^{-\infty} =  \pi(C^\infty_c(G_\RR))V \subset V^\infty$.
In particular, $V^\infty$ is sequentially dense in $V^{-\infty}$.
See e.g.\ \cite[1.3]{Ca76} and \cite[Proposition 5.1]{He14}.
In the literature, this result is stated for unitary representations, but the unitarity is not essential.

By Lemma \ref{lem:HomBarrelled}, the bilinear map $C^\infty_c(G_\RR)\times V^{-\infty} \rightarrow V^\infty$ is separately continuous.
Hence we obtain a continuous linear map
\begin{align*}
	C^\infty_c(G_\RR)\toptensor[i] V^{-\infty} \rightarrow V^\infty.
\end{align*}

\begin{lemma}\label{lem:HomIsomSmooth}
	Let $(\pi, V)$ be a smooth Fr\'echet representation of $G_\RR$ and $(\tau, W)$ a Hilbert representation of $G_\RR$.
	Then the natural injection $\Hom_{G_\RR}(V, W^{\infty}) \rightarrow \Hom_{G_\RR}(V, W^{-\infty})$ is bijective.
\end{lemma}

\begin{proof}
	Let $T \in \Hom_{G_\RR}(V, W^{-\infty})$.
	By the Dixmier--Malliavin theorem \cite{DiMa78}, any vector of $V$ is a finite sum of vectors of the form $\pi(f)w$ ($f \in C^\infty_c(G_\RR)$ and $w \in V$).
	Hence the image of $T$ is contained in $\tau(C^\infty_c(G_\RR))W^{-\infty}\subset W^\infty$ and the lemma follows from Lemma \ref{lem:HomBarrelled}.
\end{proof}

The following proposition is a variation of the Frobenius reciprocity for the induction $\ccInd^{G_\RR}_{G'_\RR}(V)$.
For some literature (e.g.\ \cite{Wa72}), the Frobenius reciprocity for smooth representations is described by the terminology of bilinear forms.
We rewrite it using homomorphisms.

\begin{proposition}\label{prop:SmoothReciprocity}
	Let $(\pi, V)$ be a Hilbert representation of $G_\RR$ and $(\tau, W)$ a Hilbert representation of $G'_\RR$.
	Then there exists a natural isomorphism
	\begin{align*}
		\Hom_{G'_\RR}((\topdual{V})^{\infty}, (\topdual{W})^\infty) \simeq \Hom_{G_\RR}(\ccInd^{G_\RR}_{G'_\RR}(W^\infty), V^\infty)
	\end{align*}
	of vector spaces.
	In particular, if $V$ and $W$ are unitary, then we have
	\begin{align*}
		\Hom_{G'_\RR}(V^{\infty}, W^\infty) \simeq \Hom_{G_\RR}(\ccInd^{G_\RR}_{G'_\RR}(W^\infty), V^\infty).
	\end{align*}
\end{proposition}

\begin{proof}
	By Proposition \ref{prop:IsomorphismTensor} and Lemma \ref{lem:HomIsomSmooth}, we have natural isomorphisms
	\begin{align*}
		\Hom_{G'_\RR}((\topdual{V})^{\infty}, (\topdual{W})^\infty) &\simeq \Hom_{G'_\RR}((\topdual{V})^{\infty}, (\topdual{W})^{-\infty}) \\
		&\simeq \Hom_{G'_\RR}((\topdual{V})^{\infty} \toptensor[i] W^\infty, \CC)
		\simeq \Hom_{G'_\RR}(W^\infty, V^{-\infty}).
	\end{align*}
	For $T \in \Hom_{G'_\RR}(W^\infty, V^{-\infty})$, consider the composition map $\widetilde{\alpha}(T)$
	\begin{align*}
		C^\infty_c(G_\RR) \toptensor[i] W^\infty \xrightarrow{\id\toptensor T} C^\infty_c(G_\RR) \toptensor[i] V^{-\infty} \rightarrow V^\infty,
	\end{align*}
	where is written as
	\begin{align*}
		\widetilde{\alpha}(T)(f) = \int_{G_\RR/G'_\RR} \int_{G'_\RR} \pi(g)T(\tau(h)f(gh)) d\mu'(h)d\nu(gG'_\RR)
	\end{align*}
	for $f \in C^\infty_c(G_\RR, W^\infty)$ (up to scaling).
	Then we have $\widetilde{\alpha}(T)(\Ker(I)) = 0$, where $I$ is the map defined in Fact \ref{fact:CompactInd}.
	By Fact \ref{fact:CompactInd}, the map $\widetilde{\alpha}(T)$ factors through a continuous linear map $\alpha(T)\colon \ccInd^{G_\RR}_{G'_\RR}(W^\infty) \rightarrow V^\infty$.
	Hence we obtain a linear map
	\begin{align*}
		\alpha \colon \Hom_{G'_\RR}(W^\infty, V^{-\infty}) \rightarrow \Hom_{G_\RR}(\ccInd^{G_\RR}_{G'_\RR}(W^\infty), V^\infty).
	\end{align*}
	By the Frobenius reciprocity \cite[Theorem 5.3.3.1]{Wa72} for bilinear forms, the composition 
	\begin{align*}
		\Hom_{G'_\RR}((\topdual{V})^\infty\toptensor[i]W^\infty, \CC) &\simeq \Hom_{G'_\RR}(W^\infty, V^{-\infty}) \\
		&\xrightarrow{\alpha} \Hom_{G_\RR}(\ccInd^{G_\RR}_{G'_\RR}(W^\infty), V^\infty) \\
		&\hookrightarrow \Hom_{G_\RR}(\ccInd^{G_\RR}_{G'_\RR}(W^\infty), V^{-\infty})\\
		&\simeq \Hom_{G_\RR}((\topdual{V})^\infty \toptensor[i] \ccInd^{G_\RR}_{G'_\RR}(W^\infty), \CC)
	\end{align*}
	is bijective and hence so is $\alpha$.
\end{proof}

\subsection{Direct integral decomposition}\label{subsect:DirectIntegral}

In this and next subsections, we summarize several results about the direct integral decomposition of unitary representations.
We recall the notion of direct integrals of Hilbert spaces following \cite[Part II. Chapter 1]{Di81}.

Let $(X, \mu)$ be a measure space with $\sigma$-finite measure $\mu$.
Let $(\calH_x)_{x \in X}$ be a family of separable Hilbert spaces.
We say that $(\calH_x)_{x \in X}$ equipped with a subspace $S \subset \prod_{x \in X} \calH_x$ is a measurable field of Hilbert spaces if the following conditions are satisfied:
\begin{enumerate}
	\item For any $v \in S$, the function $X \ni x \mapsto \|v(x)\|_{x} \in \CC$ is measurable.
	\item Let $v \in \prod_{x \in X} \calH_x$.
	If the function $x \mapsto \langle v(x), u(x)\rangle_x$ is measurable for any $u \in S$, then $v$ belongs to $S$.
	\item There exists a sequence $v_0, v_1, \ldots \in S$ such that $\set{v_i(x): i \in \NN}$ is dense in $\calH_x$ for any $x \in X$.
\end{enumerate}
Here $\langle\cdot, \cdot \rangle_x$ and $\|\cdot\|_x$ are the inner product and the norm of $\calH_x$, respectively.
An element of $S$ is called a measurable section.
If $v \in S$ satisfies $\|v\| := \left(\int_X \|v(x)\|^2_x d\mu(x)\right)^{1/2} < \infty$, the section $v$ is said to be square integrable.
We set
\begin{align*}
	\int_X^\oplus \calH_x d\mu(x) := \set{v \in S: v\text{ is square integrable}}/\set{v \in S: \|v\| = 0},
\end{align*}
which is called the direct integral of the measurable family $(\calH_x)_{x \in X}$.
Then it is known that $\int_X^\oplus \calH_x d\mu(x)$ is a separable Hilbert space with the norm $\|\cdot \|$.
See \cite[Propositions II.1.5 and 6]{Di81}.
We set $\calH:= \int_X^\oplus \calH_x d\mu(x)$.

\begin{definition}\label{def:PointwiseDefined}
	Let $\varphi \colon \calW\rightarrow \calH = \int_X^\oplus \calH_x d\mu(x)$ be a continuous linear map from a topological vector space.
	We say that $\varphi$ is \define{pointwise defined} if there exists a family of continuous linear maps $\varphi_x \colon \calW\rightarrow \calH_x$ ($x \in X$) such that
	$\varphi(v) = (\varphi_x(v))_{x \in X}$ for any $v \in \calW$ as elements of $\calH$.
\end{definition}
Note that the choice of $(\varphi_x)_{x \in X}$ is not unique.
Hereafter whenever we say that $\varphi$ is pointwise defined, we fix a family of continuous linear maps $\varphi_x$ satisfying the condition in Definition \ref{def:PointwiseDefined},
and we say that $\varphi_x$ is the component of $\varphi$ at $x \in X$.
The following fact is useful to find pointwise defined linear maps.
See \cite[Theorem 1.5]{Be88_plancherel} and \cite[Proposition 12.2.1]{Sc89}.

\begin{fact}\label{fact:PointWiseDefined}
	Let $\varphi \colon \calW \rightarrow \calH = \int^\oplus_X \calH_x d\mu(x)$
	be a Hilbert--Schmidt operator from a separable Hilbert space to the direct integral.
	Then $\varphi$ is pointwise defined.
\end{fact}

Let $G_\RR$ be a Lie group (or more generally separable locally compact group).
Consider unitary representations $(\pi_x, \calH_x)$ of $G_\RR$ on the Hilbert spaces $\calH_x$ in the above notation.
We say that the family $((\pi_x, \calH_x))_{x \in X}$ is measurable if $S$ is $G_\RR$-stable in $\prod_{x \in X} \calH_x$.
Then we can provide $\calH$ with a $G_\RR$-action by $\pi(g)v := (\pi_x(g)v(x))_{x \in X}$ for $v \in \calH$.
It is easy to see that $(\pi, \calH)$ is a unitary representation of $G_\RR$ (see \cite[Lemma 14.9.2]{Wa92_real_reductive_II}), which is called the direct integral of the unitary representations $((\pi_x, \calH_x))_{x\in X}$.

\begin{fact}\label{fact:PointWiseGequiv}[See e.g.\ {\cite[Lemma 1.3]{Be88_plancherel}}]
	Let $\varphi \colon \calW \rightarrow \calH = \int^\oplus_X \calH_x d\mu(x)$
	be a continuous linear map from a separable locally convex space to the direct integral.
	Assume that $\varphi$ is pointwise defined.
	\begin{enumparen}
		\item If $\varphi(\calW)$ is dense in $\calH$, then $\varphi_x(\calW)$ is dense in $\calH_x$ $\mu$-a.e.
		\item Suppose that $\calW$ is a continuous representation of $G_\RR$ and $\varphi$ is $G_\RR$-linear.
		Then $\varphi_x$ is $G_\RR$-linear $\mu$-a.e.
	\end{enumparen}
\end{fact}

The notion of nuclear spaces is useful to construct a Hilbert--Schmidt operator as in Fact \ref{fact:PointWiseDefined}.
Let $\calW$ be a separable nuclear space.
As we have seen in Fact \ref{fact:NuclearCharacterization}, any continuous linear map $\varphi \colon \calW\rightarrow \calH$ to a Hilbert space factors through a Hilbert--Schmidt operator $\overline{\varphi} \colon \calW_p \rightarrow \calH$, where $\calW_p$ is the completion of $\calW$ with respect to some continuous Hilbert semi-norm on $\calW$.
Clearly, $\calW_p$ is also separable.
The following formulation is useful to the representation theory, which has appeared in \cite[3.3]{Li18} and \cite[Theorem 5.2]{Ki19}.

\begin{proposition}\label{prop:DisintegrationNuclear}
	Let $\calW$ be a continuous representation of $G_\RR$ on a separable nuclear space and $\calH = \int^\oplus_X \calH_x d\mu(x)$ a direct integral of unitary representations of $G_\RR$.
	Then any continuous $G_\RR$-linear map $\varphi\colon \calW\rightarrow \calH$ is pointwise defined and each component $\varphi_x\colon \calW\rightarrow \calH_x$ is $G_\RR$-linear $\mu$-a.e.
	If, in addition, $\varphi(\calW)$ is dense in $\calH$, then $\varphi_x(\calW)$ is dense in $\calH_x$ $\mu$-a.e.
\end{proposition}

\begin{proof}
	The proposition follows from Facts \ref{fact:PointWiseDefined} and \ref{fact:PointWiseGequiv}.
\end{proof}

In the representation theory of Lie groups, many smooth representations are nuclear and separable.
It is well-known that if $G_\RR$ is reductive or nilpotent, for any irreducible unitary representation $V$, the smooth representation $V^\infty$ is nuclear and separable.
Hence its induced representation (see Definition \ref{def:CompactInd}) is also separable and nuclear by Fact \ref{fact:CompactInd}.
A vector space of countable dimension equipped with the strongest locally convex linear topology is nuclear and separable.
Note that $V^\infty$ of a unitary representation $V$ of $G_\RR$ is nuclear only if $V$ is trace class.

\subsection{Disintegration of invariant operators}

We shall study disintegration of actions of unbounded operators.
We refer the reader to \cite[Chapter 12]{Sc89}.
In this subsection, we assume that any closable operator is densely defined.

\begin{definition}
	Let $\calH$ be a Hilbert space and $S$ a subset of $\End_{\CC}(\calH)$.
	We denote by $S'$ the commutator of $S$ in $\End_{\CC}(\calH)$.
	We say that a closable operator $T$ on $\calH$ with domain $D$ commutes with $S$ if $XD \subset D$ and $TXv = XTv$ for any $v \in D$ and $T \in S$.
\end{definition}

\begin{definition}
	A bounded operator $T$ on a direct integral $\calH = \int^\oplus_X \calH_x d\mu(x)$ is said to be \define{diagonalizable} if $T$ can be represented as a multiplication operator by some $f \in L^\infty(X, \mu)$.
\end{definition}

Let $\calH$ be a Hilbert space and $S$ a $*$-stable subset of $\End_{\CC}(\calH)$.
Then $S'$ is a von Neumann algebra.
By the von Neumann bicommutant theorem \cite[Corollary I.3.1]{Di81}, $S''$ is the closure of the algebra generated by $S \cup \set{\id}$ with respect to the strong operator topology.
Hence we have the following lemma.

\begin{lemma}\label{lem:CommutativeClosedOperator}
	Retain the above setting.
	Let $X$ be a closable operator on $\calH$ with domain $D$.
	If $X$ commutes with $S$, then the closure $\overline{X}$ commutes with $S''$.
\end{lemma}

\begin{proof}
	Let $\mathcal{G} \subset \calH\oplus \calH$ be the graph of $X$.
	By the definition of the commutativity, $\mathcal{G}$ is $\Delta(S)$-stable, where $\Delta(S) = \set{T\oplus T \in \End_{\CC}(\calH\oplus \calH) : T \in S}$.
	This implies that $\overline{\mathcal{G}}$ is $\Delta(S)$-stable and hence $\Delta(S'')$-stable.
	Therefore $\overline{X}$ commutes with $S''$.
\end{proof}

Let $\calA$ be a $*$-algebra, i.e.\ $\calA$ is a $\CC$-algebra equipped with a conjugate linear involutive antiautomorphism $(\cdot)^*$.
Let $V$ be a dense subspace of a Hilbert space $\calH$.
We say that an $\calA$-module $(\pi, V)$ is a $*$-representation if $\langle \pi(a)v, w\rangle = \langle v, \pi(a^*)w\rangle$ holds for any $a \in \calA$, $v, w \in V$.
Then it is easy to see that $\pi(a)$ is a closable operator for any $a \in \calA$.

Let $\calA$ be a $*$-algebra and $(\pi, V)$ a $*$-representation on a dense subspace $V$ of a direct integral $\calH = \int^\oplus_X \calH_x d\mu(x)$.
Assume that $V$ admits a separable nuclear locally convex linear topology such that the inclusion $\varphi \colon V\hookrightarrow \calH$ is continuous.
Then $\varphi$ is pointwise defined by Fact \ref{fact:PointWiseDefined}.

\begin{fact}[{\cite[Proposition 12.2.3]{Sc89}}]\label{fact:PointWiseStarRep}
	Let $\calN$ be the von Neumann algebra of all bounded diagonalizable operators on $\calH= \int^\oplus_X \calH_x d\mu(x)$ and set $V_x := \varphi_x(V)$.
	Assume that $\overline{\pi(X)}$ commutes with $\calN$ for any $X \in \calA$.
	Then there exist a $\mu$-null set $N$ and a family $\set{(\pi_x, V_x)}_{x \in X-N}$ of $*$-representations of $\calA$ such that $\varphi_x \colon V\rightarrow V_x$ is an $\calA$-homomorphism for any $x \in X-N$.
\end{fact}

\begin{remark}
	In the reference \cite{Sc89}, $X$ is assumed to be locally compact $\sigma$-compact metric space and $\mu$ to be regular Borel.
	We can replace $(X, \mu)$ in our direct integral to satisfy this assumption.
	In fact, the von Neumann algebra $\calN$ contains a norm closed separable unital $C^*$-algebra $\calN_0$ which is dense in $\calN$ with respect to the strong operator topology,
	and we can replace $X$ with the compact metric space $\spec(\calN_0)$
	(see \cite[Corollary II.6 and Theorem II.6.4]{Di81}).
\end{remark}

Let $G_\RR$ be a reductive Lie group (or more generally separable locally compact group of type I)
and $(\pi, \calH)$ a unitary representation of $G_\RR$ on a separable Hilbert space.
Then $(\pi, \calH)$ has the irreducible decomposition of the form
\begin{align*}
	\calH \simeq \int^\oplus_{\widehat{G}_\RR} \calH_\tau \toptensor \calM_\tau d\mu(\tau),
\end{align*}
where $\widehat{G}_\RR$ is the unitary dual of $G_\RR$, $\mu$ is a Borel measure on $\widehat{G}_\RR$ and $\calH_\tau$ is a representation space of $\tau \in \widehat{G}_\RR$.
The space $\calM_\tau$ is a Hilbert space on which $G_\RR$ acts trivially and $\toptensor$ means the tensor product of Hilbert spaces.
It is known that the function $\widehat{G}_\RR \ni \tau \mapsto \dim_{\CC}(\calM_\tau)$ is $\mu$-measurable.
See \cite[Theorem 8.6.6]{Di77}.

\begin{definition}\label{def:SupMultiplitiy}
	We set $\calM_{G_\RR}(\calH) := \esssup_{\tau \in \widehat{G}_\RR}\set{\dim_{\CC}(\calM_\tau)}$.
\end{definition}

The von Neumann algebra of all bounded diagonalizable operators of the direct integral coincides with $\pi(G_\RR)' \cap \pi(G_\RR)''$, which is the center of $\pi(G_\RR)''$.
See the construction of the direct integral decomposition into factorial representations (\cite[Theorem 8.4.2]{Di77} or \cite[Theorem 14.10.2]{Wa92_real_reductive_II}).

\begin{lemma}\label{lem:FactorStarRep}
	Fix $\tau \in \widehat{G}_\RR$.
	Let $\calA$ be a $*$-algebra and $(\alpha, D)$ a $*$-representation of $\calA$ on a dense subspace $D\subset \calH_\tau \toptensor \calM_\tau$.
	Assume that any element of $\alpha(\calA)$ commutes with $\tau(G_\RR)\toptensor \id$.
	Then there exists a $*$-representation $(\alpha', D')$ of $\calA$ on a dense subspace $D' \subset \calM_\tau$ such that $\overline{\alpha(X)} = \overline{\id\otimes \alpha'(X)}$ for any $X \in \calA$.
\end{lemma}

\begin{proof}
	For a Hilbert space $\calW$ and a subspace $V\subset \calH_\tau \toptensor \calW$, we set
	\begin{align*}
		\Phi(V) := \spn{(\varphi\toptensor \id)(v)\in \calW: v\in V, \varphi \in \topdual{\calH_\tau}}.
	\end{align*}
	Note that for $u = \sum_{i=0}^\infty v_i \otimes w_i \in \calH_\tau\toptensor \calW$, if $\set{v_i}$ is an orthonormal basis of $\calH_\tau$ and $u$ belongs to $V$, then we have $w_i \in \Phi(V)$ for any $i$.

	Fix $X \in \calA$ and let $\mathcal{G}$ be the graph of $\alpha(X)$.
	We shall construct a closable operator $\alpha'(X)$ on $\calM_\tau$ with the domain $\Phi(D)$ satisfying
	\begin{align}
		\alpha(X)\left(\sum_{i=0}^\infty v_i \otimes w_i \right)
		= \sum_{i=0}^\infty v_i \otimes \alpha'(X) w_i \label{eqn:GequivClosable}
	\end{align}
	for any orthonormal basis $\set{v_i}$ of $\calH_\tau$ and $\sum_{i=0}^\infty v_i \otimes w_i \in D$.

	$\overline{\mathcal{G}}$ is a $G_\RR$-stable closed subspace of $\calH_\tau \toptensor \calM_\tau \oplus \calH_\tau \toptensor \calM_\tau$.
	We regard $\mathcal{G}$ and $\overline{\mathcal{G}}$ as subspaces of $\calH_\tau \toptensor (\calM_\tau \oplus \calM_\tau) \simeq \calH_\tau \toptensor \calM_\tau \oplus \calH_\tau \toptensor \calM_\tau$.
	Since $\calH_\tau$ is an irreducible $G_\RR$-representation, $\Phi(\overline{\mathcal{G}})$ is closed and we have
	$\overline{\mathcal{G}} = \calH_\tau \toptensor \Phi(\overline{\mathcal{G}})$.
	Since $\overline{\mathcal{G}}$ is a graph, so is $\Phi(\overline{\mathcal{G}})$ in $\calM_\tau \oplus \calM_\tau$.

	We set $D' := \Phi(D)$ and $\mathcal{G'} := \Phi(\mathcal{G})$.
	Then $\mathcal{G'}$ is a dense subspace of the graph $\Phi(\overline{\mathcal{G}})$ and hence $\mathcal{G'}$ is a graph.
	Let $\alpha'(X)$ denote the corresponding closable operator to the graph $\mathcal{G}'$.
	Then the domain of $\alpha'(X)$ is equal to $D'$.
	Since $\calH_\tau \otimes \mathcal{G}'$ is dense in $\overline{\mathcal{G}} = \calH_\tau \toptensor \Phi(\overline{\mathcal{G}})$, we have $\overline{\id \otimes \alpha'(X)} = \overline{\alpha(X)}$.

	By construction, $\alpha'(X)$ satisfies \eqref{eqn:GequivClosable} and $\Im(\alpha'(X)) = \Phi(\Im(\alpha(X))) \subset D'$.
	Hence $(\alpha', D')$ is an $\calA$-module.
	Since $(\alpha, D)$ is a $*$-representation, we have $\overline{\alpha(X)} \subset \alpha(X^*)^* = \overline{\alpha(X^*)}^*$.
	For any $w, w' \in D'$ and unit vector $v \in \calH_\tau$, we have
	\begin{align*}
		\langle \alpha'(X)w, w'\rangle &= \langle v\otimes \alpha'(X)w, v\otimes w\rangle \\
		&= \langle \overline{\alpha(X)} (v\otimes w), v\otimes w\rangle \\
		&= \langle v\otimes w, \overline{\alpha(X^*)} (v\otimes w)\rangle \\
		&= \langle w, \alpha'(X^*)w'\rangle.
	\end{align*}
	This shows that $(\alpha', D')$ is a $*$-representation of $\calA$.
\end{proof}

Summarizing Proposition \ref{prop:DisintegrationNuclear}, Fact \ref{fact:PointWiseStarRep} and Lemma \ref{lem:FactorStarRep}, we obtain the following result.
Retain the above notations $G_\RR$ and $\calH$.
Let $\calW$ be a dense $G_\RR$-stable subspace of $\calH$ and $(\alpha, \calW)$ a $*$-representation of a $*$-algebra $\calA$.
Assume that
\begin{enumerate}
	\item $\calW$ admits a separable nuclear locally convex linear topology such that the $G_\RR$-representation on $\calW$ is continuous and the inclusion $\iota\colon \calW\hookrightarrow \calH$ is continuous
	\item $\alpha(X)$ commutes with $\pi(G_\RR)$ for any $X \in \calA$.
\end{enumerate}
Then $\iota$ is pointwise defined by Proposition \ref{prop:DisintegrationNuclear}, and $\overline{\alpha(X)}$ commutes with $\pi(G_\RR)''$ for any $X \in \calA$ by Lemma \ref{lem:CommutativeClosedOperator}.

\begin{theorem}\label{thm:PointWiseGAmodule}
	There exist a $\mu$-null set $N$ of $\widehat{G}_\RR$ and $\set{(\alpha_\tau, V_\tau)}_{\tau \in \widehat{G}_\RR -N}$, a family of $*$-representations of $\calA$ on dense subspaces $V_\tau \subset \calM_\tau$ satisfying the following conditions for any $\tau \in \widehat{G}_\RR -N$:
	\begin{enumparen}
		\item\label{enum:PointWiseGAmoduleConti} the component $\iota_\tau\colon \calW\rightarrow \calH_\tau \toptensor \calM_\tau$ is continuous and $G_\RR$-linear
		\item\label{enum:PointWiseGAmoduleDense} the image of $\iota_\tau$ is dense in $\calH_\tau \toptensor \calM_\tau$ and contained in the domain of $\overline{\id \otimes \alpha_\tau(X)}$ for any $X \in \calA$
		\item\label{enum:PointWiseGAmoduleEquiv} $\iota_\tau \circ \alpha(X) = \overline{\id \otimes \alpha_\tau(X)} \circ \iota_\tau$ holds for any $X \in \calA$.
	\end{enumparen}
	Moreover, we have $\bigcap_{\tau \in \widehat{G}_\RR-N} \Ker(\iota_\tau) = 0$.
\end{theorem}

We shall state corollaries of Theorem \ref{thm:PointWiseGAmodule}.
Retain the notation in the theorem.

\begin{corollary}\label{cor:DirectIntegralAnnPIdeg}
	We have
	\begin{align}
		\Ann_{\calA}(\calW) &= \bigcap_{\tau \in \widehat{G}_\RR - N} \Ann_{\calA}(V_\tau), \label{eqn:DirectIntegralAnn} \\
		\PIdeg(\calA/\Ann_{\calA}(\calW)) &\leq \supmul_{G_\RR}(\calH). \label{eqn:DirectIntegralPIdeg}
	\end{align}
\end{corollary}

\begin{proof}
	\eqref{eqn:DirectIntegralAnn} follows easily from $\bigcap_{\tau \in \widehat{G}_\RR-N} \Ker(\iota_\tau) = 0$.
	We can replace $N$ by a larger $\mu$-null set to satisfy
	\begin{align*}
		\supmul_{G_\RR}(\calH) &= \sup\set{\dim_{\CC}(\calM_\tau): \tau \in \widehat{G}_\RR-N} \\
		&= \sup\set{\dim_{\CC}(V_\tau): \tau \in \widehat{G}_\RR-N}.
	\end{align*}
	Then \eqref{eqn:DirectIntegralPIdeg} follows from \eqref{eqn:DirectIntegralAnn} and Proposition \ref{prop:pidegUpperbound}.
\end{proof}

\begin{corollary}\label{cor:MultiplictySpaceAndHom}
	For any $\tau \in \widehat{G}_\RR - N$, there exists an injection $\varphi_\tau \colon \topdual{\calM_\tau} \hookrightarrow \Hom_{G_\RR}(\calW, \calH_\tau)$.
	In particular, we have
	\begin{align*}
		\dim_{\CC}(\calM_\tau) \leq \dim_{\CC}(\Hom_{G_\RR}(\calW, \calH_\tau)).
	\end{align*}
\end{corollary}

\begin{proof}
	We define $\varphi_\tau$ by $\varphi_\tau(f) = (\id\toptensor f)\circ \iota_\tau$ for $f \in \topdual{\calM_\tau}$.
	Since $\iota_\tau(\calW)$ is dense in $\calH_\tau\toptensor \calM_\tau$,
	the linear map $\varphi_\tau$ is injective.
\end{proof}

\begin{remark}
	The corollary is known in \cite[Proposition 3.3.3]{Li18}.
\end{remark}

If $\alpha(X)$ is continuous on $\calW$ for any $X \in \calA$, then $\Hom_{G_\RR}(\calW, \calH_\tau)$ is a right $\calA$-module.
In the case, we can extend $\varphi_\tau$ to an $\calA$-module homomorphism.

\begin{corollary}
	Fix $\tau \in \widetilde{G}_\RR - N$.
	We provide $V_\tau$ with the locally convex linear topology given by the graph norms of the operators $\overline{\alpha_\tau(X)}$ ($X \in \calA$).
	Assume that $\alpha(X)$ is continuous on $\calW$ for any $X \in \calA$.
	Then $\varphi_\tau$ in Corollary \ref{cor:MultiplictySpaceAndHom} extends to an injective $\calA$-homomorphism
	\begin{align*}
		\topdual{V_\tau} \hookrightarrow \Hom_{G_\RR}(\calW, \calH_\tau).
	\end{align*}
\end{corollary}

\begin{proof}
	We provide $\calH_\tau \otimes V_\tau$ with the locally convex linear topology given by the graph norms of the operators $\overline{\id\otimes \alpha_\tau(X)}$ ($X \in \calA$) and write $D_\tau$ for the completion of $\calH_\tau \otimes V_\tau$.
	Then for any $f \in \topdual{V_\tau}$, the $G_\RR$-linear map $\id\otimes f\colon \calH_\tau \otimes V_\tau \rightarrow \calH_\tau$ extends to a continuous $G_\RR$-linear map $f'\colon D_\tau\rightarrow \calH_\tau$.
	
	By Theorem \ref{thm:PointWiseGAmodule} \ref{enum:PointWiseGAmoduleConti} and \ref{enum:PointWiseGAmoduleDense}, the image of $\iota_\tau$ is contained in $D_\tau$ and $\iota_\tau\colon \calW\rightarrow D_\tau$ is continuous.
	By \cite[Proposition 2.2.12]{Sc89}, $\iota_\tau(\calW)$ is dense in $D_\tau$.
	Therefore we have $f' \circ \iota_\tau \in \Hom_{G_\RR}(\calW, \calH_\tau)$.
	Hence we obtain an injective $\calA$-homomorphism $\topdual{V_\tau}\rightarrow \Hom_{G_\RR}(\calW, \calH_\tau)$ whose restriction to $\topdual{\calM_\tau}$ coincides with $\varphi_\tau$.
\end{proof}

\subsection{Unitary representation and polynomial identity}

In Theorems \ref{thm:BoundedRestriction} and \ref{thm:BoundedInduction}, we have given an upper bound and a lower bound for multiplicities of restrictions and inductions of smooth representations.
We shall show analogues of these results for unitary representations.

Let $G_\RR$ be a reductive Lie group and $G'_\RR$ a reductive subgroup of $G_\RR$.
Assume that $\lie{g'}$ is algebraic in $\lie{g}$.
See Subsection \ref{subsect:CategoriesGKCW} and Definitions \ref{def:supmul}
and \ref{def:SupMultiplitiy} for the notations $\calC_{G'_\RR}$, $\SupDim(\cdot)$ and $\supmul_{G'_\RR}(\calH)$.

\begin{theorem}\label{thm:BoundedRestrictionUnitary}
	Let $(\pi, \calH)$ be an irreducible unitary representation of $G_\RR$
	and set $I:=\Ann_{\univ{g}}(\calH^\infty)$.
	Then there exists some constant $C > 0$ independent of $\calH$ such that
	\begin{align*}
		\PIdeg((\univ{g}/I)^{G'_\RR}) &\leq \supmul_{G'_\RR}(\calH) \\
		&\leq \SupDim(\Hom_{G'_\RR}(\calH^\infty, \calC_{G'_\RR})) \\
		&\leq C\cdot \PIdeg((\univ{g}/I)^{G'_\RR}).
	\end{align*}
	Moreover, $\calH|_{G'_\RR}$ is multiplicity-free only if $(\univ{g}/I)^{G'_\RR}$ is commutative.
\end{theorem}

\begin{proof}
	The Casselman--Wallach representation $\calH^\infty$ is separable and nuclear.
	Since any $\pi(X)$ ($X \in \univ{g}^{G'_\RR}$) with the domain $\calH^\infty$ commutes with $\pi(G'_\RR)$, we can apply Theorem \ref{thm:PointWiseGAmodule} to
	\begin{align*}
		(\calH, \calW, \calA) = (\calH|_{G'_\RR}, \calH^\infty, \univ{g}^{G'_\RR}).
	\end{align*}
	The first two inequalities in the assertion are special cases of Corollaries \ref{cor:DirectIntegralAnnPIdeg} and \ref{cor:MultiplictySpaceAndHom}.
	The last inequality has been proved in Theorem \ref{thm:BoundedRestriction}.

	Suppose that $\calH|_{G'_\RR}$ is multiplicity-free.
	By Corollary \ref{cor:DirectIntegralAnnPIdeg}, $(\univ{g}/I)^{G'}$ can be embedded in a commutative algebra.
	Hence $(\univ{g}/I)^{G'}$ is commutative.
\end{proof}

If a $*$-representation is finite dimensional, then the representation is completely reducible.
Hence we can show that if the multiplicities in $\calH|_{G'_\RR}$ are finite $\mu$-a.e., then $(\univ{g}/I)^{G'}$ is semiprimitive.
The last assertion in Theorem \ref{thm:BoundedRestrictionUnitary} can be deduced from this result.
See also \ref{prop:PIdegSemiprimitive}.

\begin{theorem}\label{thm:BoundedInductionUnitary}
	Let $\calH'$ be an irreducible unitary representation of $G'_\RR$
	and set $I:=\Ann_{\univ{g'}}((\calH')^\infty)$.
	Then there exists some constant $C > 0$ independent of $\calH'$ such that
	\begin{align*}
		\PIdeg((\univ{g}/I\univ{g})^{G'_\RR}) &\leq \supmul_{G_\RR}(\uInd^{G_\RR}_{G'_\RR}(\calH')) \\
		&\leq \SupDim(\Hom_{G'_\RR}(\calC_{G_\RR}, (\calH')^\infty)) \\
		&\leq C\cdot \PIdeg((\univ{g}/I\univ{g})^{G'_\RR}).
	\end{align*}
	Moreover, $\uInd^{G_\RR}_{G'_\RR}(\calH')$ is multiplicity-free only if $(\univ{g}/I\univ{g})^{G'_\RR}$ is commutative.
\end{theorem}

\begin{proof}
	By Fact \ref{fact:CompactInd}, the induced representation $\ccInd^{G_\RR}_{G'_\RR}((\calH')^\infty)$ is separable and nuclear.
	Hence we can apply Theorem \ref{thm:PointWiseGAmodule} to
	\begin{align*}
		(\calH, \calW, \calA) = (\uInd^{G_\RR}_{G'_\RR}(\calH'), \ccInd^{G_\RR}_{G'_\RR}((\calH')^\infty), \univ{g}^{G'_\RR}).
	\end{align*}
	By Lemma \ref{lem:AnnSubInd}, we have
	\begin{align*}
		\Ann_{\univ{g}^{G'_\RR}}(\ccInd^{G_\RR}_{G'_\RR}((\calH')^\infty)) = (I\univ{g})^{G'_\RR}
	\end{align*}
	as a right module, and by Proposition \ref{prop:SmoothReciprocity},
	\begin{align*}
		\Hom_{G'_\RR}(\calH^\infty, (\calH')^\infty) \simeq \Hom_{G_\RR}(\ccInd^{G_\RR}_{G'_\RR}((\calH')^\infty), \calH^\infty)
	\end{align*}
	for any irreducible unitary representation $\calH$ of $G_\RR$.
	Therefore the theorem follows from Theorem \ref{thm:BoundedInduction} and Corollaries \ref{cor:DirectIntegralAnnPIdeg} and \ref{cor:MultiplictySpaceAndHom}.
\end{proof}


\def\cprime{$'$} \def\cprime{$'$}

\end{document}